\documentclass[11pt,a4paper]{amsart}
\usepackage{graphicx}
\usepackage{amssymb}

\usepackage[square, numbers, comma]{natbib}

\usepackage{amsmath}
\usepackage{amsthm}
\usepackage{dsfont}
\usepackage{a4wide}
\usepackage{pdfsync}
\usepackage{hyperref}
\usepackage{enumerate}
\usepackage{color}
\usepackage{cancel}

\DeclareMathOperator*{\esssup}{ess\,sup}
\DeclareMathOperator*{\essinf}{ess\,inf}
\DeclareMathOperator*{\esslim}{ess\,lim}

\newtheorem{theorem}{Theorem}[section]
\newtheorem{proposition}[theorem]{Proposition}

\newtheorem{lemma}[theorem]{Lemma}

\theoremstyle{remark}
\newtheorem{remark}[theorem]{Remark}

\theoremstyle{definition}
\newtheorem{definition}[theorem]{Definition}

\newcommand{\sgn}{\,{\rm sgn}}

\def\eps{\varepsilon}
\def\teps{\tilde \varepsilon}

\def\rr{\mathbb{R}}

\def\eps{\varepsilon}

\def\ml{\mathcal{L}}
\def\bs{\mathbb{S}}

\def\ou{\overline u}

\title[Large-time behaviour for nonlocal diffusion-convection problems]{Large-time behaviour for anisotropic stable nonlocal diffusion problems with convection}
\author[J.~Endal]{J{\o}rgen Endal}
\address[J.~Endal]{Universidad Aut\'onoma de Madrid (UAM)\\
Campus de Cantoblanco, 28049 Madrid, Spain\\
and\\
Department of Mathematical Sciences\\
Norwegian University of Science and Technology (NTNU)\\
N-7491, Trondheim, Norway}
\email[]{jorgen.endal\@@{}ntnu.no}
\urladdr{https://folk.ntnu.no/jorgeen/}

\author[L.~I.~Ignat]{Liviu I. Ignat}
\address[L.~I.~Ignat]{Institute of Mathematics ``Simion Stoilow'' of the Romanian Academy\\
21 Calea Grivitei Street, 010702 Bucharest, Romania\\
and\\
The Research Institute of the University of Bucharest - ICUB\\
University of Bucharest\\
90-92 Sos. Panduri, 5th District, Bucharest, Romania
}
\email[]{liviu.ignat\@@{}gmail.com}
\urladdr{http://www.imar.ro/~lignat}

\author[F.~Quir\'{o}s]{Fernando Quir\'{o}s}
\address[F.~Quir\'{o}s]{Departamento de Matem\'aticas\\
Universidad Aut\'onoma de Madrid (UAM)\\
Campus de Cantoblanco, 28049 Madrid, Spain\\
and\\
Instituto de Ciencias Matem\'aticas ICMAT (CSIC-UAM-UCM-UC3M)\\
28049 Madrid, Spain}
\email[]{fernando.quiros\@@{}uam.es}
\urladdr{https://matematicas.uam.es/~fernando.quiros}

\begin{document}

\begin{abstract}
We study the large-time behaviour of nonnegative solutions to the Cauchy problem for a nonlocal heat equation with a nonlinear convection term. The diffusion operator  is the infinitesimal generator of a stable L\'evy process, which may be highly anisotropic.  The initial data are assumed to be bounded and integrable. The mass of the solution is conserved along the evolution, and the large-time behaviour is given by the source-type solution with this mass of a limit equation that depends  on the relative strength of convection and diffusion. When diffusion is stronger than convection the original equation simplifies asymptotically to the purely diffusive nonlocal heat equation. When convection dominates, it does so only in the direction of convection, and the limit equation is still diffusive in the subspace orthogonal to this direction, with a diffusion operator that is a ``projection'' of the original one onto the subspace. The determination of this projection is one of the main issues of the paper. When convection and diffusion are of the same order the limit equation coincides with the original one.

Most of our results are  new even in the isotropic case in which the diffusion operator is the fractional Laplacian. We are able to cover both the cases of slow and fast convection, as long as the mass is preserved.  Fast convection, which corresponds to convection
nonlinearities that are not locally Lipschitz, but only locally H\"older, has not been considered before in the nonlocal diffusion setting.
\end{abstract}

 %Compactness in Banach (or normed) spaces

\maketitle

\tableofcontents

%%%%%%%%%%%%%%%%%%%%%%%%%%%%%%%%%%%%%%%%%%%%%%%%%%%%%%%%%%
%%%%%%%%%%%%%%%%%%%%%%%%%%%%%%%%%%%%%%%%%%%%%%%%%%%%%%%%%%
\section{Introduction and main results}
\label{sect-Introduction} \setcounter{equation}{0}

We study the large-time behaviour of solutions to the nonlocal diffusion problem with nonlinear convection
\begin{equation}
\label{eq.main0}
	\partial_tu+\ml u+\nabla\cdot(F(u)) =0 \quad\textup{in }Q:=(0,\infty)\times \rr^N,\qquad
  	u(0)=u_0 \quad\textup{on }\rr^N,
\end{equation}
where $\ml$ is a nonnegative symmetric $\alpha$-stable operator, $\alpha\in (0,2)$,
\begin{equation}
\label{eq:F}
F(u):=\mathbf{a}|u|^{q-1}u,\quad \mathbf{a}\in\rr^N,\ |\mathbf{a}|=1, \quad 1-\frac{1}{N}<q\neq1,
\end{equation}
and $u_0\in L^1(\mathbb{R}^N)\cap L^\infty(\mathbb{R}^N)$. Our goal is to describe the \emph{intermediate asymptotics}, determining the rate at which solutions approach zero, and the limit profile, after scaling the solution to take into account the decay rate. The intermediate asymptotic behaviour depends strongly on the strength of diffusion, measured in terms of the parameter $\alpha$, when compared with convection, measured in terms of the parameter $q$.

%%%%%%%%%%%%%%%%%%%%%%%%%
\subsection{On the operator $\ml$}
The nonlocal operator $\mathcal{L}$ is defined, for smooth functions which do not grow too much at infinity, by
\begin{equation*}
\label{eq:general.operator}
(\mathcal{L}\phi)(x) =
\int_{\mathbb{R}^N}\Big(\phi(x)-\frac{\phi(x +y) + \phi(x-y)}2\Big)
\,{\rm d}\nu(y),
\end{equation*}
under the assumption:
\begin{equation}
\label{mlas}\tag{$\textup{A}_{\nu}$}
\left\{\
\begin{aligned}
&\nu\geq0\text{ is a Radon measure on }\rr^N\setminus\{0\}.\\
&\text{For some measure }\mu\text{ on }\bs^{N-1}:=\{x\in\mathbb{R}^N: |x|=1\},\text{ called the \emph{spectral measure}:}\\
&\hskip.8cm\textup{(Polar decomposition)}\qquad  {\rm d}\nu(r,\theta)=\frac{{\rm d}r}{r^{1+\alpha}}{\rm d}\mu(\theta) \text{ with $0<\alpha<2$;}\\
&\hskip.8cm\textup{(Nondegenerate)}\qquad\hskip1cm \inf_{\xi \in \bs^{N-1}}\int_{\bs^{N-1}} |\xi\cdot \theta|^\alpha\,{\rm d}\mu(\theta)\geq \Lambda_1>0;\\
&\hskip.8cm\textup{(Finite)}\qquad\hskip2.57cm \mu(\bs^{N-1})=\Lambda_2<\infty.
\end{aligned}
\right.
\end{equation}
Here, by a slight abuse of notation,  $\nu(r,\theta)$ denotes the measure on $(0,\infty)\times\bs^{N-1}$ induced by $\nu(y)$ and the change to polar coordinates, $r=|y|$, $\theta=y/|y|$.
We then immediately have
$$
\int_{\mathbb{R}^N}\min\{|y|^2,1\}\,{\rm d}\nu(y)=\int_{\mathbb{S}^{N-1}}\int_0^\infty\min\{r^2,1\}\,\frac{{\rm d}r}{r^{1+\alpha}}\,{\rm d}\mu(\theta)<\infty,
$$
and hence $\nu$ is a \emph{L\'evy measure}.

Using the polar decomposition of the measure $\nu$, the operator can be written as
\begin{equation}
\label{eq:operador en polares}
(\mathcal{L}\phi)(x) =\int_{\mathbb{S}^{N-1}}\int_0^\infty
\Big(\phi( x)-\frac{\phi(x + r\theta) + \phi(x-r\theta)}2 \Big)
\frac{{\rm d}r}{r^{1+\alpha}}\,{\rm d}\mu(\theta).
\end{equation}
Notice that, although  we do not impose any symmetry on the measure $\mu$, the operator will be symmetric, thanks to the way we write it  using second differences.

Operators of this form arise as infinitesimal generators of symmetric stable L\'evy processes $X=\{X_t\}_{t\ge0}$, which satisfy
$$
\lambda X_t= X_{\lambda^\alpha t},\qquad \lambda>0,\; t\ge0;
$$
see e.g.~\cite[Chapter 1]{Bertoin-1996}. 
These processes appear  in Physics, Mathematical Finance and Biology, among other applications, and have been the subject of intensive research in the last years from the point of view both of Probability and Analysis; see for instance~\cite{Janicki-Weron-1994,Samorodnitsky-Taqqu-1994,Mantegna-Stanley-1995}.

Taking Fourier transform in the definition of the operator $\ml$ we get
\begin{equation}\label{eq.SymbolL}
\widehat {{\ml}\phi}(\xi)=|\xi|^\alpha g\Big(\frac{\xi}{|\xi |}\Big)\widehat{\phi}(\xi),\qquad g(\xi)=C_{\alpha}  \int_{\bs^{N-1}}
	   |\xi\cdot \theta|^\alpha\,{\rm d}\mu(\theta),\qquad\displaystyle C_{\alpha}=\int_0^\infty
(1-\cos t)
\frac{dt}{t^{1+\alpha}};
\end{equation}
see for instance~\cite{dePablo-Quiros-Rodriguez-2020}.
Thanks to assumption~\eqref{mlas} we have $\Lambda_1\leq g(\xi/|\xi|)\leq\Lambda_2$. Hence,  the multiplier of the operator $\ml$,
\begin{equation}
\label{eq:multiplier.L}
m(\xi)= |\xi|^\alpha g\Big(\frac{\xi}{|\xi |}\Big),
\end{equation}
satisfies $m(\xi)\eqsim |\xi|^\alpha$ (see paragraph~\ref{sec:notation} for an explanation of this and other notations in the paper). This implies that $\ml$ is nondegenerate, and hence that the diffusion operator is parabolic in all directions.

In the isotropic case, ${\rm d}\mu(\theta)=c\,{\rm d}\theta$ for some constant $c>0$, the operator reduces to a multiple of the well-known  fractional Laplacian, $\mathcal{L}=(-\Delta)^{\alpha/2}$, whose symbol is $m(\xi)=|\xi|^\alpha$. This is the only possibility for an $\alpha$-stable operator if $N=1$. However, if $N>1$ the spectral measure may be \emph{anisotropic}. Nevertheless, the symbol, and hence the operator, is still homogeneous of order $\alpha$. The spectral measure is also allowed to be singular in some directions (the set of singular directions should have Lebesgue measure zero). As an example,  we have the operator
$$
\mathcal{L}=\sum_{j=1}^N (-\partial_{x_jx_j}^2)^{\alpha/2},
$$
that corresponds to the spectral measure $\mu(\theta)=c_{N,\alpha}\sum_{j=1}^N\delta_{e_j}(\theta)$, where
$\{e_j\}_{j=1}^N$ is the canonical basis in $\mathbb{R}^{N}$ and $c_{N,\alpha}>0$ is a normalization constant. The symbol is in this case a multiple of $\sum_{j=1}^N|\xi_j|^\alpha$.

%%%%%%%%%%%%%%%%
\subsection{Assumptions on the initial data and the convection nonlinearity}
\label{sec:AssumptionsInitialData.Nonlinearity}

Throughout the paper we will always assume that:
\begin{equation}
\label{u_0as}
0\leq u_0\in L^1(\rr^N)\cap L^\infty(\rr^N)\text{ with mass }M:=\|u_0\|_{L^1(\rr^N)}.
\tag{$\textup{A}_{u_0}$}
\end{equation}
The boundedness assumption is not essential, since the problem has an $L^1$--$L^\infty$ smoothing effect: solutions with a possibly unbounded but integrable initial data become bounded for any positive time;  see Remark~\ref{rk:smoothing}(c).

Since we will always consider nonnegative initial data, and the concept of solution for  the equation we will deal with has a comparison principle, see Theorem~\ref{thm.UniquenessPropertiesEntropy} below,  solutions will  be nonnegative. Hence, the convection nonlinearity $|r|^{q-1}r$ can be replaced by  $r^q$.
Moreover, with a change of variables and a change of the measure, $\mu$ to $\tilde \mu(\theta):=\mu(A\theta)$, $A$ being a rotation matrix, it is enough to consider the case $\mathbf{a}=(0,\dots, 0,1)$ in~\eqref{eq:F}. Notice  that rotations preserve~\eqref{mlas}.  Thus, we will always assume
\begin{equation}
F(u):=\mathbf{a}f(u),\quad\text{where }\mathbf{a}=(0,\ldots,0,1)\in\rr^N,\  f(u)=u^q,\text{ and }1-\frac{1}{N}<q\neq1.
\tag{$\textup{A}_q$}
\label{qas}
\end{equation}
The restriction on $q$ from below is imposed to guarantee the conservation of mass. Notice that if $q\in(1-\frac{1}{N},1)$, a case denoted in the literature as \emph{fast} convection (in contrast with the case of \emph{slow} convection, $q>1$), then $f$, and hence $F$, are  not locally Lipschitz, but only  locally $q$-H\"older continuous.

Taking all the above into account, problem~\eqref{eq.main0} can be reduced to
\begin{equation}
\label{eq.main}\tag{P}
\partial_tu+\ml u+\partial_{x_N}(u^q)=0 \quad\text{in }Q,\qquad u(0)=u_0 \quad\text{on }\rr^N.
\end{equation}

%%%%%%%%%%%%%%%%%%%%%%%%%%%%%%%
\subsection{On the concept of solution}
\label{sec:SolutionConcept}

When $\alpha \in (0, 1)$ regularity is not guaranteed\footnote{The case $\alpha=1$ is special, and has to be treated with care. For instance, in~\cite{KNS08} the authors conclude that regularization actually happens for equations of the type \eqref{eq.main} with $\ml=(-\Delta)^{\frac{1}{2}}$ and $q=2$. See also \cite{CV12} for related results and discussions.}, since the convection term may lead to the formation of shocks, and  very weak solutions  (see~Remark~\ref{regularity.issue}(d) below for a precise definition)  are in general not unique~\cite{Alibaud-Andreianov-2010}. Hence the need of a more restrictive notion of solution.  The concept of entropy solution, which is valid in all the range of parameters under consideration, $\alpha\in(0, 2)$, $q>1-\frac1N$, will serve for our purposes. Even though such a concept of solution is not needed if $\alpha\in(1,2)$, shocks may show up in the description of the large-time behaviour also in that range, if convection is strong enough. Hence we prefer to deal always with entropy solutions, since they offer a unified framework for both the original problem and its possible asymptotic limits in all cases.

Unfortunately, there are no references which include fast and slow convection, as well as nonlocal diffusion. Hence, we will build the well-posedness  theory from scratch in Appendix~\ref{sect:appendix.results.entropy.solutions}, based on~\cite{Car99, MaTo03, AlibaudEntropy, CifaniJakobsenEntropySol, EndalJakobsen, AnBr20, Pan20}.

To write down the definition of entropy solution, we split the nonlocal operator as follows:
\begin{equation}\label{eq.SplittingOperatorInTwo}
\begin{split}
&(\ml\phi)(x)=(\ml^{\leq \rho}\phi)(x)+(\ml^{>\rho}\phi)(x)\quad\text{for all }\phi\in C_\textup{c}^{\infty}(\rr^N),\ \rho>0, \text{and }x\in \rr^N,\text{ where}\\
&(\ml^{\leq \rho}\phi)(x):=\int_{\bs^{N-1}}\int_0^\rho \Big(\phi(x)-\frac{\phi(x+r\theta)+\phi(x-r\theta)}2 \Big)\,\frac{{\rm d}r}{r^{1+\alpha}}{\rm d}\mu(\theta),\\
&(\ml^{> \rho}\phi)(x):=\int_{\bs^{N-1}}\int_\rho^\infty \Big(\phi(x)-\frac{\phi(x+r\theta)+\phi(x-r\theta)}2 \Big)\,\frac{{\rm d}r}{r^{1+\alpha}}
{\rm d}\mu(\theta).
\end{split}
\end{equation}
Although  we are assuming~\eqref{u_0as}, we will allow bounded measures as initial data in our definition of entropy solution, since this possibility will appear in the description of large-time behaviours.
\begin{definition}[Entropy solution]\label{def.entropySolution}
A function $u$ is an \emph{entropy solution} of~\eqref{eq.main} if:
\begin{enumerate}[{\rm (a)}]
\item \textup{(Regularity)} $u\in L^\infty((0,\infty);L^1(\rr^N))\cap L^\infty_\textup{loc}((0,\infty);L^\infty(\rr^N))$.
\item \textup{(Entropy inequality)} For all $k\in\rr$, all $\rho>0$, and all $0\leq \phi\in C_\textup{c}^\infty(Q)$,
\begin{equation}\label{eq:entropy.inequality}
\begin{split}
&\iint_Q \Big( |u-k|\partial_t\phi + \sgn(u-k)\big(F(u)-F(k)\big)\cdot\nabla\phi\Big)\\
&\qquad-\iint_Q \Big(\sgn(u-k)\phi\ml^{> \rho}u + |u-k|\ml^{\leq \rho}\phi\Big)\geq0.
\end{split}
\end{equation}
\item\textup{(Initial data in the sense of bounded measures)} For all $\psi\in C_\textup{b}(\rr^N)$,
\begin{equation}
	\label{eq:initial.data}
\esslim _{t\to0^+}\int_{\rr^N}u(t)\psi= \int_{\rr^N}\psi\,{\rm d}u_0.
\end{equation}
\end{enumerate}
\end{definition}

\begin{remark}\label{regularity.issue}
\begin{enumerate}[{\rm (a)}]
\item If $u_0$ satisfies~\eqref{u_0as}, condition~\eqref{eq:initial.data} is implied if the data is taken in an $L^1$-sense. It is then standard to show that instead of assuming the $L^1$-continuity at $t=0$,  we can   take test functions such that $0\leq \phi\in C_\textup{c}^\infty(\overline{Q})$ and add the term $\int_{\rr^N}|u_0-k|\phi(0)$ on the left-hand side of the entropy inequality~\eqref{eq:entropy.inequality}. 

\item If $u_0$ satisfies~\eqref{u_0as}, the entropy solution of~\eqref{eq.main} belongs to $C([0,\infty);L^1(\mathbb{R}^N))$ and mass is conserved, $\int_{\mathbb{R}^N}u(t)=\int_{\mathbb{R}^N}u_0$ for all $t>0$;  see  Theorem~\ref{thm.UniquenessPropertiesEntropy}(b).

\item As a consequence of (b), for general initial data which are just bounded measures we have $u\in C((0,\infty);L^1(\rr^N))$, and $\int_{\mathbb{R}^N}u(t)=\int_{\mathbb{R}^N}{\rm d}u_0$ for all $t>0$.

\item The entropy solution of~\eqref{eq.main} is a \emph{very weak solution} of that problem:
$$
\begin{gathered}
u\in L^1_{\rm loc}\big((0,\infty); L^1(\mathbb{R}^N;\rho\,{\rm d}x)\big)\cap L^q_{\rm loc}(Q),\quad  \rho(x)=(1+|x|)^{-(1+\alpha)},\\
\displaystyle \iint_Q \big(u( \partial_t\phi  - \ml\phi)+u^q\partial_{x_N}\phi\big)=0\quad\text{for all }\phi\in C_\textup{c}^\infty(Q),
\end{gathered}
$$
and~\eqref{eq:initial.data}; see Theorem~\ref{thm.UniquenessPropertiesEntropy}. The introduction of the weighted space $L^1(\mathbb{R}^N;\rho\,{\rm d}x)$ allows for the term $\iint_Q u\mathcal{L}\phi$ to be well defined, since $\mathcal{L}\phi=O(\rho)$, as can be easily checked.
\end{enumerate}
\end{remark}

%%%%%%%%%%%%%%%%%%%%%%%%%%%%%%
\subsection{Main results}
The large-time behaviour of solutions to~\eqref{eq.main} depends on the size of $q$ as compared to the critical value
$$
q_*(\alpha):=1+\frac{\alpha-1}N.
$$
  If $q\ne q_*(\alpha)$ there is a phenomenon of \emph{asymptotic simplification}: if $q>q_*(\alpha)$ the convection term is lost in the limit, while if $q<q_*(\alpha)$ convection is kept, but the diffusion operator $\mathcal{L}$ simplifies to an operator $\mathcal{L}'$ acting only on the first $N-1$ spatial variables, defined for smooth enough functions by
\begin{equation}\label{tilde.L}
\begin{split}
  &(\ml' \phi)(x',x_N):=\int_{\bs^{N-1}}\int_0^\infty\Big(\phi(x',x_N)-\frac{\phi(x'+r\theta',x_N)+\phi(x'-r\theta',x_N)}2 \Big)\,\frac{{\rm d}r}{r^{1+\alpha}}
{\rm d}\mu(\theta ).
\end{split}
\end{equation}
Here, $x=(x',x_N)$, where $x'\in\mathbb{R}^{N-1}$ denotes the first $N-1$ coordinates and $x_N$ the last one.

\begin{theorem}[Large-time behaviour]\label{asymptotic}
Assume~\eqref{u_0as},~\eqref{qas}, and~\eqref{mlas}.
Then the entropy solution $u$ of~\eqref{eq.main} is nonnegative, has mass $M$ for all times, and satisfies, for all $p\in[1,\infty)$,
\begin{equation}\label{main.limit}
t^{\max\{\frac N\alpha, \frac1q(1+\frac{N-1}\alpha)\}(1-\frac 1p)}\|u(t)-U(t)\|_{L^p(\rr^N)}\to0\quad\text{as }t
\to\infty,
\end{equation}
where $U$ is the unique very weak (if $q> q_*(\alpha)$) or entropy (if $q\le q_*(\alpha)$) solution  in $Q$ with initial data $U(0)=M\delta_0$ of
\begin{eqnarray}
\label{lim.1}
\partial_tU+\ml U  =0&&\text{if }q>q_*(\alpha),\\
\label{lim.2}
\partial_tU+\ml U + \partial_{x_N} (U^q)=0&&\text{if } q=q_*(\alpha),\\
\label{lim.3}
\partial_tU+ \ml'{U} + \partial_{x_N}( U^q) =0&&\text{if }
q<q_*(\alpha).
\end{eqnarray}
These solutions are known as the \emph{fundamental} solutions with mass $M$ of the corresponding problem.
\end{theorem}

Let us be precise: as in~\cite{dePablo-Quiros-Rodriguez-2020}, we define a very weak solution of~\eqref{lim.1} with initial data $U(0)=M\delta_0$ as a function $U\in L^1_{\rm loc}\big((0,\infty); L^1(\mathbb{R}^N;\rho\,{\rm d}x)\big)$ such that
\begin{equation}
\label{eq:definition.fundamental.vws}
\iint_Q  U( \partial_t\phi  - \ml\phi)+M\phi(0,0)=0\quad\text{for all }\phi\in C_\textup{c}^\infty(\overline{Q}).
\end{equation}

The case $q>q_*(\alpha)$ described by~\eqref{lim.1} is known as the diffusion, supercritical, or weakly nonlinear regime, while the case $q<q_*(\alpha)$ described by~\eqref{lim.3} is known as the convection, subcritical,   or strongly nonlinear  regime. Finally, when $q=q_*(\alpha)$ and the asymptotic behaviour is described by~\eqref{lim.2} we are in the critical  or self-similar  regime.  In the three regimes the function $U$ describing the large time behaviour satisfies
$$
t^{\max\{\frac N\alpha, \frac1q(1+\frac{N-1}\alpha)\}(1-\frac 1p)}\|U(t)\|_{L^p(\rr^N)}=C\quad\text{for some }C\in(0,\infty).
$$
Hence, the limit~\eqref{main.limit} is meaningful and yields in particular the decay rate for solutions to~\eqref{eq.main}, namely
$$
\|u(t)\|_{L^p(\rr^N)}\eqsim t^{-\max\{\frac N\alpha, \frac1q(1+\frac{N-1}\alpha)\}(1-\frac 1p)}.
$$
Notice that there is a change between a diffusion-like decay rate and a convection-like one precisely when $q=q_*(\alpha)$, since
$$
\frac N\alpha=\frac1{q_*(\alpha)}\Big(1+\frac{N-1}\alpha\Big).
$$
We refer the reader to the Figures \ref{fig:alphastable-convectionNeq1} and \ref{fig:alphastable-convectionNgt1} below.

\begin{figure}[h]
\includegraphics[width=0.6\textwidth]{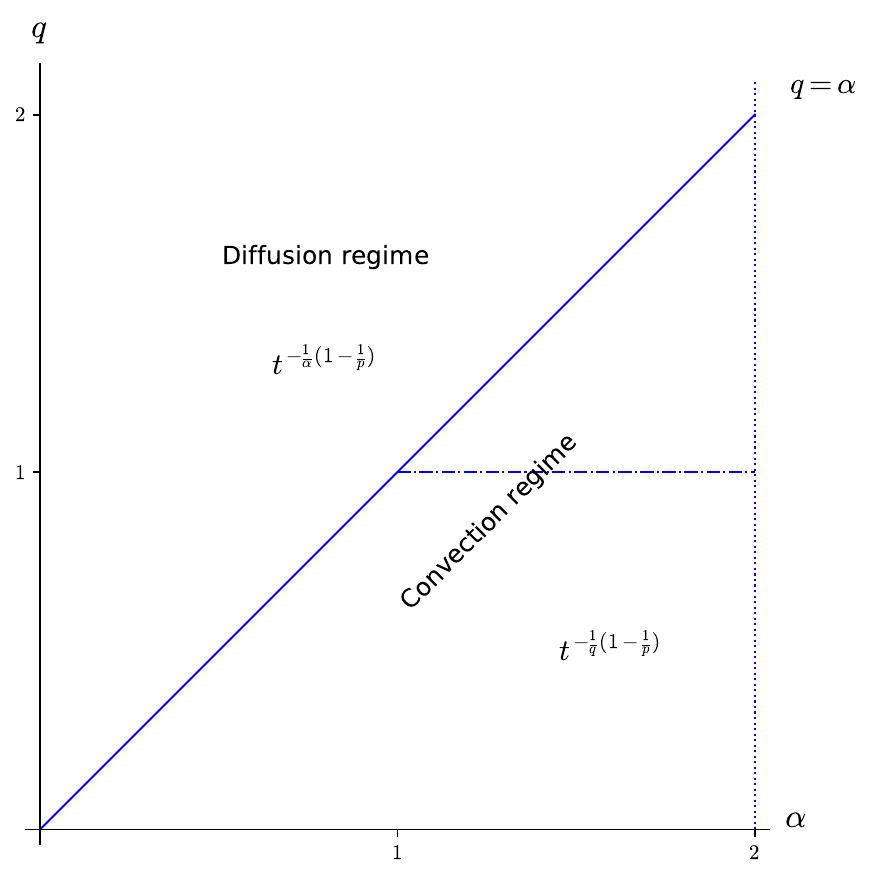}
\caption{Different regimes and $L^p$-decay rates when $N=1$.}
\label{fig:alphastable-convectionNeq1}
\end{figure}
\begin{figure}[h]
\includegraphics[width=0.8\textwidth]{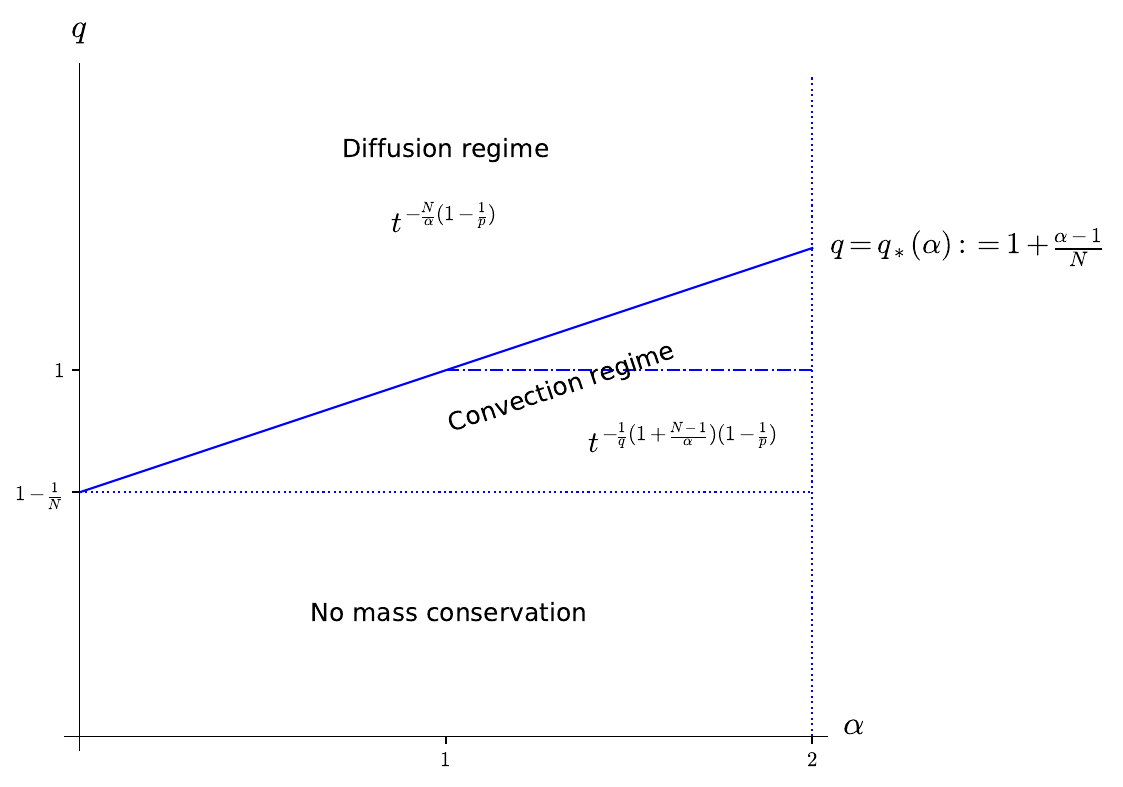}
\caption{Different regimes and $L^p$-decay rates when $N>1$.}
\label{fig:alphastable-convectionNgt1}
\end{figure}

It is worth noticing that Theorem~\ref{asymptotic} is new even in the  isotropic case $\mathcal{L}=(-\Delta)^{\alpha/2}$,
except for the critical exponent $q=q_*(\alpha)$, $\alpha\in[1,2)$, and in the subcritical range $1<q<\alpha$ in dimension $N=1$, which were considered for that operator respectively in~\cite{BilerKarchWoycz2001} and~\cite{Ignat-Stan-2018}.

Let us note that,   even though~\eqref{eq.main} is parabolic when $\alpha>1$, the equation describing the limit is hyperbolic if $q<1+\frac{\alpha-1}{N}$. We already have this phenomenon in the local case $\alpha=2$, in which $\mathcal{L}=-\Delta$. The \lq\lq dual'' situation in which equation~\eqref{eq.main} is hyperbolic and the limit equation is parabolic also occurs, when $\alpha<1$ and $q>q_*(\alpha)$. This phenomenon is  therefore purely nonlocal.  The hyperbolic character also explains why convergence in $L^\infty(\mathbb{R}^N)$ cannot be expected.

\begin{remark}
\begin{enumerate}[{\rm (a)}]
\item
In the case $q=1$, which corresponds to the $\alpha$-stable linear heat equation with a drift, the change of variables $\overline{u}(t,x):=u(t,x',x_N+t)$ yields a solution to the purely diffusive equation~\eqref{lim.1}, whose asymptotic behaviour is given by its fundamental solution $U$ with mass $M$. Hence, the asymptotic behaviour of the solution of the original problem is given by $U(t,x',x_N-t)$.

\item Our proofs and results are still valid if $q=1$ if we are in the diffusion regime or in the critical one. This does not contradict the above remark, since the $L^p$ norm of the difference of the two functions describing the limit, $U(t,x',x_N)-U(t,x',x_N-t)$, is $o(t^{-\frac N\alpha(1-1/p)})$.
\end{enumerate}
\end{remark}

To prove Theorem~\ref{asymptotic}, we will follow the ``four-step method'' developed in~\cite{KamVaz88} by Kamin and V\'azquez. It consists in using the natural scaling invariance of the diffusion or the convection term in~\eqref{eq.main} to build a one-parameter family of solutions which is relatively compact thanks to certain parabolic and hyperbolic  estimates for solutions of~\eqref{eq.main}.  These estimates will follow from their analogues for a regularized version of the problem which are obtained in Section~\ref{sect:estimates}.    To complete the proof, we pass to the limit in the parameter;  the precise details can be found in Section~\ref{sect:scalings}.  A key ingredient is then the uniqueness of the fundamental solution with mass $M$ for the limit problems,  that we state next.

\begin{theorem}[Uniqueness of fundamental solutions]\label{thm.UniquenessOfLimitEquationsIandII}
Assume  ~\eqref{mlas} and  ~\eqref{qas}.
\begin{enumerate}[{\rm (a)}]
\item There is at most one very weak fundamental solution with mass $M$ of the equation in~\eqref{lim.1}.
\item There is at most one  fundamental entropy  solution  with mass $M$  of the equation in~\eqref{lim.2}.
\item There is at most one fundamental entropy solution with mass $M$ of the equation in~\eqref{lim.3}.
\end{enumerate}
\end{theorem}

The proof of (a) was already given in~\cite{dePablo-Quiros-Rodriguez-2020}.  The ones of (b) and (c), which can be found in Section~\ref{sec:UniquenessOfLimitProblems}, are quite technical  and draw inspiration from \cite{EVZIndiana}.  They involve the Hamilton-Jacobi equation
$$
\partial_tv+\ml v -\eps \Delta v =|\partial_{x_N} v|^q
$$
when dealing with~\eqref{lim.2}, with $\ml'$ instead of $\ml$ when dealing with~\eqref{lim.3}. These equations appear after taking the primitive of $u$ in the direction of convection in a regularized version of the original equation.  Proving a comparison principle for these Hamilton-Jacobi equations is a main issue here  since comparison for the primitives will imply uniqueness for the original variable $u$. Usually, such a comparison result would follow from the theory of viscosity solutions, but this is not exactly our setting and we have to develop an independent proof. In the cases $q\in(1-\frac1N,1)$ for the equation in~\eqref{lim.2} in all dimensions and the equation in~\eqref{lim.3} when $N>1$, we had to approximate the convection nonlinearity by a Lipschitz nonlinearity to be able to justify some of the arguments. Hence, the proofs in the Lipschitz and the H\"older cases are similar up to this modification.  If $N=1$, the equation in~\eqref{lim.3} is a pure scalar conservation law, and its uniqueness when $q\in(0,1)$ was already proved in~\cite{LaurencotFast}.

 As a consequence of uniqueness and of the invariance of the limit equations~\eqref{lim.1}--\eqref{lim.3} and the initial data $M\delta_0$ under certain scalings (see the beginning of Section~\ref{sect:scalings}),  we obtain that the fundamental solutions giving the limit behaviour are self-similar. In particular, the fundamental solutions of~\eqref{lim.1} and~\eqref{lim.2} have the form
$$
U(t,x)=t^{-N/\alpha}f(xt^{-1/\alpha})
$$
and the fundamental solution of~\eqref{lim.3} has the form
$$
U(t,x',x_N)=t^{-\gamma}f(x't^{-1/\alpha},x_Nt^{-\beta}),\qquad \beta=\frac1q\Big(1+\frac {N-1}{\alpha}\Big)-\frac {N-1}{\alpha},\quad \gamma=\frac {N-1}{\alpha}+\beta.
$$
The different profiles $f$ depend on $M$, $\alpha$, $q$, and the spectral measure $\mu$, and are not radial in general.

%%%%%%%%%%%%%%%%%%%%%%%%%%%%%%%%%%%
\subsection{Precedents}
As an important precedent, we have the local case in which $\mathcal{L}=-\Delta$, which  corresponds to $\alpha=2$. The Laplacian is invariant under rotations. Hence, in this case problem~\eqref{eq.main0} can be reduced to
$$
\partial_tu-\Delta u+\partial_{x_N}(u^q)=0\quad\text{in }Q, \qquad u(0)=u_0\quad\text{on }\mathbb{R}^N.
$$
If  $u_0$ is integrable, this problem admits a classical solution, whose large-time behaviour is given by
$$
t^{\max\{\frac N2, \frac{N+1}{2q}\}(1-\frac 1p)}\|u(t)-U(t)\|_{L^p(\mathbb{R}^N)}\to0\quad\textup{as }t
\to\infty,
$$
 for all $p\in[1,\infty]$ if $q\geq q_*(2)$, and all $p\in [1,\infty)$ if $q\in(1,q_*(2))$  or $q\in(0,1)$  and $N=1$, 
where $U$  is the unique very weak (if $q\ge q_*(2)=1+\tfrac1N$) or entropy (if $q<q_*(2)$) fundamental solution  in $Q$ with mass $M$ of
$$
\begin{array}{rl}
\partial_tU-\Delta U  =0\quad&\textup{if }q>q_*(2),\\
\partial_tU-\Delta U + \partial_{x_N} (U^q)=0\quad&\textup{if } q=q_*(2),\\
\partial_tU-\Delta_{x'}{U} + \partial_{x_N}(U^q) =0\quad&\textup{if } q\in(1,q_*(2)),\text{ or }
q\in(0,1)\text{ and } N=1.
\end{array}
$$
Here $\Delta_{x'}$ stands for the Laplacian in the first $N-1$ coordinates.
The results for $q>1$ were proved in the remarkable series of papers~\cite{EZ,EVZArma,EVZIndiana, EsZu97} (see also~\cite{Zua20}), and the ones for $q<1$ in~\cite{LaurencotFast}. We will take profit of some of the techniques used in those papers, adapted to our nonlocal setting: scalings, maximum principle, entropy inequalities\dots\ However, as pointed out above,  the nonlocal equation~\eqref{eq.main}
may be hyperbolic, which leads to a lack of regularity that makes the proofs much more involved than in the local case.

As we have already mentioned, there are two precedents  that fall directly into our nonlocal framework,  both of them for the isotropic case in which $\mathcal{L}$ is the fractional Laplacian:~\cite{BilerKarchWoycz2001}, dealing with the critical exponent $q=q_*(\alpha)$ for $\alpha\in[1,2)$ (so that $q\ge 1$), and~\cite{Ignat-Stan-2018}, that considers the subcritical range $1<q<\alpha$ in dimension $N=1$. Another important precedent is~\cite{BilerKarchWoyczAsymp}, in which the diffusion operator is $-\Delta+\mathcal{L}$, with a nonlocal operator $\mathcal{L}$ like the one we are considering here. The Laplacian has a regularizing effect that makes things simpler. This will be exploited later in this paper as well.

\begin{remark}
In the local framework, $\alpha=2$, uniqueness, and hence convergence, for the limit problems~\eqref{lim.2} and~\eqref{lim.3}  when $q\in(1-\frac1N,1)$ is only known for the latter problem when $N=1$. The remaining cases will be treated in a future work.
\end{remark}

Also, in the nonlocal framework, much attention has been paid to the case in which $\mathcal{L}$ is an operator of convolution type,
\begin{equation*}
\label{eq:def.L.convolution}
\mathcal{L}\phi=J*\phi-\phi
\end{equation*}
for some integrable kernel $J$; see for instance~\cite{Laurencot-2005,Ignat-Rossi-2007,diFrancesco-Fellner-Liu-2008,Ignat-Pazoto-2014,CazacuIgnatPazoto}.

%%%%%%%%%%%%%%%%%%%%%%%%%%%%%%%%%%%%%%%%%
\subsection{Comments and extensions}

\noindent\textsc{Solutions with sign changes. } In this paper we have only considered the case of nonnegative solutions. The case of nonpositive solutions is easily reduced to this one. Sign changes in the solutions offer some technical difficulties, which may be handled if the mass $M$ is different from 0 following what was done for the local case in~\cite{EVZArma,Carpio1996}. When $M=0$,  things become more involved in the diffusion range, since the solutions of the limit problem decay faster than the fundamental solution.  As a consequence, the critical line, where the decay rates in the diffusion and convection regimes are expected to coincide, should be different from that of the case $M\neq0$.  This  situation in which solutions of the limit diffusion problem decay faster than the fundamental solution  was considered for the local case in~\cite{Karch-Schonbek-2002}. It may be interesting to see whether the techniques in that paper apply to the nonlocal case.

\noindent\textsc{Nonintegrable initial data. } Problem~\eqref{eq.main0} makes sense for initial data which are bounded but not necessarily in $L^1$. What is then the asymptotic behaviour? We still expect a diffusion regime and a convection one. However, in the diffusion regime the behaviour should not be given by a fundamental solution (solutions have infinite mass). If the initial data has a precise power-like nonintegrable decay at infinity, solutions of the purely diffusive problem converge towards a nonintegrable self-similar solution, with a decay that is different from the one of fundamental solutions; see~\cite{Herraiz-1999}. This special solution is expected to give the behaviour in the diffusion regime. However, the critical line, where the decay rates of the limit diffusive and convective problems coincide, should change with respect to the integrable case.

\noindent\textsc{Nonhomogeneous media. }  It would be interesting to think about nonlocal models involving \lq\lq non-constant diffusivities\rq\rq, and to study their large-time behaviour. In the local case, if the diffusivity approaches a constant at infinity, there is an asymptotic simplification towards the problem with diffusivity equal to this constant; see~\cite{Duro-Zuazua-1999}. In the nonlocal setting things might be very different.

\noindent\textsc{Nonlinear diffusion. } The results in the local case have been extended by several authors to the case in which the diffusion operator is a nonlinear operator of porous medium type, $\mathcal{L}=-\Delta u^m$, $m>1$; see for instance~\cite{LaurencotFast,LaurencotSimondon,Rey99, ReVa99,Escobedo-Feireisl-Laurencot-2000,Carrillo-Fellner-2005,Andreucci-Tedeev-2008}.  Two different nonlocal nonlinear diffusion operators of porous medium-type have become popular in the last years, $\mathcal{L}u=(-\Delta)^{\alpha/2} u^m$ and $\mathcal{L}u=-\operatorname{div}(u^{m-1}\nabla(-\Delta)^{-\alpha/2}u)$; see for instance~\cite{dePablo-Quiros--Rodriguez-Vazquez-2012} for the former and~\cite{Caffarelli-Vazquez-2011} for the latter. Diffusion-convection problems for operators of this kind have only been considered for the second operator in spatial dimension $N=1$ by means of entropy methods~\cite{Feo-Huang-Volzone-2020}. Such methods are not suitable for the first model, hence a different approach is needed.

\noindent\textsc{Other convection nonlinearities. }  It should be possible to extend our results to the case in which the convection nonlinearity is not exactly a power, as long as it behaves like one of them for $u\sim0$; see e.g.~\cite{EVZArma,EZ} for the local case.

\noindent\textsc{Multidirectional convection. } It would also be interesting to consider convection nonlinearities $F$ in~\eqref{eq.main0} having different behaviours for $u\sim0$ in different directions. To our knowledge, this situation has only been considered up to now in the local case, for problems involving different power-like convection nonlinearities in different directions, both with linear and nonlinear diffusion, under conditions that guarantee that asymptotically convection acts only in one direction; see~\cite{EsZu97,Escobedo-Feireisl-Laurencot-2000}.  The general case is more involved, as pointed out by the  recent paper \cite{Ser21}, which shows that the multidimensional Burgers equation with a Dirac delta as initial data is not well-posed: either the solution does not exist or it is not self-similar. So, even if such an equation enjoys $L^1$--$L^\infty$-smoothing effects \cite{SeSi19}, we cannot obtain its asymptotic behaviour with the methods developed in this paper.

\noindent\textsc{Nonlocal convection. } The papers~\cite{Ignat-Rossi-2007,MR2888353,MR3328145} study  diffusion-convection problems with a diffusion operator of convolution type in which also the convection term is nonlocal. In both cases the involved kernel is integrable with some finite moments. The asymptotic behaviour of models involving singular kernels in both the diffusion and convection terms, and the manner in which they interact, remains to be analysed.

\noindent\textsc{Numerical schemes. } It would be interesting to develop numerical schemes reproducing qualitatively the large-time behaviour of solutions of~\eqref{eq.main}. Such issues have been considered  in problems involving both local and nonlocal (of convolution type) diffusion operators  for instance in~\cite{Ignat-Pozo-Zuazua-2015,Ignat-Pozo-2017,Ignat-Pozo-2018}.

%%%%%%%%%%%%%%%%%%%%%%%%%%%%%%%%%%%%%%%%%%%%%
\subsection{Notation}\label{sec:notation}
Let us explain some of the notations that will be used throughout the paper.

\noindent\emph{Asymptotic symbols. } Let $f$ and $g$ be positive functions.  By $f\eqsim g$ we mean that there are constants $c, C>0$ such that $c\leq f/g\leq C$, and by $f\lesssim g$ that there is a constant $C>0$ such that $f\le C g$.

\noindent\emph{Vector valued functions. } We will often identify $u(t)$ with $u(t,\cdot)$. Strictly speaking, the notation $u(t)$ means a function $[0,T]\to X$ for some Banach space $X$, while $u(t,\cdot)$ denotes a function $[0,T]\times \rr^N\to \rr$, and hence, we mean that $u(t)$ is an a.e.~representative of $u(t,\cdot)$.

\noindent\emph{Tail-control and cut-off functions. } In order to control the tails of solutions we use the functions
$\rho_R$, defined for all $R>0$ by $\rho_R(x):=\rho(x/R)$, where $0\leq \mathcal{\rho}\in C^\infty(\rr^N)$ is such that
$$
\rho(x)=\begin{cases}0&\text{for }|x|\leq1,\\
1&\text{for }|x|\geq2.
\end{cases}
$$
As cut-off functions we will use $\mathcal{X}_R:=1-\rho_R$.  Note that $\rho_R\to0$, $\mathcal{X}_R\to 1$ pointwise as $R\to\infty$.

\noindent\emph{Standard mollifiers. } They will be denoted by $\omega_\delta$.

\noindent\emph{Energy spaces. }
Besides the standard fractional Sobolev spaces $W^{\alpha,p}(\mathbb{R}^N)$, we will also consider the Bessel potential spaces $(I-\Delta)^{-\alpha/2}L^p(\rr^N)$, $1<p<\infty$, that we will denote by $H^{\alpha,p}(\mathbb{R}^N)$.\footnote{They will be important when proving regularity estimates for problems of the type \eqref{eq.main}, see Proposition \ref{prop.propOfRegularizedEquationSmoothness2}.}

The bilinear form associated to $\mathcal{L}$,
$$
\begin{array}{rl}
	\mathcal{E}(u,v)=&\displaystyle\frac14\int_{\mathbb{R}^{N}}\int_{\mathbb{R}^{N}}
	\big(u(x+y)-u(x)\big)\big(v(x+y)-v(x)\big)\,{\rm d}\nu(y){\rm d}x \\ [4mm]
	&\displaystyle+\frac14\int_{\mathbb{R}^{N}}\int_{\mathbb{R}^{N}}
	\big(u(x)-u(x-y)\big)\big(v(x)-v(x-y)\big)\,{\rm d}\nu(y){\rm d}x,
\end{array}
$$
is well defined in the energy espace $X:=\left\{ u\colon \mathbb{R}^N \to \mathbb{R} \text{ measurable}:\mathcal{E}(u,u)<\infty \right\}$. Since
$$
\mathcal{E}(u,u)\eqsim\int_{\mathbb{R}^N}|\xi|^\alpha|\hat u(\xi)|^2\,{\rm d}\xi=\|(-\Delta)^{\alpha/4}u\|_{L^2(\mathbb{R}^N)}^2,
$$
this space, endowed with the norm $\mathcal{E}(u,u)^{\frac12}$ is equivalent to the homogeneous Sobolev space $\dot W^{\frac\alpha2,2}(\mathbb{R}^N)$.

\noindent\emph{Dual spaces. } Given an space $X$ and its dual $X'$, we denote the associated duality  pairing  by $\langle\cdot,\cdot\rangle_{X\times X'}$.

\noindent\emph{Sign function. }  By $\sgn$ we denote the sign function defined by
$$
\sgn(x)=\begin{cases}
-1,&x<0,\\
0,&x=0,\\
1,&x>0.
\end{cases}
$$

%%%%%%%%%%%%%%%%%%%%%%%%%%%%%%%%%%%%%%%%%%%%%%%%%%%%%%%%%%%%%%%%%%%%%%%%%%%%%%
\section{On the various nonlocal operators}
\label{sect:SomeDiscussionOnTheVariousNonlocal}
\setcounter{equation}{0}
In this section,  we present the operators that appear in our analysis and some of their properties that we will use in  the paper. The first two operators $\ml$ and $\ml'$ were already introduced in the first section, respectively in~\eqref{eq:operador en polares} and~\eqref{tilde.L}.
The third operator, $\widetilde{\ml}$, closely related to $\mathcal{L}'$,  appears by integrating in the direction of convection, as we see next.

%%%%%%%%%%%%%%%%%%%%%%%%%%%%%5
\subsection{The operator $\protect\widetilde{\ml}$}
Given a smooth  function $u:\rr^N\rightarrow\rr$ such that
\begin{equation}
\label{v.integrate}
  v( x')=\int_{\rr} u(x', x_N)\,{\rm d}x_N, \quad x'\in \rr^{N-1},
\end{equation}
is well defined,  our goal is to construct a measure $\tilde \mu$ on $\bs^{N-2}$ satisfying the non-degeneracy and finiteness conditions in~\eqref{mlas} such that if we define
\begin{equation}
\label{tildeL}
  (\widetilde{\mathcal{L}}v)(x')= \int_{\bs^{N-2}}\int_0^\infty\Big(v( x')-\frac{v(x'+r\sigma')+v(x'-r\sigma')}2\Big)\,\frac{{\rm d}r}{r^{1+\alpha}}{\rm d}\tilde\mu(\sigma'),
\end{equation}
then 
\begin{equation}
\label{id.widetilde}
  (\widetilde{\mathcal{L}}v)( x')=\int_\rr (\mathcal{L}u)( x',x_N)\,{\rm d}x_N,\quad x'\in \rr^{N-1}.
\end{equation}

We start with the measurable spaces $(\bs^{N-1}, B_{\bs^{N-1}})$ and $(\bs^{N-2}, B_{\bs^{N-2}})$, which are endowed with the $\sigma$-algebra of Borel sets. Let $\mu$ be a Borel measure on $(\bs^{N-1}, B_{\bs^{N-1}})$ satisfying the conditions for the spectral measure in~\eqref{mlas}. We define a weighted Borel measure $\overline \mu$ on $(\bs^{N-1}, B_{\bs^{N-1}})$ by
\[
\begin{gathered}
{\rm d}\overline \mu (\sigma)=(1-\sigma_N^2)^{\alpha/2}\, {\rm d}\mu(  \sigma',\sigma_N), \quad \sigma=(\sigma',\sigma_N)\in \bs^{N-1},\quad\text{i.e., }
\\
\int_{\bs^{N-1}}\psi(\sigma)\,{\rm d}\overline \mu (\sigma)=\int_{\bs^{N-1}}\psi( \sigma', \sigma_N)(1-\sigma_N^2)^{\alpha/2} \,{\rm d}\mu (  \sigma', \sigma_N) \quad \text{for all } \psi\in C_{\rm b}(\bs^{N-1}).
\end{gathered}
\]
Let $\bs^{N-1}_*=\bs^{N-1}\setminus\{|\sigma_N|=1\}$. We consider
the map $T:(\bs^{N-1}_*, B_{\bs^{N-1}_*})\rightarrow (\bs^{N-2}, B_{S^{N-2}})$ defined by
\[
T(\sigma', \sigma_N)=\frac{  \sigma'}{{(1-\sigma_N^2)^{1/2}}},
\]
which is continuous and measurable. Let $\tilde \mu$  be the push-forward measure of $\overline \mu$ under $T$,
$\tilde \mu = T_\ast \overline \mu$, i.e.
\[
\tilde \mu (B)=\overline \mu (T^{-1}(B)) \quad \text{for all } B\in B_{\bs^{N-2}}.
\]
For any $\psi\in C_{\rm b}(\bs^{N-2})$,  we have
\begin{equation}\label{star}
\int_{\bs^{N-2}} \psi (\sigma')\,{\rm d}\tilde \mu (\sigma')=\int_{\bs^{N-1}_*} \psi(T(\sigma))\,{\rm d}\overline \mu (\sigma)=
\int_{\bs^{N-1}_*} \psi\Big(\frac{  \sigma'}{(1-\sigma_N^2)^{1/2}}\Big)(1-\sigma_N^2)^{\alpha/2} \,{\rm d}\mu(\sigma',\sigma_N).
\end{equation}
In particular, taking $\psi\equiv 1$,  we get that $\tilde \mu$ is a finite measure,
\[
\tilde \mu(\bs^{N-2})=\int_{\bs^{N-1}_*}(1-\sigma_N^2)^{\alpha/2}\, {\rm d}\mu(\sigma',\sigma_N)\leq \mu (\bs^{N-1}),
\]
and, taking $\psi(\sigma')=|\xi'\cdot \sigma' |$ with $\xi',\sigma'\in \bs^{N-2}$, we deduce that  it satisfies the non-degeneracy condition  in~\eqref{mlas},
\[
\begin{aligned}
\int_{\bs^{N-2}} |\xi'\cdot \sigma'|^\alpha\, {\rm d}\tilde \mu (\sigma')&=\int_{\bs^{N-1}_*}\Big|\xi'\cdot \frac{\sigma'}{(1-\sigma_N^2)^{1/2}}  \Big|^\alpha (1-\sigma_N^2)^{\alpha/2}\,{\rm d}\mu (\sigma',\sigma_N)\\
&=
\int_{\bs^{N-1}_*}|\xi'\cdot \sigma' |^\alpha \,{\rm d}\mu (\sigma',\sigma_N)\geq \Lambda_2.
\end{aligned}
\]

We finally check that, with this choice,~\eqref{id.widetilde} is fulfilled. 
\begin{lemma}
Let $u\in C_{\rm c}^\infty(\rr^N)$  and $\tilde\mu$ given by~\eqref{star}.  If $v$ is defined by~\eqref{v.integrate} and  $\widetilde \ml v$ by~\eqref{tildeL}, then~\eqref{id.widetilde} holds.
\end{lemma}
\begin{proof}
Using~\eqref{tildeL}, the first identity in~\eqref{star},~\eqref{v.integrate}, and a change in the variable $x_N$, we get
\begin{align*}
&(\widetilde{\mathcal{L}}v)( x')= \int_{\bs^{N-2}}\int_0^\infty \Big(v( x')-\frac{v(x'+r\sigma')+v(x'-r\sigma')}2\Big)\frac{{\rm d}r}{r^{1+\alpha}}{\rm d}\tilde\mu(\sigma')\\
&=  \int_{\bs^{N-1}_*}\int_0^\infty\Big(v( x')-\frac 12 \sum_{\pm}v\Big(x'\pm \frac{r\sigma'}{(1-\sigma_N^2)^{1/2}}\Big)\Big)\frac{{\rm d}r}{r^{1+\alpha}}{\rm d}\overline\mu(\sigma',\sigma_N) \\
&=  \int_{\bs^{N-1}_*}\int_0^\infty\int_{\rr}\Big(u( x',x_N)- \frac 12 \sum_{\pm}u\big(x'\pm \frac{r  \sigma'}{(1-\sigma_N^2)^{1/2}},x_N\big)\,{\rm d}x_N\Big)\frac{{\rm d}r}{r^{1+\alpha}} {\rm d}\overline\mu(\sigma',\sigma_N) \\
&=\int_{\bs^{N-1}_*}\int_0^\infty \int_{\rr}\Big(u(x',x_N)- \frac 12 \sum_{\pm}u\big(x'\pm\frac{r\sigma'}{(1-\sigma_N^2)^{1/2}},x_N\pm \frac{r\sigma_N}{(1-\sigma_N^2)^{1/2}}\big)\Big)\,{\rm d}x_N \frac{{\rm d}r}{r^{1+\alpha}} {\rm d}\overline\mu(\sigma',\sigma_N) .
\end{align*}
A further change of variables,  $r\mapsto r(1-\sigma_N^2)^{1/2}$, and the second identity in~\eqref{star} give us
\begin{align*}
(\widetilde{\mathcal{L}}v)( x')&=\int_{\bs^{N-1}_*}\int_0^\infty \int_{\rr}\big(u( x',x_N)-\frac 12 \sum_{\pm} u((x',x_N)\pm r(\sigma',\sigma_N) )\big)\,{\rm d}x_N \frac{{\rm d}r}{r^{1+\alpha}}{\rm d}\mu(\sigma',\sigma_N).
\end{align*}
We observe now that $\int_{\rr}\big(u( x',x_N)-\frac 12 \sum_{\pm} u(x',x_N\pm r)\big)\,{\rm d}x_N=0$ for every $r>0$ and a.e. $x'\in\mathbb{R}^{N-1}$. Hence, using also Fubini's theorem,
\begin{align*}
(\widetilde{\mathcal{L}}v)( x')&=\int_0^\infty
\int_{\bs^{N-1}_*}\int_{\rr}\big(u( x',x_N)-\frac 12 \sum_{\pm} u((x',x_N)\pm r( \sigma',\sigma_N) )\big)\,{\rm d}x_N {\rm d}\mu(\sigma',\sigma_N) \frac{{\rm d}r}{r^{1+\alpha}} \\
 &\quad+ \int_0^\infty\mu ((0',\pm 1))\int_{\rr}\big(u( x',x_N)-\frac 12 \sum_{\pm} u(x',x_N\pm r)\big)\,{\rm d}x_N \frac{{\rm d}r}{r^{1+\alpha}}\\
 &=\int_{\rr}\int_{\bs^{N-1}}\int_0^\infty
\big(u( x',x_N)-\frac 12 \sum_{\pm} u((x',x_N)\pm r( \sigma',\sigma_N) )\big)\,\frac{{\rm d}r}{r^{1+\alpha}} {\rm d}\mu(\sigma',\sigma_N){\rm d}x_N \\
&=\int_\rr (\mathcal{L}u)( x',x_N)\,{\rm d}x_N,
 \end{align*}
as desired.
 \end{proof}
%%%%%%%%%%%%%%%%%%%%%%%%%%%%%%%%%%%%%%%%%%%%%%%%%%%%
\subsection{Fourier symbols}  We already know the Fourier symbol $m(\xi)$ of $\ml$; see~\eqref{eq.SymbolL}--\eqref{eq:multiplier.L}. We obtain now the ones of the truncated operators. We start with the \lq\lq inner'' ones.
\begin{lemma}\label{symbols}Let $\rho>0$. The symbols of the operators $\ml ^{\le\rho}$, $\ml'^{,\le\rho}$, and $\widetilde \ml ^{\le\rho}$ are, respectively:
\begin{align}
%\label{sym1}
&m^{\le\rho}(\xi)=\int_{\bs^{N-1}}|\theta\cdot \xi |^{\alpha}c_{\le,\alpha}(\rho|\theta\cdot\xi|)\, {\rm d}\mu(\theta),\quad &&\xi \in \rr^N,\nonumber\\
\label{sym2}
&m'^{,\le\rho}(\xi',\xi_N)=\int_{\bs^{N-1}}|\theta'\cdot \xi' |^{\alpha}c_{\le,\alpha}(\rho|\theta'\cdot\xi'|) \,{\rm d}\mu(\theta),\quad &&(\xi',\xi_N)\in \rr^{N-1}\times\rr,\\
\label{sym3}
&\widetilde m^{\le\rho}(\xi')=\int_{\bs^{N-1}}|\theta'\cdot \xi' |^{\alpha}c_{\le,\alpha}\Big(\rho\frac{|\theta'\cdot\xi'|}{(1-\theta_N^2)^{1/2}}\Big)\, {\rm d}\mu(\theta),\quad  &&\xi'\in \rr^{N-1},
\\
&\text{where } c_{\le,\alpha}(s)=\int_0^s \frac{1-\cos t}{t^{1+\alpha}} \,{\rm d}t.\nonumber
\end{align}
\end{lemma}
\begin{remark}\label{symbols.3}
 The value at infinity of $c_{\le,\alpha}$ coincides with the constant $C_\alpha$ in~\eqref{eq.SymbolL}. Hence, when $\rho=\infty$,  both  $\ml'$ and $\widetilde \ml$ have the same symbol,
\[
m'(\xi',\xi_N)=\widetilde m(\xi')\mathds{1}_{\mathbb{R}}(\xi_N);
\]
see the representations~\eqref{sym2} and~\eqref{sym3}. In view of this, we emphasize that
\begin{equation}
\label{eq:L'.Ltilde}
(\ml' \phi)(x',x_N)=(\widetilde\ml \phi (\cdot,x_N))(x')\quad\text{for any }\phi\in \mathcal{S}(\rr^{N}),
\end{equation}
and also that
the three operators $\ml$, $\ml'$ and $\widetilde{\mathcal{L}}$ satisfy
\[
\ml(\psi(\cdot'){\mathds{1}}_{\mathbb{R}}(\cdot_N))(x',x_N)={\ml'}(\psi(\cdot') \mathds{1}_{\mathbb{R}}(\cdot_N))(x',x_N)=(\widetilde{\mathcal{L}}\psi )(x')\mathds{1}_{\mathbb{R}}(x_N)\quad\text{if }\psi\in \mathcal{S}(\rr^{N-1}).
\]
\end{remark}
\proof
Using Fubini's theorem and a change of variables,  we get
\begin{align*}
\widehat{\ml^{\le\rho}\phi}(\xi)&=\widehat \psi(\xi) \int_{\bs^{N-1}}\int_0^\rho\frac{1-\cos(r\theta\cdot \xi)}{r^{1+\alpha}} \,{\rm d}r {\rm d}\mu(\theta)=
\widehat \psi(\xi) \int_{\bs^{N-1}, \theta\cdot \xi\neq 0}\int_0^\rho\frac{1-\cos(r\theta\cdot \xi)}{r^{1+\alpha}}\,{\rm d}r {\rm d}\mu(\theta)\\
&=\widehat \psi(\xi) \int_{\bs^{N-1},  \theta\cdot \xi\neq 0} |\theta\cdot \xi|^{\alpha}\int_0^{\rho |\theta\cdot \xi|}\frac{1-\cos r}{r^{1+\alpha}}\, {\rm d}r {\rm d}\mu(\theta)=\widehat \psi(\xi) \int_{\bs^{N-1} } |\theta\cdot \xi|^{\alpha}c_\alpha(\rho |\theta\cdot \xi|) \, {\rm d}\mu(\theta).
\end{align*}
A similar argument works for $m'^{,\le\rho}$.

As for $\widetilde m^{\le\rho}$ we have
\begin{align*}
\widetilde m^{\le\rho}(\xi')& =  \int_{\bs^{N-2}} |\theta'\cdot \xi'|^{\alpha}c_\alpha(\rho|\theta'\cdot\xi'|)\, {\rm d}\tilde\mu(\theta)\\
&=\int_{\bs^{N-1}_*}
\Big|\frac{\xi'\cdot \theta'}{(1-\theta_N^2)^{1/2}}\Big|^\alpha c_\alpha\Big(\rho \Big|\frac{\xi'\cdot \theta'}{(1-\theta_N^2)^{1/2}}\Big| \Big)(1-\theta_N^2)^{\alpha/2}\, {\rm d}\mu(\theta)\\
&=\int_{\bs^{N-1}_*} |\xi'\cdot \theta'|^\alpha c_\alpha\Big(\rho\frac{|\theta'\cdot\xi'|}{(1-\theta_N^2)^{1/2}}\Big) {\rm d}\mu(\theta)=\int_{\bs^{N-1}} |\xi'\cdot \theta'|^\alpha c_\alpha\Big(\rho\frac{|\theta'\cdot\xi'|}{(1-\theta_N^2)^{1/2}}\Big)\, {\rm d}\mu(\theta).
\end{align*}
\endproof

An analogous proof gives the symbols of the \lq\lq outer'' operators, which we include here for the sake of completeness.
\begin{lemma}\label{symbols.outer}Let $\rho>0$. The symbols of the operators  $\ml ^{>\rho}$, $\ml'^{,>\rho}$, $\widetilde \ml ^{>\rho}$ are, respectively:
\begin{align*}
&m^{>\rho}(\xi)=\int_{\bs^{N-1}}|\theta\cdot \xi |^{\alpha}c_{>,\alpha}(\rho|\theta\cdot\xi|)\, {\rm d}\mu(\theta),\quad &&\xi \in \rr^N,\\
&m'^{,>\rho}(\xi',\xi_N)=\int_{\bs^{N-1}}|\theta'\cdot \xi' |^{\alpha}c_{>,\alpha}(\rho|\theta'\cdot\xi'|) \,{\rm d}\mu(\theta),\quad &&(\xi',\xi_N)\in \rr^{N-1}\times\rr,\\
&\widetilde m^{>\rho}(\xi')=\int_{\bs^{N-1}}|\theta'\cdot \xi' |^{\alpha}c_{>,\alpha}\Big(\rho\frac{|\theta'\cdot\xi'|}{(1-\theta_N^2)^{1/2}}\Big)\, {\rm d}\mu(\theta),\quad  &&\xi'\in \rr^{N-1},\\
&\text{where }c_{>,\alpha}(s)=\int_s^\infty \frac{1-\cos t}{t^{1+\alpha}} \,{\rm d}t.
\end{align*}
\end{lemma}

%%%%%%%%%%%%%%%%%%%%%%%%%%%%%%%%%%%%%%%%%
\subsection{Operators acting on $H^{\alpha,p}(\rr^N)$ spaces} Since the  symbol $m(\xi)$ of the operator $\ml$ is comparable with $|\xi|^\alpha$, $m(\xi)\eqsim |\xi|^\alpha$, we immediately get $\| \ml\varphi\|_{L^2(\rr^N)} \eqsim \|(-\Delta)^{\alpha/2} \varphi\|_{L^2(\rr^N)}$. An analogous, not trivial, result in $L^p(\rr^N)$ spaces, $p\in(1,\infty)$,
\begin{equation}\label{ml.lp}
\| \ml \varphi\|_{L^p(\rr^N)}\eqsim \|(-\Delta)^{\alpha/2}\varphi\|_{L^p(\rr^N)} \quad \text{for all } \varphi\in H^{\alpha,p}(\rr^{N}),
\end{equation}
was proved in \cite[Corollary 4.4]{MR3082477}.  Moreover, in view of~\eqref{sym2} the symbol of the operator $\ml'$ is independent of the variable $\xi_N$.  Thus
 $\ml'$ has the same property \eqref{ml.lp} when freezing the last variable,
\[
\| \ml' \varphi(\cdot',x_N)\|_{L^p(\rr^{N-1})}\eqsim \|((-\Delta_{x'})^{\alpha/2}\varphi ) (\cdot',x_N) \|_{L^p(\rr^{N-1})};
\]
see~\eqref{eq:L'.Ltilde}.  Then, integrating the above inequality in the last variable we get
\[
\| \ml' \varphi \|_{L^p(\rr^{N})}\eqsim \|(-\Delta_{x'})^{\alpha/2}\varphi   \|_{L^p(\rr^{N})}.
\]

We show now that the map $a(\xi)=|\xi'|^{\alpha}/|\xi|^\alpha$ is an $L^p(\rr^N)$ multiplier for $1<p<\infty$. Indeed, since   $a\in C^\infty(\rr^{N}\setminus\{0\})$ and $a$ is homogenous of order zero, $a(\lambda \xi)=a(\xi)$, we can apply Marcinkiewicz's multiplier theorem (see \cite[Th. 5.2.4]{MR3243734} and the comments on page 366 of the same reference) to obtain:
\begin{equation}\label{ml'.lp}
\| \ml' \varphi\|_{L^p(\rr^N)} \lesssim \|(-\Delta)^{\alpha/2}\varphi\|_{L^p(\rr^N)} \lesssim  \| \varphi\|_{H^{\alpha,p}(\rr^N)} \quad \text{for all }\varphi\in H^{\alpha,p}(\rr^{N}).
\end{equation}

%%%%%%%%%%%%%%%%%%%%%%%%%%%%%%%
\subsection{Examples}
As a first example,  we consider the case in which  the spectral measure $\mu$  is absolutely continuous with respect to  Lebesgue's  surface measure on $\bs^{N-1}$, $\sigma_{N-1}$, so that ${\rm d}\mu(\theta) =h(\theta)\, {\rm d}\sigma_{N-1}(\theta)$ for some density function $h\in L^1(\bs^{N-1})$.  Then the new measure $\tilde \mu$ is given by
\[
{\rm d}\tilde \mu (\sigma)=\Big( \sum_{\pm} \int_0^1 \frac{s^{N-2+\alpha}}{(1-s^2)^{1/2}}h (s\sigma,\pm (1-s^2)^{1/2})\,{\rm d}s\Big)\,{\rm d}\sigma_{N-2}(\sigma),\quad
 \sigma\in \bs^{N-2}.
 \]

As a second example,  we consider the spectral measure $\mu=\sum_{j=1}^N\delta _{a_j}$, which is  concentrated at the points $\{a_j\}_{j=1}^N\subset \bs^{N-1}$. We assume that  the hyperplane generated by these points does not contain the origin, so that the nondegeneracy condition in  hypothesis~\eqref{mlas} is satisfied.
In this case
\[
\tilde \mu=\sum_{k=1}^N |Pa_k|^\alpha\delta_{\frac{Pa_k}{|Pa_k|}},
\]
where $Pa$ denotes the projection of the point $a\in \bs^{N-1}$ onto its first $N-1$ components. Note that if $a$ is the north or south pole the corresponding term vanishes. Thus, if
$$
\mathcal{L}=\sum_{j=1}^N (-\partial_{x_jx_j}^2)^{\alpha/2},
$$
an operator that corresponds to the spectral measure $\mu=c_{\alpha}\sum_{j=1}^N\delta_{e_j}$, where
$\{e_j\}_{j=1}^N$ is the canonical basis in $\mathbb{R}^{N}$, the measure corresponding to the projected operator is $\tilde \mu=c_{\alpha}\sum_{j=1}^{N-1}\delta_{e_j}$,  and hence $\widetilde \ml = \sum_{j=1}^{N-1} (-\partial_{x_jx_j}^2)^{\alpha/2}$.

%%%%%%%%%%%%%%%%%%%%%%%%%%%%%%%%%%%%%%%%%%%%%%%%%%%%%%%%%%%%%%%%%%%%%%%%%%%%%%%%%%%%%
\section{Parabolic and hyperbolic estimates for the regularized problem}
\label{sect:estimates}
\setcounter{equation}{0}

The main purpose of this section is to obtain estimates for the regularized problem
\begin{equation}\label{reg}\tag{$\textup{P}_{\varepsilon}$}
\partial_tu+\ml u+\partial_{x_N} (f(u))=\varepsilon \Delta u\quad\textup{in }Q,\qquad
  	u(0)=u_0\quad\textup{in }\rr^N,
\end{equation}
which will be used later  to get uniform estimates for a one-parameter family of scaled versions of the solution to~\eqref{eq.main}. These estimates will provide the required compactness leading to the proof of our convergence result. As a preliminary step we need some well-posedness and a priori estimates  for classical and entropy solutions of this equation under various conditions on $u_0$. We state some of them here;  the proofs can be found in Appendix~\ref{sect:appendix.results.entropy.solutions}. For simplicity, we will not explicitly write $u_\varepsilon$ in what follows. 

%%%%%%%%%%%%%%%%%%
\subsection{Approximation by classical solutions}

\begin{lemma}[Classical solutions]\label{lem.ClassicalSolutionsOfReg}
Assume $p\in[1,\infty)$, $\teps>0$, $\teps\leq u_0\in L^\infty(\rr^N)$,~\eqref{qas}, and~\eqref{mlas}. Then there exists a unique classical solution $u\in C_\textup{b}^\infty(Q)\cap C([0,\infty);L_\textup{loc}^p(\rr^N))$ of~\eqref{reg}. It moreover satisfies $\teps\leq u(t,x)\leq \esssup_{x\in\rr^N}u_0(x)$ for all $(t,x)\in Q$.
\end{lemma}

\begin{remark}
 The classical solution provided by Theorem~\ref{lem.ClassicalSolutionsOfReg} is a mild solution of~\eqref{reg} in the sense of semigroups  that takes the initial data in an $L_\textup{loc}^p$-sense.
\end{remark}

Since we can also have $1-\frac{1}{N}<q<1$ in~\eqref{reg}, it is not clear if one can obtain classical solutions of~\eqref{reg} in the general case $u_0\geq 0$.  Note, however, that when $q>1$ (or when the convection nonlinearity is Lipschitz-regularized),  classical solutions do exist by Proposition~\ref{prop.propOfRegularizedEquationSmoothness2}. To present a unified theory, we therefore rely on entropy solutions.

\begin{lemma}[Entropy solutions]\label{lem.EntropySolutionsOfReg}
Assume~\eqref{u_0as},~\eqref{qas}, and~\eqref{mlas}. Then there exists a unique entropy solution $u$ of~\eqref{reg}. This solution satisfies:
\begin{enumerate}[{\rm (a)}]
\item \textup{(Regularity)} For all $p\in[1,\infty)$,
$$
u\in   L^\infty(Q)\cap C([0,\infty);L^p(\rr^N))\cap L^2((0,\infty);H^1(\rr^N)\cap W^{\alpha/2,2}(\rr^N)).
$$
\item \textup{(Conservation of mass)} For all $t>0$,
$$
\int_{\rr^N}u(t)=\int _{\rr^N}u_0.
$$
\item \textup{(Preliminary estimate on the time derivative)} For all smooth bounded domain $\Omega\subset \rr^N$,
$$
\partial_t u\in L^2((0,\infty);H^{-1}(\Omega)).
$$\item \textup{(Preliminary tail control)} Let $\rho_R$ be a family of tail-control functions. For all $R>0$ and $t>0$,
\begin{equation*}
\label{tail.control}
\begin{split}
&\int_{|x|>2R}u(t)\lesssim \int_{|x|> R}u_0+M t\|\ml\rho_R\|_{L^\infty(\rr^N)}   + \eps \frac{Mt}{R^2}+
\int _0^t \int_{\rr^N} u^q|\partial_{x_N}\rho_R|.
\end{split}
\end{equation*}
\end{enumerate}
\end{lemma}

\begin{lemma}\label{lem.StabilityOfSolutionsOfReg}
Assume~\eqref{u_0as},~\eqref{qas},~\eqref{mlas}, and $0<\teps\leq 1$. Let $u$ be the entropy solution of~\eqref{reg} with initial data $u_0$, and $u_{\tilde \eps}$ the classical solution of the same problem with bounded initial data $u_{0,\teps}:=u_0+\teps\geq \teps$. Then,  as $\tilde{\varepsilon}\to0^+$,
$$
u_{\tilde{\varepsilon}}\to u\quad\text{in } C([0,\infty);L_{\textup{loc}}^1(\rr^N)),\quad\text{and hence }u_{\tilde{\varepsilon}}\to u\text{ a.e.~in }Q.
$$
\end{lemma}

%%%%%%%%%%%%%%%%%%%%%%%%%%%%%%%
\subsection{Parabolic estimates}
By~\eqref{eq.L2Energy} in Appendix~\ref{sect:appendix.results.entropy.solutions}, an energy estimate holds for~\eqref{reg}, associated to the parabolic term $\ml u-\eps\Delta u$. It is then standard to deduce a time decay of the $L^p$-norm.

\begin{lemma}\label{parabolic.estimates}
Assume~\eqref{u_0as},~\eqref{qas}, and~\eqref{mlas}. Then the entropy solution $u$ of~\eqref{reg} satisfies:
\begin{enumerate}[{\rm (a)}]
\item \textup{(Time decay of $L^p$-norm)} For all $t>0$,
\begin{equation}
\label{decayLinfty}
\|u(t)\|_{L^p(\rr^N)}\leq C(p,\alpha) \|u_0\|_{L^1(\rr^N)}t^{-\frac{N}{\alpha}(1-\frac 1p)}\quad\text{for all }p\in[1,\infty].
\end{equation}
\item \textup{(Energy estimate)} For all $0<\tau<T<\infty$,
\begin{equation*}
\Lambda_1\int_\tau^T \int _{\rr^N} |(-\Delta)^{\alpha/4}u|^2 +\varepsilon \int_\tau^T \int _{\rr^N} |\nabla u|^2\leq\frac{1}{2}\|u(\tau)\|^2_{L^2(\rr^N)}.
\end{equation*}
\end{enumerate}
\end{lemma}

%%%%%%%%%%%
\subsection{Hyperbolic estimates}
In the subcritical case, when the limit equation is hyperbolic, we prove that the previous decay of the solutions can be improved by using an Ole\u{\i}nik-type estimate in the direction~$x_N$ (we refer to \cite{Ole63} for the classical result by Ole\u{\i}nik and to \cite{Hof83} for its sharp version established by Hoff; see Remark \ref{rk:smoothing}(b) for further information).  By Lemma~\ref{lem.StabilityOfSolutionsOfReg}, we can first treat the case of strictly positive and bounded solutions. For simplicity, we write $u_0,u$ for $u_{0,\teps}, u_{\teps}$ in what follows (we also did this in Lemma \ref{lem.ClassicalSolutionsOfReg}). 

\begin{lemma}[Hoff-type estimate {{\cite{Hof83}}}]\label{lem.OleinikForPositiveBoundedFunctions}
Assume $\teps>0$, $\teps\leq u_{0}\in L^\infty(\rr^N)$,~\eqref{qas} with $q\leq 2$, and~\eqref{mlas}. Then the classical solution $u$ of~\eqref{reg} satisfies for all $(t,x)\in Q$:
\begin{eqnarray}
\label{oleinik.reg}
\partial_{x_N}(u^{q-1})(t,x)\leq \frac 1{qt}&&\text{if }1<q\leq 2,\\
\label{oleinik.reg.2}
\partial_{x_N}(u^{1-q})(t,x)\geq -\frac {\|u_0\|^{2-2q}_{L^\infty(\rr^N)}}{qt}&&\text{if }1-\frac1N<q<1.
\end{eqnarray}
\end{lemma}

\begin{proof}
It follows the lines of~\cite[Proposition~3.4]{Ignat-Stan-2018}. Lemma~\ref{lem.ClassicalSolutionsOfReg} gives that $\teps\leq u\leq \|u_0\|_{L^\infty(\rr^N)}$ and $u\in C^\infty_\textup{b}(Q)$.

Let us denote $z=u^{\gamma}$ where $\gamma\in (0,1]$ will be chosen latter. Then
$z\in C^\infty_\textup{b}(Q)$ and
\[
\partial_tz+\gamma z^{1-\frac {1}{\gamma}}\ml (z^{\frac 1\gamma})+qz^{\frac{q-1}\gamma}\partial_{x_N}z=\eps \Big(\Delta z+\frac{1-\gamma}{\gamma}\frac{|\nabla z|^2}z\Big).
\]
Let $w=\partial_{x_N}z$ and $\beta=\frac{1-\gamma}{\gamma}\geq 0$. Then $w\in C^\infty_\textup{b}(Q)$ and
\begin{equation}\label{eq.w}
\begin{aligned}&
\partial_tw+ \frac{q(q-1)}{\gamma}z^{\frac{q-1-\gamma}{\gamma}}w^2+qz^{\frac{q-1}\gamma}\partial_{x_N}w+z^{-\frac 1\gamma}A_\beta(w,z)
\\
&\qquad=
\eps\Big( \Delta w-\frac{1-\gamma}\gamma \frac{|\nabla z|^2}{z^2}w
+\frac{2(1-\gamma)}{\gamma} \frac{ \nabla z\cdot\nabla w }{z}\Big),
\\
&\text{where }
A_\beta(w,z)=z\ml (z^\beta w)-\frac{\beta}{\beta+1}w\ml (z^{1+\beta}).
\end{aligned}
\end{equation}
Let us denote $W(t)=\sup_{x\in \rr^N} w(t,x)$. If $W(t)$ is attained at some point $\overline{x}\in \rr^N$,
one can prove that
\[
A_\beta(w(t,\overline{x}), z(t,\overline{x}))\geq w(t,\overline{x})\mathcal{I}_z(t,\overline x)
\]
for some nonnegative function $\mathcal{I}_z$. However, it may happen that for a fixed $t$ the supremum is not attained, and we follow the strategy in \cite[Theorem~2, Remark~3]{DroniouImbertFractal} and \cite[Lemma 1.17]{KarchQualitProp2009}.
We consider points $x_n$ such that
\begin{equation}
\label{eq:x_n}
W(t)\ge w(t,x_n)\ge W(t)-\frac 1n.
\end{equation}
Since $w(t)\in C^3_\textup{b}(\rr^N)$, it follows that
\[
\lim_{n\rightarrow \infty}\nabla w(t,x_n)=0, \qquad \limsup_{n\rightarrow \infty}\Delta w(t,x_n)\leq 0.
\]
Besides, $z(t,x_n)\in [ \teps^{\gamma}, \|u_0\|_{L^\infty(\rr^N)}^{\gamma}]$ for all $n\in\mathbb{N}$. Hence,   $z(t,x_n)\rightarrow p_1(t)\in  [ \teps^{\gamma}, \|u_0\|_{L^\infty(\rr^N)}^\gamma]$ up to a subsequence. Also, since $\{|\nabla z(t,x_n)|\}_{n\geq 1}$ is uniformly bounded we get  $|\nabla z(t,x_n)|\rightarrow  p_2(t)$, again up to a subsequence.
We now make a claim which is proved some paragraphs below:  once more up to a subsequence,
\begin{equation}\label{claim1}
A_\beta(w(t,x_n),z(t,x_n))\geq W(t)\mathcal{I}_n(t)-o(1)\quad\text{as }n\to\infty
\end{equation}
for some uniformly bounded nonnegative sequence $\mathcal{I}_n(t)$. We therefore also have $\mathcal{I}_n(t)\rightarrow p_3(t)\geq 0$ up to a subsequence.  Arguing as in \cite[Proposition~3.4]{Ignat-Stan-2018},  we obtain that
\[
W'(t)+\frac{q(q-1)}{\gamma}p_1^{\frac{q-1-\gamma}{\gamma}}(t)W^2(t)+p_1^{- \frac 1\gamma}(t)p_3(t) W(t)+\frac{\eps(1-\gamma)}\gamma\frac{p^2_2(t)}{p_1^2(t)}W(t)\leq 0\quad\text{for all }t>0.
\]
Looking for supersolutions of the type $\overline{W}=Ct^{-1}$ we get  for $\gamma\geq q-1>0$ that
\[
w(t,x)\leq W(t)\leq \frac{\gamma}{q-1}\frac {\|u_0\|_{L^\infty(\rr^N)}^{\gamma-(q-1)}}{qt} \quad\text{for all }(t,x)\in Q.
\]
The choice $\gamma=q-1$ then gives~\eqref{oleinik.reg}.

Let us introduce $\tilde w=-w$. Using~\eqref{eq.w},  it follows that $\tilde w$ satisfies
 \[
 \begin{gathered}
\partial_t\tilde w+ \frac{q(1-q)}{\gamma}z^{\frac{q-1-\gamma}{\gamma}}\tilde w^2+qz^{\frac{q-1}\gamma} \partial_{x_N}\tilde w+z^{-\frac 1\gamma}A_\beta(\tilde w,z)\\
=
\eps\Big( \Delta \tilde w-\frac{1-\gamma}\gamma \frac{|\nabla z|^2}{z^2}\tilde w
+\frac{2(1-\gamma)}{\gamma} \frac{ \nabla z\cdot\nabla \tilde w }{z}\Big).
\end{gathered}
\]
The same argument as above gives us, for $\gamma\in (0,1]$,  in particular $1\geq \gamma >0>q-1$, that
\[
\tilde w(t,x)\leq  \frac{\gamma}{1-q}\frac {\|u_0\|_{L^\infty(\rr^N)}^{\gamma-(q-1)}}{qt}\quad\text{for all }(t,x)\in Q.
\]
The choice $\gamma=1-q$ then gives~\eqref{oleinik.reg.2}.

It remains to prove claim~\eqref{claim1}. For simplicity, we will sometimes suppress the $t$-dependence of $w$ and $z$. The definition of $\ml$ yields
\begin{align*}
A_\beta(w(t,x),&z(t,x))=\mathcal{J}^>(t,x,R)+\mathcal{J}^<(t,x,R),\quad\text{where }\\
\mathcal{J}^>(t,x,R)&=\int_{\bs^{N-1}}\int_R^\infty  \Big[\frac{w(x)}{\beta+1}\Big(z^{\beta+1}(x)+\frac\beta2 \big(z^{\beta+1}(x+r\theta)+z^{\beta+1}(x-r\theta)\big) \Big)
\\
&\quad-\frac{z(x)}2\big(w(x+r\theta)z^\beta(x+r\theta)+w(x-r\theta)z^\beta(x-r\theta)\big)\Big]\, \frac{{\rm d}r}{r^{1+\alpha}}{\rm d}\mu(\theta),\\
\mathcal{J}^<(t,x,R)
&=z(x)\int_{\bs^{N-1}}\int_0^R \Big(z^\beta(x) w(x)-\frac{z^\beta(x+r\theta) w(x+r\theta)+z^\beta(x-r\theta) w(x-r\theta)}2 \Big)\frac{{\rm d}r}{r^{1+\alpha}}
{\rm d}\mu(\theta)\\
&\quad-\frac{\beta}{\beta+1} w(x) \int_{\bs^{N-1}}\int_0^\infty\Big(z^{\beta+1}(x)-\frac{z^{\beta+1}(x+r\theta)+z^{\beta+1}(x-r\theta)}2 \Big)\frac{{\rm d}r}{r^{1+\alpha}}
{\rm d}\mu(\theta).
\end{align*}
Note   that terms were grouped in a different way in $\mathcal{J}^>$ and $\mathcal{J}^<$. The term $\mathcal{J}^<$ can be easily estimated as
\begin{align*}
|\mathcal{J}^<&(t,x,R)|\\
&\leq \big(\|z\|_{L^\infty(\mathbb{R}^N)}\|D^2(z^\beta w)\|_{L^\infty(\mathbb{R}^N)} + \|w\|_{L^\infty(\rr^N)}\| D^2 (z^{\beta+1})\|_{L^\infty(\mathbb{R}^N)}\big)
 \int_{\bs^{N-1}}\int_0^R\frac{r^2|\theta |^2\,{\rm d}r}{r^{1+\alpha}}{\rm d}\mu(\theta)\\
&\leq C(\teps,\|z\|_{C^3_{\textup{b}}(\rr^N)})R^{2-\alpha}.
\end{align*}
On the other hand, using~\eqref{eq:x_n},  we get

\begin{align*}
\mathcal{J}^>(t,x_n,R)&\ge  W(t) \mathcal{I}(t,x_n,R)+\mathcal{K}(t,x_n,R),\quad\text{where }\\
\mathcal{I}(t,x_n,R)&=
\int_{\bs^{N-1}}\int_R^\infty \Big[\frac{1}{\beta+1} \Big(z^{\beta+1}(t,x_n)+\frac\beta2 (z^{\beta+1}(t,x_n+r\theta)+z^{\beta+1}(t,x_n-r\theta)\Big)\\
&\quad -z(t,x_n)\frac{z^{\beta}(t,x+r\theta)+z^{\beta}(t,x-r\theta) }2 \Big]\, \frac{{\rm d}r}{r^{1+\alpha}}{\rm d}\mu(\theta),\\
\mathcal{K}(t,x_n,R)&=- \frac 1{n(\beta+1)}  \int_{\bs^{N-1}}\int_R^\infty  \Big(z^{\beta+1}(t,x_n)\\
&\quad+\frac\beta2 (z^{\beta+1}(t,x_n+r\theta)+z^{\beta+1}(t,x_n-r\theta) \Big)\, \frac{{\rm d}r}{r^{1+\alpha}}{\rm d}\mu(\theta).
\end{align*}
By H\"older's inequality,  $\mathcal{I}$ is  nonnegative.   Moreover, \cite[Lemma~3.5]{Ignat-Stan-2018} shows that it is uniformly bounded: $|\mathcal{I}(t,x_n,R)|\leq C(\alpha,\beta,\tilde\eps, \|z(t)\|_{L^\infty(\rr^N)})\|z(t)\|_{C^1_{\rm b}(\rr^N)}^2$. Finally, an easy computation shows that $|\mathcal{K}(t,x_n,R)|\le \frac{C}{nR^\alpha}\|z(t)\|_{L^\infty(\mathbb{R}^N)}^{\beta+1}$. Combining everything,
$$
A_\beta(w(t,x_n),z(t,x_n))\ge W(t) \mathcal{I}(t,x_n,R)-\frac{C}{nR^\alpha}-CR^{2-\alpha}.
$$

Choosing $R=R_n\rightarrow 0$ such that  $nR_n^\alpha\rightarrow \infty$ we obtain~\eqref{claim1} with $\mathcal{I}_n(t)=\mathcal{I}(t,x_n,R_n)$.
\end{proof}

To continue, we need a lemma whose proof can be found in Appendix~\ref{sec:PrimitivesOfEntropySolutions}.
\begin{lemma}\label{lem.PrimitiveSolvesParabolic}
Assume~\eqref{u_0as},~\eqref{qas}, and~\eqref{mlas}. If $u$ is the entropy solution of~\eqref{reg}, then the function $v(t,x'):=\int_{\rr}u(t,x',x_N)\,{\rm d}x_N$ satisfies $v\in C([0,\infty);L^1(\rr^{N-1}))$ and solves
\begin{equation}\label{restricted.epsilon}
\partial_tv+\widetilde \ml v-\eps \Delta_{x'}v=0 \quad\textup{in }(0,\infty)\times \rr^{N-1},\qquad v(0,\cdot)=\int_{\rr}u_0(\cdot,x_N)\,{\rm d}x_N\quad\textup{in }\rr^{N-1},
\end{equation}
in the very weak sense, where $\widetilde \ml$ is the operator introduced in~\eqref{tildeL}.
\end{lemma}

We are then ready to present and prove the hyperbolic estimates.
\begin{lemma}\label{estimari.hiperbolice}
Assume~\eqref{u_0as},~\eqref{qas} with $q\leq 2$, and~\eqref{mlas}.  There exists a positive constant $C(q,\alpha,N)$ such that the entropy solution $u$ of~\eqref{reg} satisfies:
\begin{enumerate}[{\rm (a)}]
\item \textup{(Time decay of $L^\infty$-norm)}
\begin{equation}
\label{linfty.hyp}
\|u(t)\|_{L^\infty(\rr^N)}\leq C(q,\alpha,N)\|u_0\|_{L^1(\rr^N)}^{\frac{1}{q}} t^{-\frac{1}q(1+\frac{N-1}\alpha)}\quad\text{for all }t>0.
\end{equation}
\item \textup{(Bounds on the $x_N$-derivative)} The following inequalities hold in $\mathcal{D}'(\rr^N)$ for all $t>0$:
\begin{eqnarray}
\label{upper.ux.q>1}
\partial_{x_N}u(t)\leq C(q,\alpha,N)\|u_0\|_{L^1(\rr^N)}^{\frac {2-q}q} t^{-1-(1+\frac{N-1}\alpha)(\frac {2-q}q)}  &&\text{if }1<q\leq 2,\\
\label{upper.ux.q<1}
\partial_{x_N}u(t)\geq -C(q,\alpha,N) \|u_0\|_{L^1(\rr^N)}^{\frac {2-q}q} t^{-1-(1+\frac{N-1}\alpha)(\frac {2-q}q)}&&\text{if }1-\frac1N<q<1.
\end{eqnarray}

\item \textup{(Local estimate on the $x_N$-derivative)} For all $R, R'>0$, all $|h_N|<R$, and all $t>0$,
\begin{equation}
\label{W11loc}
\begin{aligned}
&\int_{|x'|<R', |x_{N}|< R}|u(t,x+(0,h_N))-u(t,x)|\,{\rm d}x\\
&\quad\leq C(q,\alpha,N)|h_N|{R'}^{N-1}\Big(R t^{-1-(1+\frac{N-1}\alpha)(\frac {2-q}q)} \|u_0\|_{L^1(\rr^N)}^{\frac {2-q}q}+  t^{-\frac{1}q(1+\frac{N-1}\alpha)} \|u_0\|_{L^1(\rr^N)}^{\frac{1}{q}}\Big).
\end{aligned}
\end{equation}
\end{enumerate}
\end{lemma}

\begin{remark}\label{rk:smoothing}
\begin{enumerate}[{\rm (a)}]
\item  The hyperbolic estimate~\eqref{linfty.hyp} is an improvement of the parabolic estimate~\eqref{decayLinfty} if and only if $q<q_*(\alpha)$. In this respect,  it is important to notice that $q_*(\alpha)<2$ for all $\alpha\in(0,2)$, so that the hyperbolic estimate is valid in the convection regime.
\item Note that $f''(r)=q(q-1)r^{-(2-q)}$. On the  one hand, when $1<q\leq 2$, $f''(r)>0$ for all $r>0$, and on the other hand, when $1-\frac{1}{N}<q<1$, $f''(r)<0$ for all $r>0$. Hence, $f$ is so-called \emph{genuinely nonlinear} in these ranges of parameters. In the case $f''(r)>0$, the Ole\u{\i}nik-type estimate and the $L^1$--$L^\infty$-smoothing effects in Lemmas~\ref{lem.OleinikForPositiveBoundedFunctions} and~\ref{estimari.hiperbolice} fall into the well-known theory of scalar conservation laws, see e.g. \cite{Hof83} and \cite[Theorems 11.2.1 and 11.5.2]{Daf00}. The maybe less known case $f''(r)<0$ can be found in \cite[Lemma 2.6]{LaurencotFast}. See also the recent paper  \cite{SeSi19} for a novel multidimensional analogue.
\item By passing to the limit $\varepsilon\to0^+$ we can check that the $L^1$--$L^\infty$ smoothing estimate~\eqref{linfty.hyp} is still valid for our original problem~\eqref{eq.main}.
\end{enumerate}
\end{remark}

\begin{proof}[Proof of Lemma~\ref{estimari.hiperbolice}] (i) \emph{The case $1<q\leq 2$.}  We approximate $u_0$ with smooth functions $u_{0,\teps}$ such that $\teps\le u_{0,\teps}\le 2\|u_0\|_{L^\infty(\rr^N)}$ and use the estimates in Lemma~\ref{lem.OleinikForPositiveBoundedFunctions} for the corresponding classical solution~$u_{\teps}$ of~\eqref{reg}. By Lemma~\ref{lem.StabilityOfSolutionsOfReg}, we may pass to the limit $\teps\to 0$ to obtain that
\begin{equation}
	\label{oleinik.10}
\partial_{x_N}(u^{q-1})(t)\leq \frac 1{qt} \quad\textup{in }\mathcal{D'}(\rr^N)\textup{ for all }t>0.
\end{equation}
Now, consider $v(t,x')=\int_{\rr}u(t,x',x_N)\,{\rm d}x_N$ from Lemma~\ref{lem.PrimitiveSolvesParabolic}.
 Using the nondegeneracy condition on the new measure $\tilde\mu$ (corresponding to $\widetilde \ml$),  the results of Lemma~\ref{parabolic.estimates} still hold for~\eqref{restricted.epsilon}.
We then obtain that $v$ satisfies the $L^1$--$L^\infty$-smoothing
  \begin{equation}
  	\label{est.v.new}
  \|v(t)\|_{L^\infty(\rr^{N-1})}\leq c(\alpha)  t^{-\frac{N-1}\alpha}\|v_0\|_{L^1(\rr^{N-1})}=c(\alpha) t^{-\frac{N-1}\alpha}\|u_0\|_{L^1(\rr^N)}.
  \end{equation}
This estimate, the decay of the $x_N$ derivative of $u$ in~\eqref{oleinik.10},    and Lemma \ref{ineq.dis.appendix} in variable $x_N$
show that for a.e. $x'\in \rr^{N-1}$ (see also   \cite[Lemma 2.2]{EVZIndiana})
  \[
  t(q-1)\|u(t,x',\cdot)\|_{L^\infty(\rr)}^q \leq \int_{\rr}u(t,x',x_N)\,{\rm d}x_N\leq\esssup _{x'\in \rr^{N-1}} v(t,x')\leq c(\alpha) t^{-\frac{N-1}\alpha}\|u_0\|_{L^1(\rr^N)}
  \]
which gives us~\eqref{linfty.hyp}.

Let us now prove~\eqref{upper.ux.q>1}. Even though the argument may be classical, we prefer to add a few lines here. Lemma~\ref{lem.ClassicalSolutionsOfReg} guarantees that $\teps \leq  u_{\teps}  \leq 2 \|u_0\|_{L^\infty(\rr^N)} $, and then
\[
    \partial_{x_N}(u_{\teps}(t,x))=\frac 1{q-1}u_{\teps}^{2-q}(t,x)\partial_{x_N}\big(u_{\teps}^{q-1}(t,x)\big)
    \leq \frac{u_{\teps}^{2-q}(t,x)}{q(q-1)t}.
\]
  Letting $\teps\rightarrow 0$, using that $u_{\teps}(t)\rightarrow u(t)$ a.e.~$x\in \rr^N$ (cf.~Lemma~\ref{lem.StabilityOfSolutionsOfReg}) and the $L^1$--$L^\infty$-smoothing obtained before, we get that
\begin{equation}\label{eq.EstimateOnThex_NDerivativeL1}
  \partial_{x_N}u(t)\leq \frac{{u}^{2-q}(t)}{q(q-1)t}\leq C(q,\alpha,N)\|u_0\|_{L^1(\rr^N)}^{\frac {2-q}q} t^{-1-(1+\frac{N-1}\alpha)(\frac {2-q}q)}  \quad\textup{in }\mathcal{D}'(\rr^N)\textup{ for all }t>0.
\end{equation}

It remains to prove~\eqref{W11loc}. By Lemma~\ref{lem.EntropySolutionsOfReg}, $u\in L^2((0,\infty);H^1(\rr^N))$,
and then $u(t)\in W^{1,1}_{\textup{loc}}(\rr^N)$ for a.e.~$t>0$.  Hence,~\eqref{eq.EstimateOnThex_NDerivativeL1} holds a.e.~in $\rr^N$ for a.e $t>0$. This gives us that
  \begin{align*}
  \int_{|x'|<R',|x_N|<R}&|\partial_{x_N}u(t)|\\
  &= \int_{|x'|<R'} \Big(2\int_{(-R,R),\partial_{x_N}u>0}\partial_{x_N}u(t,x',x_N)\,{\rm d}x_N -\int_{-R}^R\partial_{x_N}u(t,x',x_N)\,{\rm d}x_N\Big){\rm d}x'\\
  &=  \int_{|x'|<R'} \Big( 2\int_{(-R,R),\partial_{x_N}u>0}\partial_{x_N}u(t,x',x_N)\,{\rm d}x_N+
  u(t,x',-R)-u(t,x',R)\Big)\,{\rm d}x'\\
  &\lesssim {R'}^{N-1}\Big(R \|u_0\|_{L^1(\rr^N)}^{\frac {2-q}q} t^{-1-(1+\frac{N-1}\alpha)(\frac {2-q}q)} +  \|u(t)\|_{L^\infty(\rr^N)}\Big)\\
  &\lesssim {R'}^{N-1}\Big(R \|u_0\|_{L^1(\rr^N)}^{\frac {2-q}q} t^{-1-(1+\frac{N-1}\alpha)(\frac {2-q}q)} +   t^{-\frac{1}q(1+\frac{N-1}\alpha)} \|u_0\|_{L^1(\rr^N)}^{\frac{1}{q}}\Big).
  \end{align*}
  Since for $|h_N|<R$  we have
  \[
  \int_{|x_N|<R}|u(t,x',x_N+h_N)-u(t,x',x_N)|\,{\rm d}x_N\leq |h_N|\int_{|x_N|<2R}|\partial_{x_N}u(t,x',x_N)|\,{\rm d}x_N,
  \]
 we obtain the desired estimate  for a.e.~$t>0$. Since $u\in C([0,\infty);L^1(\rr^N))$,  estimate~\eqref{W11loc} holds for all $t>0$.

\smallskip
\noindent (ii) \emph{The case $1-\frac1N<q< 1$.} We proceed as in \cite[Lemma 2.6]{LaurencotFast}. We may assume that $u_0\not\equiv0$; otherwise the result is trivial. We fix $\tau>0$. Since $u_0$  is a bounded nonnegative function, $u(\tau)$ is also a bounded nonnegative function.   We consider a sequence of initial data   $w_{0,k}=u(\tau)+1/k$, with $k$ large enough such that  $1/k\leq w_{0,k}\leq 2\|u(\tau)\|_{L^\infty(\rr^N)}$ (notice that $\|u(\tau)\|_{L^\infty(\rr^N)}>0$, since $u_0$ is nontrivial). Then the corresponding  solutions $w_k$ with these initial data satisfy  $1/k\leq w_k(t,x)\leq 2\|u(\tau)\|_{L^\infty(\rr^N)}$ (cf.~Lemma~\ref{lem.ClassicalSolutionsOfReg})
and by Lemma~\ref{lem.OleinikForPositiveBoundedFunctions},
\begin{equation}
	\label{est.wk}
\partial_{x_N}(w_k^{1-q})(t,x)\geq -\frac{(2\|u(\tau)\|_{L^\infty(\rr^N)})^{2-2q}}{qt}\quad\text{for all }(t,x)\in Q.
\end{equation}
We now use Lemma~\ref{lem.StabilityOfSolutionsOfReg} to let $k\rightarrow \infty$. It follows that $w_k(t)\rightarrow u(\tau+t)$ a.e.~in $\rr^N$, so that
\[
\partial_{x_N}(u^{1-q})(t)\geq -\frac{(2\|u(\tau)\|_{L^\infty(\rr^N)})^{2-2q}}{q(t-\tau)}\quad\text{in } \mathcal{D}'(\rr^N)\text{ for all } t>\tau.
\]

In the next step, we show that there are convergent sequences $\{\alpha_n\}_{n\in\mathbb{N}}$, $\{\gamma_n\}_{n\in\mathbb{N}}$ such that
\[
\|u(t)\|_{L^\infty(\rr^N)}\leq \gamma_n t^{-\alpha_n} \quad \text{for all }t>0.
\]
We proceed by induction.  Using Lemma~\ref{parabolic.estimates},  we can choose  $\gamma_0=c(\alpha)\|u_0\|_{L^1(\rr^N)}$ and $\alpha_0=-\frac{N}{\alpha}$.

Denoting again $v(t,x')=\int_{\rr}u(t,x',x_N)\,{\rm d}x_N$, we obtain that $v$ satisfies~\eqref{est.v.new}.  Using Lemma~\ref{ineq.dis.appendix}  in the variable $x_N$, the parabolic estimate~\eqref{est.v.new}, and the induction hypotheses we get, for a.e.~$x'\in \rr^{N-1}$,
\begin{align*}
\|u(t,x',\cdot)\|_{L^\infty(\rr)}&\leq \Big ( \frac{2-q}{1-q}\|u(t,x',\cdot)\|_{L^1(\rr)} \frac{(2\|u(\tau)\|_{L^\infty(\rr^N)})^{2-2q}}{q(t-\tau)} \Big)^{\frac 1{2-q}}\\
&\leq \Big ( \frac{2^{2-2q}(2-q)}{(1-q)q}\|v(t)\|_{L^\infty(\rr^{N-1})} \frac{\|u(\tau)\|_{L^\infty(\rr^N)}^{2-2q}}{ (t-\tau)} \Big)^{\frac 1{2-q}}\\
&\leq \Big ( \frac{2^{2-2q}(2-q)}{(1-q)q}c( \alpha) t^{-\frac{N-1}\alpha}\|u_0\|_{L^1(\rr^N)} \frac{( \gamma_n\tau^{-\alpha_n})^{2-2q}}{(t-\tau)} \Big)^{\frac 1{2-q}}.
\end{align*}
Choosing $\tau=t/2$,  we can define
\[
\begin{aligned}
\alpha_{n+1}&=\alpha_n\frac{2-2q}{2-q}+\frac{1}{2-q}\Big(\frac{N-1}\alpha+1\Big),\\
\gamma_{n+1}&=k \gamma_n^{\frac{2-2q}{2-q}}2^{\alpha_{n+1}},\quad k=c(q,\alpha,N )\|u_0\|_{L^1(\rr^N)}^{\frac 1{2-q}}.
\end{aligned}
\]
It follows that $\alpha_n\rightarrow \frac 1q(1+\frac{N-1}\alpha)$ and $\gamma_n\rightarrow \|u_0\|_{L^1(\rr^N)}^{1/q}C(q,\alpha,N )$, which gives us~\eqref{linfty.hyp}.

Let us now obtain the estimate~\eqref{upper.ux.q<1}. By~\eqref{est.wk},
\[
\partial_{x_N}w_k(t,x)=\frac{w_k^q \partial_{x_N} (w_k^{1-q})}{(1-q)}(t,x)\geq -\frac{w_k^q(t,x)}{1-q}\frac{(2\|u(\tau)\|_{L^\infty(\rr^N)} )^{2-2q}}{qt}\geq -\frac{(2\|u(\tau)\|_{L^\infty(\rr^N)} )^{2-q}}{(1-q)qt}.
\]
Letting $k\rightarrow\infty$, choosing $\tau=t/2$,  and using the decay estimate~\eqref{linfty.hyp},  we arrive to
the desired estimate.

Finally, estimate~\eqref{W11loc} follows exactly as in the case $q>1$, since by Lemma~\ref{lem.EntropySolutionsOfReg} the solution of~\eqref{reg} satisfies  $u(t)\in W^{1,1}_{\rm loc}(\rr^N)$.
\end{proof}

%%%%%%%%%%%%%%%%%%%%%%%%%%%%%%%%%%%%%%%%%%%%%%%%%%%%%%%%%%%%%%%%%%%%%%%%%%%%%%%%%%%%
\section{Rescaled solutions and limit equations}
\label{sect:scalings}
\setcounter{equation}{0}

We introduce the one-parameter family of functions
\[
u_\lambda(t,x',x_N):=\lambda^\gamma u(\lambda t, \lambda^{1/\alpha}x',\lambda^\beta x_N),
\]
where the exponents $\gamma,\beta>0$ are such that:
\begin{enumerate}[{\rm (i)}]
\item The mass of the solutions remains constant.
\item The decay in the $L^\infty$-norm remains independent of $\lambda$.
\end{enumerate}
The conservation of mass imposes
\[
\gamma=\frac{N-1}\alpha + \beta.
\]
According to~Lemmas~\ref{parabolic.estimates} and~\ref{estimari.hiperbolice}, in order to keep the best decay estimate of the $L^\infty$-norm we set
\[
\beta=
\begin{cases}
\frac 1\alpha,& q\geq q_*(\alpha),\\
\frac 1q \big(1+\frac {N-1}{\alpha}\big)-\frac {N-1}{\alpha},&1-\frac{1}{N}<q <q_*(\alpha).
\end{cases}
\]
The large time behaviour for $u$ will follow from the behaviour as $\lambda\to\infty$ of $u_\lambda$ for some fixed time $t$. Hence, we divide the proof of Theorem~\ref{asymptotic}, into two cases, depending on whether or not $q\geq q_*(\alpha)$.

%%%%%%%%%%%
\subsection{The case $q\geq q_*(\alpha)$}
We need uniform estimates for the rescaled solutions $\{u_\lambda\}_{\lambda>0}$. Note that they are entropy solutions of
\begin{equation}\label{eq.lambda.1}
 \partial_t u_\lambda +\mathcal{L} u_\lambda  =-\lambda^{-c(q,\alpha,N)}\partial_{x_N}(u_\lambda^q)
\quad \text{in } Q,\qquad u_\lambda(0)=u_{0,\lambda}\quad \textup{in } \rr^N,
\end{equation}
where $c(q,\alpha,N)=\frac N\alpha(q-q_*(\alpha))\geq 0$ and $u_{0,\lambda}(x):=\lambda^{N/\alpha} u_0(\lambda ^{1/\alpha }x)$.

\begin{proposition}\label{uniform.supercritical}
Assume~\eqref{u_0as},~\eqref{qas} with $q\geq q_*(\alpha)$, and~\eqref{mlas}. Then the entropy solution $u_\lambda$ of~\eqref{eq.lambda.1} satisfies:
\begin{enumerate}[{\rm (a)}]
\item \textup{(Time decay of $L^p$-norm)} For all $p\in[1,\infty]$,
$$
\|u_\lambda(t)\|_{L^p(\rr^N)}\lesssim   \|u_0\|_{L^1(\rr^N)} t^{-\frac N\alpha (1-\frac 1p)} \quad\text{for a.e }t>0.
$$
\item \textup{(Energy estimate)} For a.e. $0<\tau<T<\infty$,
\begin{equation*}
\int_\tau^T \int _{\rr^N} |(-\Delta)^{\alpha/4}u_\lambda(t)|^2 \lesssim
\tau^{-\frac{N}{\alpha}}\|u_0\|^2_{L^1(\rr^N)}.
\end{equation*}
\item \textup{(Conservation of mass)} For all $t>0$,
\[\int_{\rr^N}u_\lambda(t)=\int_{\rr^N}u_0.
\]
\item \textup{(Estimate on the time derivative)} For all smooth bounded domains $\Omega\subset\rr^N$ and a.e $0<\tau<T<\infty$,
\[
\|\partial_tu_\lambda\|_{L^2((\tau,T),H^{-1}(\Omega))}\lesssim C(\tau, \|u_0\|_{L^1(\rr^N)})\quad\text{for all }\lambda\ge1.
\]
\item \textup{(Tail control)} For all $R>0$, all $t>0$,  and all $\lambda\ge1$,
\begin{equation}
\label{tail.control.1}
\begin{aligned}
\int_{|x|>2R}u_\lambda(t)&\lesssim \int_{|x|> \lambda^{1/\alpha} R}u_0+C(\|u_0\|_{L^1(\rr^N)},\|u_0\|_{L^\infty(\rr^N)} ) \big(
tR^{-\alpha} +r_{q,\alpha,N}(t,\lambda)R^{-b(q,N)}\big),\\
r_{q,\alpha,N}(t,\lambda)&:=
\begin{cases}
	\lambda^{-1/\alpha}, & q>1+\frac{\alpha}N,\\
		\lambda^{-1/\alpha}  \log(1+t\lambda), &q=1+\frac{\alpha}N,\\
		\lambda^{-1/\alpha} + t^{1-\frac N\alpha(q-1)}\lambda ^{-c(q,\alpha,N)},&
		1< q<1+\frac{\alpha}N,\\
	\lambda^{-c(q,\alpha,N)}t, &q<1,
\end{cases}
\\
b(q, N)&:=
\begin{cases}
	1, & q>1,\\
	1-N(1-q), &q<1.
\end{cases}
\end{aligned}
\end{equation}
\end{enumerate}
\end{proposition}
\begin{remark}
Under our assumptions the exponent $b$ is always positive.
\end{remark}

The proof relies on considering the regularized problem
\begin{equation}\label{eq.lambda.1.eps}
 \partial_t u_\lambda^\eps +\mathcal{L} u_\lambda^\eps-\eps \Delta u_\lambda^\eps  =-\lambda^{-c(q,\alpha,N)}\partial_{x_N}((u_\lambda^\eps)^q)
\quad \text{in } Q,\qquad u_\lambda(0)=u_{0,\lambda}\quad \textup{in } \rr^N,
\end{equation}
and then transferring properties from $u_\lambda^\eps$ to $u_\lambda$.  We then need the following lemma, proved in Appendix~\ref{sect:appendix.results.entropy.solutions}, Section~\ref{sec.Appendix.ConvInCL1}:

\begin{lemma}\label{lem.StabilityOfSolutionsOfReg2}
Assume~\eqref{u_0as},~\eqref{qas} with $q\geq q_*(\alpha)$, and~\eqref{mlas}. Let $u_\lambda^\eps$ be the entropy solution of~\eqref{eq.lambda.1.eps} and $u_{\lambda}$ the entropy solution of~\eqref{eq.lambda.1}. Then,
$$
u_{\lambda}^\eps\to u_\lambda\quad\text{in } C([0,\infty);L^1(\rr^N))\text{ as }
\varepsilon\to0^+.
$$
\end{lemma}

\begin{proof}[Proof of Proposition~\ref{uniform.supercritical}]
Note that $u_\lambda^\eps$ satisfies the hypotheses  of Lemmas~\ref{lem.EntropySolutionsOfReg} and~\ref{parabolic.estimates}. The convergence given in Lemma~\ref{lem.StabilityOfSolutionsOfReg2} then ensures that all the estimates for  $u_\lambda^\eps$ transfer to $u_\lambda$. We thus focus on obtaining them for $u_\lambda^\eps$.

\smallskip
\noindent(a) This is exactly Lemma~\ref{parabolic.estimates}(a).

\smallskip
\noindent(b) The energy estimate in Lemma~\ref{parabolic.estimates}(b) and the $L^1$--$L^2$-smoothing deduced in (a) gives the result.

\smallskip
\noindent(c) A consequence of Lemma~\ref{lem.EntropySolutionsOfReg}(b).

\smallskip
\noindent(d) Let us consider a function $\phi\in C_\textup{c}^\infty(\rr^N)$ supported in $\Omega$. Let us choose $p\geq 2\geq p'\geq 1$ such that $pq\geq 1$. By Lemma~\ref{lem.EntropySolutionsOfReg}(c), we know that, for a.e.~$t>0$, $\partial_t u_\lambda^\eps(t) \in H^{-1}(\Omega)$ and
\begin{equation*}
\begin{split}
|\langle\partial_t u_\lambda^\eps(t) ,&\phi \rangle_{H^{-1}\times H^1}|= \bigg |\int _{\rr^N} \Big( \lambda^{-c(q,\alpha,N)}(u_\lambda^\eps(t))^q\partial_{x_N}\phi   -\varepsilon \nabla u_\lambda^\eps(t)\cdot \nabla \phi  \Big) + \mathcal{E}(u_\lambda^\eps(t), \phi) \bigg |\\
&\leq \lambda^{-c(q,\alpha,N)}\| (u_\lambda^\eps(t))^q\|_{L^p} \|\partial_{x_N}\phi\|_{L^{p'}}+\eps \|\nabla u_\lambda^\eps(t)\|_{L^2}\| \nabla\phi\|_{L^2}+\mathcal{E}(u_\lambda^\eps(t),u_\lambda^\eps(t))^{1/2}\mathcal{E}(\phi,\phi)^{1/2}\\
&\leq \lambda^{-c(q,\alpha,N)}\| u_\lambda^\eps(t)\|_{L^{pq}}^q \|\nabla \phi\|_{L^{p'}}+\eps \|\nabla u_\lambda^\eps(t)\|_{L^2}\|\phi\|_{H^1}+\mathcal{E}(u_\lambda^\eps(t),u_\lambda^\eps(t))^{1/2} \|\phi\|_{H^{\alpha/2}}\\
&\lesssim \Big(\lambda^{-c(q,\alpha,N)}\| u_\lambda^\eps(t)\|_{L^{pq}(\rr^N)}^q+\eps \|\nabla u_\lambda^\eps(t)\|_{L^2(\rr^N)}+\|(-\Delta)^{\alpha/4}u_\lambda^\eps(t)\|_{L^2(\rr^N)}\Big) \|\phi\|_{H^{1}(\Omega)}.
\end{split}
\end{equation*}
That is,
\begin{equation*}
\begin{split}
\|\partial_tu_\lambda^\eps(t)\|_{H^{-1}(\Omega)}&=\sup_{\|\phi\|_{H^{1}(\Omega)}\leq 1}|\langle\partial_t u_\lambda^\eps(t) ,\phi \rangle_{H^{-1}\times H^1}|\\
&\lesssim   \lambda^{-c(q,\alpha,N)}\| u_\lambda^\eps(t)\|_{L^{pq}(\rr^N)}^{q}+\eps \|\nabla u_\lambda^\eps(t)\|_{L^2(\rr^N)}+\|(-\Delta)^{\alpha/4}u_\lambda^\eps(t)\|_{L^2(\rr^N)}.
\end{split}
\end{equation*}

Now, integrating in the time variable and using the  properties for $u_\lambda^\eps$  in Lemma~\ref{parabolic.estimates}(b) give us that the following holds uniformly in $\varepsilon\in (0,1)$, for all $\lambda>1$:
\begin{align*}
\int_\tau^T \|\partial_tu_\lambda^\eps(t)\|^2_{H^{-1}(\Omega)}\,{\rm d}t&\lesssim \int_\tau^T \Big(\| (-\Delta)^{\alpha/4}u_\lambda^\eps(t)\|^2_{L^2(\rr^N)}+\eps^2 \|\nabla u_\lambda^\eps(t)\|^2_{L^2(\rr^N)}\Big)\,{\rm d}t \\
&\quad+\int_{\tau}^T\| u_\lambda^\eps(t)\|_{L^{pq}(\rr^N)}^{2q}\,{\rm d}t\lesssim C(\tau,\|u_0\|_{L^1(\rr^N)}).
\end{align*}

\smallskip
\noindent(e)  We continue the estimation of Lemma~\ref{lem.EntropySolutionsOfReg}(d):
\begin{align*}
%\label{eq:tail.control.lambda}
&\int_{|x|>2R} u_\lambda^\eps(t)\lesssim \int_{|x|>\lambda^{1/\alpha} R} u_{0}+M t\|\ml\rho_R\|_{L^\infty(\rr^N)}  +\eps \frac{tM}{R^2} +\lambda^{-c(q,\alpha,N)} \int_0^t \int_{\rr^N}(u_\lambda^\eps)^q|\partial_{x_N}\rho_R|.
\end{align*}
The $\alpha$-homogeneity of $\ml$  gives $\|\ml \rho_R\|_{L^\infty(\rr^N)}\lesssim R^{-\alpha}$.  It remains to estimate the nonlinear term.

If $q>1$, the conservation of mass and the maximum principle yield
$$
\begin{aligned}
\|u_\lambda^\eps(t)\|_{L^q(\rr^N)}^q&\le \|u_\lambda^\eps(t)\|_{L^1(\rr^N)}\|u_\lambda^\eps(t)\|_{L^\infty(\rr^N)}^{q-1}
\le \|u_\lambda^\eps(0)\|_{L^1(\rr^N)}\|u_\lambda^\eps(0)\|_{L^\infty(\rr^N)}^{q-1}
\\
&= M\lambda^{\frac{N}{\alpha}(q-1)}\|u_0\|_{L^\infty(\rr^N)}^{q-1},
\end{aligned}
$$
which,  combined with  the $L^1$--$L^q$-decay estimate from part (a),  gives
\begin{equation}
\label{eq:power.q}
\begin{aligned}
\|u_\lambda^\eps(t)\|_{L^q(\rr^N)}^q&\lesssim\min\{ M\lambda^{\frac{N}{\alpha}(q-1)}\|u_0\|_{L^\infty(\rr^N)}^{q-1},M^q t^{-\frac{N}\alpha(q-1)}\}\\
&\lesssim C(M,\|u_0\|_{L^\infty(\rr^N)})\bigg(\frac{ \lambda}{1+\lambda t}\bigg)^{\frac{N}\alpha(q-1)},
\end{aligned}
\end{equation}
so that
$$
\begin{aligned}
\int_0^t\int_{\rr^N}(u_\lambda^\eps)^q|\partial_{x_N}\rho_R|&\lesssim\frac1R\int_0^t\|u_\lambda^\eps(s)\|_{L^q(\rr^N)}^q\,{\rm d}s\lesssim \frac{C(M,\|u_0\|_{L^\infty(\rr^N)})\lambda^{\frac N\alpha(q-1)-1} }{R} \int_0^{\lambda t} \frac{{\rm d}s}{(1+ s)^{\frac{N}\alpha(q-1)}}\\
&\lesssim\frac{C(M,\|u_0\|_{L^\infty(\rr^N)})\lambda^{\frac N\alpha(q-1)-1} }{R}\begin{cases}
		1, & q>1+\frac{\alpha}N,\\
		 \log(1+t\lambda), &q=1+\frac{\alpha}N,\\
		(1 +\lambda t)^{1- \frac{N}\alpha(q-1)},&
		q<1+\frac{\alpha}N,
	\end{cases}
\\&=\lambda^{c(q,\alpha,N)}C(M,\|u_0\|_{L^\infty(\rr^N)})r_{q,\alpha,N}(t,\lambda)R^{-b(q,N)}.
\end{aligned}
$$

If $q<1$,   H\"older's inequality with exponents $1/q$ and $(1/q)'=1/(1-q)$ yields
\[
\int_{\rr^N}(u_\lambda^\eps(s))^q|\partial_{x_N}\rho_R|\leq \|u_\lambda^\eps(s)\|_{L^1(\mathbb{R}^N)}^q \|\partial_{x_N}\rho_R\|_{L^{1/(1-q)}(\mathbb{R}^N)}\lesssim M^q R^{-b(q,N)}.
\]

Summarizing: both if $q>1$ or $q<1$,  we have
$$
\begin{aligned}
&\int_{|x|>2R} u_\lambda^\eps(t)\lesssim \int_{|x|>\lambda^{1/\alpha} R} u_{0}+ MtR^{-\alpha} +\eps tMR^{-2} +C(M,\|u_0\|_{L^\infty(\rr^N)})r_{q,\alpha,N}(t,\lambda)R^{-b(q,N)},
\end{aligned}
$$
and the result follows  letting $\eps\to0^+$.
\end{proof}

\begin{proof}[Proof of Theorem~\ref{asymptotic} when $q\ge q_*(\alpha)$] We organize it in four steps.

\noindent (i) \emph{Compactness.} Consider Theorem~\ref{Aubin-Lions-Simon} with $X=H^{\alpha}(\rr^N)$, $B=L^2(B_R)$, and $Y=H^{-1}(B_R)$ for all $R>0$. By Proposition~\ref{uniform.supercritical}(a) and (b), e.g.~Theorem 2.1 in~\cite{DTG-CVa20} gives that $X\hookrightarrow B$ is compact, and by Proposition~\ref{uniform.supercritical}(d) time translations are controlled in $Y$. We then deduce that, up to a subsequence,  $u_\lambda\rightarrow U$ in $L^2_{\rm loc}(Q)$, so in  $L^1_{\textup{loc}}(Q)$. The tail control in Proposition~\ref{uniform.supercritical}(e)  gives us that the convergence also holds in $L^1_\textup{loc}((0,\infty);L^1(\rr^N))$. Then for a.e.~$t>0$, $u_\lambda(t)\rightarrow U(t)$ in $L^1(\rr^N)$ and $u_\lambda(t,x)\rightarrow U(t,x)$ for a.e.~$(t,x)\in Q$. Moreover, the limit point $U$ inherits the properties in Proposition~\ref{uniform.supercritical}. In particular,
$$
U\in L^\infty((0,\infty);L^1(\rr^N))\cap L^\infty((\tau,\infty);L^\infty(\rr^N))\text{ for all }\tau>0,\quad \int_{\rr^N} U(t)=M\text{ for a.e. }t>0.
$$

We will later prove that the limit profile $U$ is an entropy solution of equation~\eqref{lim.1} or~\eqref{lim.2}. In fact, for~\eqref{lim.1} we only need that $U$ is a very weak solution.
The uniqueness of the very weak/entropy solution (cf. Theorem~\ref{thm.UniquenessOfLimitEquationsIandII}) shows that in fact the whole sequence $\{u_\lambda\}_{\lambda>0}$ converges to $U$, not only a subsequence. Also,  the regularity of such solutions observed in Remark~\ref{regularity.issue},  $U\in C((0,\infty); L^1(\rr^N))$, will show that  mass conservation holds for all $t>0$.
The convergence for a given $t_1>0$ of $u_\lambda(t_1)$ toward $U(t_1)$ in $L^1(\rr^N)$ shows the desired convergence in~\eqref{main.limit} for $p=1$. The general case follows by interpolating between the convergence in the $L^1$-norm and the decay in the $L^{2p}$-norm of $u$ and $U$.

\smallskip
\noindent (ii) \emph{Identification of the limit equation.} Let us now concentrate on the equation satisfied by the limit point $U$.
 In view of Remark \ref{regularity.issue}(a),  since $u_\lambda$ is an entropy solution of~\eqref{eq.lambda.1}, then: for all $k\in \rr$, all $\rho>0$, and all $0\leq \phi\in C_\textup{c}^\infty(\overline{Q})$, 
\[
\begin{split}
&\iint_Q \Big( |u_\lambda-k|\partial_t\phi + \lambda^{-c(q,\alpha,N)}\sgn(u_\lambda-k)\big(F(u_\lambda)-F(k)\big)\cdot\nabla\phi \Big)\\
&\qquad-\iint_Q \Big(\sgn(u_\lambda-k)\phi(\ml^{> \rho}u_\lambda) + |u_\lambda-k|(\ml^{\leq \rho}\phi)\Big)+\int_{\rr^N}|u_{0,\lambda}-k|\phi(0)\geq 0.
\end{split}
\]
Let us consider $\phi\in C_\textup{c}^\infty(Q)$ and $\tau$ and $T$ such that $\phi$ is supported in $[\tau,T]$.
On $(\tau,T)$,  we use the fact that $u_\lambda\rightarrow U$ in $L^1((\tau,T)\times \rr^N)$, so
\[
\int_\tau^T \int_{\rr^N} \big||u_\lambda -k| -|U-k|\big| |\partial_t\phi| \leq \int_\tau^T \int_{\rr^N}
|u_\lambda  -U| |\partial_t\phi|\rightarrow 0 \quad\text{as }\lambda\rightarrow \infty.
\]
Similar arguments work for the term $|u_\lambda-k|(\ml^{\leq \rho}\phi)$. In the case of $\ml^{>\rho}$ we use that $\| \ml^{>\rho} v \|_{L^1(\rr^N)}\lesssim \rho^{-\alpha}\|v\|_{L^1(\rr^N)}$ to obtain
\begin{align*}
&\bigg| \iint_Q  \sgn(u_\lambda-k)\phi(\ml^{> \rho}u_\lambda) -\sgn(U -k)\phi(\ml^{> \rho}U )\bigg| \\
&\quad\leq
\bigg| \int_\tau^T \int_{\rr^N}   \sgn(u_\lambda-k)\phi (\ml^{>\rho}u_\lambda -\ml^{>\rho} U)     \bigg|+ \Big| \int_\tau^T \int_{\rr^N}  ( \sgn(u_\lambda -k)-\sgn(U -k))\phi(\ml^{> \rho}U )  \Big|\\
&\quad\lesssim  \rho^{-\alpha} \int_\tau^T \|u_\lambda(t)-U(t)\|_{L^1(\rr^N)}+\int_\tau^T \int_{\rr^N}  |\sgn(u_\lambda -U) \phi(\ml^{> \rho}U )|.
\end{align*}
Using the strong convergence for the first term  and the dominated convergence theorem for the last one we obtain that the right-hand side goes to zero as $\lambda\rightarrow \infty$.

As for the nonlinear term, if  $q>q_*(\alpha)$, so that $c(q,\alpha,N)>0$, it goes to 0. Indeed, in view of~\eqref{eq:power.q},
\begin{align*}
\lambda^{-c(q,\alpha,N)}\int_0^T \int_{\rr^N}|F(u_\lambda)|&=\lambda^{-c(q,\alpha,N)}\int_0^T \|u_\lambda(t)\|_{L^q(\rr^N)}^q\,{\rm d}t\\
&\lesssim \lambda^{-c(q,\alpha,N)}
\int_0^T \Big(\frac{\lambda}{1+\lambda t}\Big)^{\frac N\alpha(q-1)}\,{\rm d}t\rightarrow 0 \quad\text{as }\lambda\to \infty
\end{align*}
when $q> 1$,   while for $q\in(q_*(\alpha),1)$, using H\"older's inequality with exponents $1/q$ and $(1/q)'=1/(1-q)$ and the conservation of mass,
\[
\lambda^{-c(q,\alpha,N)}\int_0^T \int_{\rr^N}|F(u_\lambda)\nabla \phi|\lesssim M\lambda^{-c(q,\alpha,N)}\int_0^T \int_{\rr^N} \|\nabla \phi\|_{L^{1/(1-q)}(\rr^N)}\rightarrow 0\quad\text{as }\lambda\to \infty.
\]
We consider now the case $q=q_*(\alpha)$, for which $c(q,\alpha,N)=0$.
Observe that, on $(\tau,T)$,  both $u_\lambda$ and $U$ are uniformly bounded by $\tau^{-N/\alpha}$. So
$F(u_\lambda)\rightarrow F(U)$ in $L^1((\tau, T); L^1(\rr^N))$  and
we can use the dominated convergence theorem to obtain, as $\lambda\to\infty$,
\[
\int_\tau^T \int_{\rr^N}\sgn(u_\lambda-k)\big(F(u_\lambda)-F(k)\big)\cdot\nabla\phi
\rightarrow \int_\tau^T \int_{\rr^N}\sgn(U-k)\big(F(U)-F(k)\big)\cdot\nabla\phi.
\]
We conclude that  the limit point $U$ satisfies: for all $k\in \rr$, all $\rho>0$, and all $0\leq \phi\in C_\textup{c}^\infty(Q)$,
\begin{equation}\label{eq.EntropyInequalityU}
\begin{split}
&\iint_Q \Big( |U-k|\partial_t\phi + (1-\sgn(c(q,\alpha,N)))^+\sgn(U-k)\big(F(U)-F(k)\big)\cdot\nabla\phi \Big)\\
&\qquad-\iint_Q \Big(\sgn(U-k)\phi(\ml^{> \rho}U) + |U-k|(\ml^{\leq \rho}\phi)\Big) \geq 0,
\end{split}
\end{equation}
i.e., $U$ satisfies Definition~\ref{def.entropySolution}(b).

\smallskip
\noindent(iii) \emph{Identification of the initial data.} We recall that $u_\lambda$ is not only an entropy solution but also a very weak solution, i.e.,
\[
\iint_Q  \Big( u_\lambda  \partial_t\phi +\lambda^{-c(q,\alpha,N)}u_\lambda^q\partial_{x_N}\phi
 -u_\lambda (\mathcal{L}  \phi) \Big)+\int _{\rr^N}u_{0,\lambda}\phi(0)=0\quad\text{for all }\phi \in C_{\textup{c}}^\infty(\overline{Q}).
 \]
Using test functions $\phi(t,x)=\theta(t)\psi(x)$ as in \cite[Proof of Theorem 1, page 56]{EVZArma} (i.e., $\theta$ approximates  $\mathds{1}_{[0,t]}$),    we get that, for all $\psi\in C_\textup{c}^\infty(\rr^N)$,
 \begin{align*}
\bigg|  \int_{\rr^N}u_\lambda(t)\psi&-\int_{\rr^N}u_{0,\lambda}\psi\bigg|\leq
\int_0^t\int_{\rr^N}\Big(  \lambda^{-c(q,\alpha,N)}| u_\lambda|^q |\partial_{x_N}\psi |+|u_\lambda| |\mathcal{L}  \psi|\Big).
\end{align*}
The same arguments as in the proof of the tail control in Proposition~\ref{uniform.supercritical}(e) show that
\[
\bigg|  \int_{\rr^N}u_\lambda(t)\psi -\int_{\rr^N}u_{0,\lambda}\psi\bigg|\leq
  C(\psi)(t+r_{q,\alpha,N}(t,\lambda)),
\]
where $r_{q,\alpha,N}(t,\lambda)$ is the one  in~\eqref{tail.control.1}. Letting $\lambda\rightarrow\infty$ we get
\[
    \bigg|  \int_{\rr^N}U (t)\psi-M\psi(0)\bigg|\leq C(\psi)(t+t^{1/\alpha}) \quad \text{for a.e. }t>0,
\]
 which proves Definition~\ref{def.entropySolution}(c) for all $\psi\in C_{\textup{c}}^\infty(\rr^N)$. To extend this result to all $\psi\in C_{\textup{b}}(\rr^N)$, we first obtain the tail control for  the limit $U$. Letting $\lambda\rightarrow\infty$ in
~\eqref{tail.control.1},
 \[
 \int_{|x|>2R}U(t)\lesssim \frac{t}{R^{\alpha}}+\frac{t^{1/\alpha}}{R} \quad \text{for a.e. }t >0\text{ if }R>0.
 \]
This estimate and classical arguments as in e.g.~\cite{EVZArma} and \cite[Pages 277--278]{Ignat-Stan-2018}  show that $U$ takes as initial data $M\delta_0$ in the sense of bounded measures:
\[
\esslim_{t\to 0^+}\int_{\rr^N} U(t)\psi=M\psi(0) \quad \text{for all }\psi\in C_\textup{b}(\rr^N).
\]
It follows that $U$ is an entropy solution of problem~\eqref{lim.2}.

\smallskip
\noindent(iv) \emph{Identification of $U$ as a very weak fundamental solution of~\eqref{lim.1} if $q>q_*(\alpha)$.}  Since $U\in L^\infty((\tau,\infty)\times \rr^N))$ for all $\tau>0$, we may choose $k=\pm \|U\|_ {L^\infty((\tau,\infty)\times \rr^N))}$ in~\eqref{eq.EntropyInequalityU} to obtain that $U$ satisfies~\eqref{lim.1} in the sense of distributions (we recall that $\sgn(c(q,\alpha,N))=1$ in this case), and hence also in a very weak sense. Very weak solutions $U\in L^\infty ((0,\infty);L^1(\rr^N))$ of the linear problem~\eqref{lim.1} which take $M\delta_0$ as initial data in the sense of bounded measures are in fact very weak fundamental solutions with mass $M$ of~\eqref{lim.1} in the sense of~\cite[Section 3]{dePablo-Quiros-Rodriguez-2020} (cf.~\eqref{eq:definition.fundamental.vws}),  which completes the identification of $U$.
\end{proof}

\subsection{The case $q<q_*(\alpha)$}
It follows that $u_\lambda$ is an entropy solution of
\begin{equation}
\label{rescalada.subcritico}
\partial_t u_\lambda+\ml^\lambda u_\lambda+\partial_{x_N}(u_\lambda^q)=0\quad\textup{in }Q,\qquad   	u_\lambda(0)=u_{0,\lambda}  \quad \textup{in }\rr^N,
\end{equation}
where, for all $\phi\in C_{\textup{c}}^\infty(\rr^N)$,
\begin{equation}\label{ml.lambda}
\begin{split}
(\ml^\lambda \phi)(x',x_N)
&=\int_{\bs^{N-1}}\int_0^\infty \Big( \phi(x',x_N)\\
&\qquad-\frac{\phi(x'+r\theta',x_N+r\lambda^{\frac 1\alpha-\beta}\theta_N)+\phi(x'-r\theta',x_N-r\lambda^{\frac 1{\alpha}-\beta}\theta_N)}2 \Big)\frac{{\rm d}r}{r^{1+\alpha}}
{\rm d}\mu(\theta),
\end{split}
\end{equation}
and
$u_{0,\lambda}(x',x_N):=\lambda^\gamma u_0(\lambda ^{1/\alpha }x',\lambda^\beta x_N)$.

Observe that, for any function $\phi\in C_{\textup{c}}^\infty(\rr^N)$,  we have $(\ml^\lambda \phi)(x)\rightarrow (\ml'\phi)(x)$ for any $x\in \rr^N$ when $\lambda\rightarrow\infty$. Hence, one expects the solution of system~\eqref{rescalada.subcritico} to converge  to a solution of $\partial_t u +\ml' u+\partial_{x_N}(u^q)=0$ when $\lambda\rightarrow\infty$. This will be proved carefully in what follows.

Let us take a closer look at the new operator $\ml^\lambda$.
By recalling~\eqref{eq.SymbolL}, it is easy to observe that
\[
\widehat{\ml^\lambda \phi}(\xi',\xi_N)=\Big(|\xi'|^{2}+|\lambda ^{\frac 1\alpha-\beta}\xi_N|^2\Big)^{\alpha/2}g\Big(\frac{(\xi',\lambda^{\frac 1\alpha -\beta}\xi_N)}{|(\xi',\lambda^{\frac 1\alpha -\beta}\xi_N)|}\Big)\widehat \phi(\xi',\xi_N) \quad\text{for all }\phi\in C_{\textup{c}}^\infty(\rr^N).
\]
Hence, by denoting $\mathcal{E}_\lambda(\cdot,\cdot)$ as the bilinear form associated to the energy,
\[
\mathcal{E}_\lambda(\phi,\phi):=  \int _{\rr^N}\Big( \big|(-\Delta_{x'})^{\alpha/4}\phi \big|^2+\lambda^{1-\alpha\beta} \big|(-\partial^2_{x_Nx_N})^{\alpha/4}\phi \big|^2
 \Big),
\]
it follows that $|\langle\ml^\lambda \phi, \phi\rangle_{L^2(\rr^N)}|\eqsim \mathcal{E}_\lambda(\phi,\phi) $ and
$|\langle\ml^\lambda \phi, \psi \rangle_{L^2(\rr^N)}|\lesssim \mathcal{E}^{1/2}_\lambda(\phi,\phi)\mathcal{E}^{1/2}_\lambda(\psi,\psi) $.
We emphasize that, in the considered case,
\[
\alpha\beta =-N+1+\frac {\alpha+N-1}q> 1,
\]
and thus, for $\lambda>1$, 
\[
\mathcal{E}_\lambda(\phi,\phi)\lesssim \|(-\Delta)^{\alpha/4}\phi\|_{L^2(\rr^N)}\lesssim \| \phi\|_{H^{\alpha/2}(\rr^N)}.
\]

\begin{proposition}\label{uniform.subcritical}
Assume~\eqref{u_0as},~\eqref{qas} with $q< q_*(\alpha)$, and~\eqref{mlas}. Then the entropy solution $u_\lambda$ of~\eqref{rescalada.subcritico} satisfies:
\begin{enumerate}[{\rm(a)}]
\item	\textup{(Time decay of $L^p$-norm)} For all $p\in[1,\infty]$,
$$
\|u_\lambda(t)\|_{L^p(\rr^N)}\lesssim    t^{-\frac 1q(1+\frac{N-1}\alpha)(1-\frac 1p)} \|u_0\|_{L^1(\rr^N)}^{\frac{1}{q}(1-\frac{1}{p})}\quad\text{for a.e. }t>0.
$$
\item \textup{(Energy estimate)} For all a.e. $0<\tau<T<\infty$,
\begin{equation*}
\int_\tau^T \int _{\rr^N}\Big( \big|(-\Delta_{x'})^{\alpha/4}u_\lambda(t) \big|^2+\lambda^{1-\alpha\beta} \big|(-\partial^2_{x_Nx_N})^{\alpha/4}u_\lambda (t)\big|^2
 \Big) \lesssim
\tau^{-\frac 1q(1+\frac{N-1}{\alpha})}\|u_0\|^{\frac{1}{q}}_{L^1(\rr^N)}.
\end{equation*}
\item \textup{(Conservation of mass)} For all $t>0$,
\[
\int_{\rr^N}u_\lambda(t)=\int_{\rr^N}u_0.
\]
\item \textup{(Local estimate on the $x_N$ derivative)} For all $R, R'>0$, all $|h_N|<R$, and all $t>0$,
\begin{equation*}
\label{W11loc.lambda}
\begin{split}
&\int_{|x'|<R', |x_{N}|< R}|u_\lambda(t,x+(0,h_N))-u_\lambda(t,x)|\,{\rm d}x\\
&\qquad\leq C(q,\alpha,N)|h_N|{R'}^{N-1}\Big(R t^{-1-(1+\frac{N-1}\alpha)(\frac {2-q}q)} \|u_0\|_{L^1(\rr^N)}^{\frac {2-q}q}  +  t^{-\frac{1}q(1+\frac{N-1}\alpha)} \|u_0\|_{L^1(\rr^N)}^{\frac{1}{q}}\Big).
\end{split}
\end{equation*}
\item \textup{(Estimate on the time derivative)} For all bounded domains $\Omega\subset\rr^N$ and all $\lambda\ge1$, and a.e. $0<\tau<T<\infty$,
\[
\|\partial_t u_\lambda\|_{L^2((\tau,T),H^{-1}(\Omega))}\lesssim C(\tau, \|u_0\|_{L^1(\rr^N)}).
\]
\item \textup{(Tail control)} For all $R>0$, all $t>0$,  and all $\lambda\ge1$,
\begin{equation}
\label{tail.control.2}
\int_{|x|>2R}u_\lambda(t)\lesssim \int_{|x|>\lambda^{1/\alpha} R}u_0+C(\|u_0\|_{L^1(\rr^N)},\|u_0\|_{L^\infty(\rr^N)} ) \Big(
\frac{ t}{R^\alpha} +\frac{ t^{a(\alpha,q,N)}}{R^{b(q,N)}}\Big),
\end{equation}
with $b(q,N)$ is as in~\eqref{tail.control.1} and
\[
a(\alpha,q,N):=
\begin{cases}
	1-(1+\frac{N-1}\alpha)(1-\frac 1q), & q>1,\\
	1 ,&q<1.
\end{cases}
\]
\end{enumerate}
\end{proposition}

Again, the proof relies on considering the regularized problem
\begin{equation}
\label{rescalada.subcritico.eps}
\partial_t u_\lambda^\eps+\ml^\lambda u_\lambda^\eps+\partial_{x_N}((u_\lambda^\eps)^q)=\eps \Delta u_\lambda^\eps\quad\textup{in }Q,\qquad   	u_\lambda(0)=u_{0,\lambda}  \quad \textup{in }\rr^N,
\end{equation}
and then transferring properties to $u_\lambda$. We then need the following lemma which is proved in Appendix~\ref{sect:appendix.results.entropy.solutions}, Section~\ref{sec.Appendix.ConvInCL1}:

\begin{lemma}\label{lem.StabilityOfSolutionsOfReg3}
Assume~\eqref{u_0as},~\eqref{qas} with $q< q_*(\alpha)$, and~\eqref{mlas}. Let $u_\lambda^\eps$ be the entropy solution of~\eqref{rescalada.subcritico.eps}  and $u_{\lambda}$ the entropy solution of~\eqref{rescalada.subcritico}. Then,
$$
u_{\lambda}^\eps\to u_\lambda\quad\text{in } C([0,\infty);L^1(\rr^N))\text{ as }
\varepsilon\to0^+.
$$
\end{lemma}

\begin{proof}[Proof of Proposition~\ref{uniform.subcritical}]
Note that $u_\lambda^\eps$ satisfies the hypotheses of  Lemmas~\ref{lem.EntropySolutionsOfReg},~\ref{parabolic.estimates}, and~\ref{estimari.hiperbolice}. The convergence given in Lemma~\ref{lem.StabilityOfSolutionsOfReg3} then ensures that all the estimates for  $u_\lambda^\eps$ transfer to $u_\lambda$. We thus focus on obtaining them for $u_\lambda^\eps$. Note also that the first five estimates are reduced to the case $\lambda=1$ by rescaling the inequalities according to the definition of $u_\lambda$.

\smallskip
\noindent(a) The time decay of the $L^\infty$-norm follows by Lemma~\ref{estimari.hiperbolice}(a). We then interpolate this inequality with the $L^1$-norm bound to obtain the estimate for all $p\in[1,\infty]$.

\smallskip
\noindent(b) Since $|\xi'|^{2\alpha}+|\xi_N|^{2\alpha}\lesssim (|\xi'|^2+|\xi_N|^2)^{\alpha}$, Lemma~\ref{parabolic.estimates}(b) and the $L^1$--$L^2$ decay estimate from~(a) yield

\begin{align*}
\int_\tau^T \int _{\rr^N}\Big( \big|(-\Delta_{x'})^{\alpha/2}u_\lambda^\eps  \big|^2&+ \lambda^{1-\alpha\beta}\big|(-\partial^2_{x_N x_N})^{\alpha/4}u_\lambda^\eps \big|^2
 \Big)\\
 & \lesssim
 \int_\tau^T \int _{\rr^N} |(-\Delta)^{\alpha/4}u_\lambda^\eps|^2\lesssim
 \|u_\lambda^\eps(\tau)\|^2_{L^2(\rr^N)}\lesssim \tau^{-\frac {1}{q}(1+\frac{N-1}\alpha)} \|u_0\|_{L^1(\rr^N)}^{\frac{1}{q}}.	
\end{align*}

\smallskip
\noindent(c) A consequence of Lemma~\ref{lem.EntropySolutionsOfReg}(b).

\smallskip
\noindent(d) The shift in the  $x_N$-variable is reduced to the case $\lambda=1$ by observing that
\begin{align*}
\label{}
  \int_{|x'|<R', |x_{N}|\leq R}&|u_\lambda^\eps(t,x+(0,h_N))-u_\lambda^\eps(t,x)|\,{\rm d}x\\
  & =  \int_{|x'|<R'\lambda^{1/\alpha}, |x_{N}|\leq R\lambda^\beta} |u_1^\eps (\lambda t,x+(0,h_N\lambda^\beta))-u_1^\eps (\lambda t,x)|\,{\rm d}x
\end{align*}
and applying  Lemma~\ref{estimari.hiperbolice}(c)   with $R'\lambda^{1/\alpha}$, $R\lambda^\beta$ and $h_N\lambda^\beta$ instead of $R',R$ and $h_N$.

\smallskip
\noindent(e) By following the proof of Proposition~\ref{uniform.supercritical}(d), for $\eps\in (0,1)$, 
\begin{equation*}
\begin{split}
\|\partial_tu_\lambda^\eps(t)\|^2_{H^{-1}(\Omega)}
&\lesssim   \| u_\lambda^\eps(t)\|_{L^{pq}(\rr^N)}^{2q}+\eps^2 \|\nabla u_\lambda^\eps(t)\|^2_{L^2(\rr^N)}+\mathcal{E}_\lambda(u_\lambda^\eps(t),u_\lambda^\eps(t))\\
&\lesssim\| u_\lambda^\eps(t)\|_{L^{pq}(\rr^N)}^{2q}+\eps  \|\nabla u_\lambda^\eps(t)\|^2_{L^2(\rr^N)}+\|(-\Delta)^{\alpha/4}u_\lambda^\eps\|_{L^2(\rr^N)}^2,
\end{split}
\end{equation*}
and hence, the conclusion follows as before by using Lemma~\ref{parabolic.estimates}(b).

\smallskip
\noindent(f)  We continue the estimation of Lemma~\ref{lem.EntropySolutionsOfReg}(d):  since $\beta>1/\alpha$, for all $\lambda\ge 1$,  we have
\begin{align*}
&\int_{|x|>2R}u_\lambda^\eps(t)\lesssim \int_{|x|>\lambda^{1/\alpha} R}u_0+Mt \|\ml ^\lambda\rho_R\|_{L^\infty(\rr^N)}  +\eps \frac{tM}{R^2} + \int_0^t \int_{\rr^N}(u_\lambda^\eps)^q|\partial_{x_N}\rho_R|.
\end{align*}
To estimate the term involving $\ml^\lambda$ we use its definition, see~\eqref{ml.lambda}, and also $\beta>1/\alpha$, to get, for $\lambda\ge 1$,
\begin{align}\label{est.L.lambda}
|(\ml^\lambda \rho_R)(x',x_N)|
  &\lesssim \|D^2 \rho_R\|_{L^\infty(\rr^N)}
 \int_{\bs^{N-1}} \int_0^{\bar{r}} \frac{r^2 (|\theta'|^2 +\lambda^{2(\frac{1}{\alpha}-\beta)}\theta_N^2)}{r^{1+\alpha}}\,{\rm d}r{\rm d}\mu(\theta)\\
 \nonumber&\quad+\|  \rho_R\|_{L^\infty(\rr^N)} \int_{\bs^{N-1}} \int_{\bar{r}}^\infty \frac{1}{r^{1+\alpha}}\,{\rm d}r{\rm d}\mu(\theta) \lesssim \frac 1{R^2} \bar{r}^{2-\alpha}+\bar{r}^{-\alpha}\eqsim R^{-\alpha},
\end{align}
where we chose $\bar{r}=R$. When $1<q<q_*(\alpha)$, we use part (a) to show that the last term satisfies
\begin{align*}
   \int_0^t \int_{\rr^N}(u_\lambda^\eps)^q|\partial_{x_N}\rho_R|\lesssim \frac{1}{R} \int_0^t \|u_\lambda^\eps(s)\|_{L^q(\rr^N)}^q\,{\rm d}s\eqsim \frac{t^{a(\alpha,q,N)}}R.
\end{align*}
When $1-\frac 1N<q<1$, we proceed as in the proof of Proposition~\ref{uniform.supercritical}(e):
\begin{align*}
\int_{\rr^N}(u_\lambda^\eps(s))^q|\partial_{x_N}\rho_R|\leq   \|u_\lambda^\eps(s)\|_{L^1(\mathbb{R}^N)}^q \|\partial_{x_N}\rho_R\|_{L^{1/(1-q)}(\mathbb{R}^N)}= M^q R^{-b(q,N)} \|\partial_{x_N}\rho \|_{L^{1/(1-q)}(\mathbb{R}^N)},
\end{align*}
which yields the desired estimate.
\end{proof}

\begin{proof}[Proof of Theorem~\ref{asymptotic} when $q<q_*(\alpha)$] We organize it in three steps.

\smallskip
\noindent (i) \emph{Compactness.} Let $0<\tau<T<\infty$. Consider Theorem~\ref{Aubin-Lions-Simon} with $X_{R,R'}$ (introduced in Appendix~\ref{sec:CompactEmbeddingOfMixedSpaces}, Remark~\ref{X.compact}), $B=L^2(B_{R'}\times(-R,R))$, and $Y=H^{-1}(B_{R'}\times(-R,R))$ and time interval $(\tau,T)$. From~(b),
$$
\displaystyle
\begin{aligned}
\int_\tau^T|u_\lambda(t)|_{L^2(\rr; H^{\alpha/2}(\rr^{N-1}))}^2&=\int_\tau^T \int _{\rr^N} \big|(-\Delta_{x'})^{\alpha/4}u_\lambda(t) \big|^2\\
&\lesssim\int_\tau^T \int _{\rr^N}\big(|(-\Delta_{x'})^{\alpha/4}u_\lambda(t)|^2+\lambda^{1-\alpha\beta} |(-\partial^2_{x_Nx_N})^{\alpha/4}u_\lambda(t)|^2 \big)  \\
&\lesssim \tau^{-\frac 1q(1+\frac{N-1}{\alpha})}  \|u_0\|^{\frac{1}{q}}_{L^1(\rr^N)}.
\end{aligned}
$$
Hence, using also Proposition~\ref{uniform.subcritical}(a)  and (d), we obtain that the family $\{u_\lambda\}_{\lambda>1}$ is uniformly bounded in $L^2((\tau,T);X_{R,R'})$. By Proposition~\ref{uniform.subcritical}(e), time translations are controlled in $L^2((\tau,T);Y)$. We then deduce that up to a subsequence $u_\lambda\rightarrow U$ in $L^2((\tau,T);B_{R'}\times (-R,R))$. A diagonal argument gives us that $u_\lambda\rightarrow U$ in $L^2_{\textup{loc}}(Q)$, so in  $L^1_{\textup{loc}}(Q)$.
The tail control in Proposition~\ref{uniform.subcritical}(f)  yields that the convergence also holds in $L^1_\textup{loc}((0,\infty);L^1(\rr^N))$. Then for a.e.~$t>0$, $u_\lambda(t)\rightarrow U(t)$ in $L^1(\rr^N)$ and $u_\lambda(t,x)\rightarrow U(t,x)$ for a.e.~$(t,x)\in Q$. Moreover, the function $U$ inherits the properties in Proposition~\ref{uniform.subcritical}. In particular,
$$
U\in L^\infty((0,\infty);L^1(\rr^N))\cap L^\infty((\tau,\infty);L^\infty(\rr^N))\text{ for all }\tau>0,\quad\int_{\rr^N} U(t)=M\text{ for a.e. }t>0.
$$

We will later prove that the limit profile $U$ is an entropy solution of equation~\eqref{lim.3}.
The uniqueness of the entropy solution (cf. Theorem~\ref{thm.UniquenessOfLimitEquationsIandII}) shows that in fact the whole sequence $\{u_\lambda\}_{\lambda>0}$ converges to $U$, not only a subsequence.
As described in the supercritical/critical case, the convergence for a fixed positive time $t_1$ of $u_\lambda(t_1)$ toward $U(t_1)$ in $L^1(\rr^N)$ shows the desired convergence in~\eqref{main.limit}.

\smallskip
\noindent (ii) \emph{Identification of the limit equation.} Let us now concentrate on the equation satisfied by $U$.  In view of Remark \ref{regularity.issue} a), since $u_\lambda$ is an entropy solution of~\eqref{rescalada.subcritico},  it satisfies: for all $k\in \rr$, all $\rho>0$, and all $0\leq \phi\in C_\textup{c}^\infty(\overline{Q})$, 
\[
\begin{split}
&\iint_Q \Big( |u_\lambda-k|\partial_t\phi +  \sgn(u_\lambda-k)\big(F(u_\lambda)-F(k)\big)\cdot\nabla\phi \Big)\\
&\qquad-\iint_Q \Big(\sgn(u_\lambda-k)\phi(\ml^{\lambda,> \rho}u_\lambda) + |u_\lambda-k|(\ml^{\lambda,\leq  \rho}\phi)\Big)+\int_{\rr^N}|u_{0,\lambda}-k|\phi(0)\geq 0.
\end{split}
\]
Let us choose $0\leq \phi\in C_\textup{c}^\infty(Q)$, and a positive $\tau$ and $T$ such that $\phi$ is supported in $[\tau,T]$. Using the same arguments as in the supercritical/critical case we have
\[
\int_\tau^T\int_{\rr^N}   |u_\lambda-k|\partial_t\phi\rightarrow \int_\tau^T\int_{\rr^N}  |U-k|\partial_t\phi\quad\text{as }\lambda\to\infty.
\]
When $q<1$,   the strong $L^1$-convergence of $u_\lambda$ towards $U$   implies $u_\lambda^q\rightarrow U^q$ in $L^{1/q}(\mathbb{R}^N)$ so, as $\lambda\to\infty$,
\[
 \int_\tau^T\int_{\rr^N} \sgn(u_\lambda-k)\big(F(u_\lambda)-F(k)\big)\cdot\nabla\phi\rightarrow\int_\tau^T\int_{\rr^N}  \sgn(U -k)\big(F(U)-F(k)\big)\cdot\nabla\phi.
\]
When $q>1$ we obtain the same result with the arguments in the supercritical/critical case.

We now consider the terms involving the truncated operators. In view of Lemma~\ref{conv.L}, we have
\[
\begin{split}
&\iint_Q |u_\lambda -k||\ml^{\lambda,\leq\rho}\phi -\ml'^{,\leq\rho}\phi |
\leq  \iint_Q |u_\lambda  ||\ml^{\lambda,\leq\rho}\phi -\ml'^{,\leq\rho}\phi |+|k|\iint_Q  |\ml^{\lambda,\leq\rho}\phi -\ml'^{,\leq\rho}\phi |\\
&\qquad\leq MT \|\ml^{\lambda,\leq\rho}\phi  -\ml'^{,\leq\rho}\phi \|_{L^\infty( \rr^N)}+|k| \|\ml^{\lambda,\leq\rho}\phi  -\ml'^{,\leq\rho}\phi \|_{L^1( \rr^N)}\lesssim \lambda^{\frac 1\alpha-\beta}.
	 	\end{split}
\]
Then, by
\[
\begin{split}
\int_\tau^T\int_{\rr^N}  \big| |u_\lambda -k| & -|U-k| \big| |\ml'^{,\leq\rho}\phi |\leq \int_\tau^T\int_{\rr^N}  |u_\lambda  -U|  |\ml'^{,\leq\rho}\phi |,
\end{split}
\]
we obtain that
\[
\int_\tau^T\int_{\rr^N}  |u_\lambda-k|(\ml^{\lambda,\leq\rho}\phi)\rightarrow \int_\tau^T\int_{\rr^N}  |U-k|(\ml'^{,\leq\rho}\phi)\quad\text{as }\lambda\to\infty.
\]
Let us now consider the term
\[
I_\lambda=\int_\tau^T\int_{\rr^N}  \sgn(u_\lambda-k)\phi(\ml^{\lambda,> \rho}u_\lambda).
\]
We will prove that $I_\lambda\rightarrow I$ as $\lambda\to\infty$, where
\[
I=\int_\tau^T\int_{\rr^N}  \sgn(U-k)\phi(\ml'^{,>\rho}U).
\]
The main ingredients are the following two uniform estimates in $\lambda$:
\begin{equation*}
  \|\ml ^{\lambda,>\rho} \phi \|_{L^1(\rr^N)}\lesssim \rho^{-\alpha} \|\phi\|_{L^1(\rr^N)}, \  \|\ml ^{\lambda,>\rho} \phi \|_{L^\infty(\rr^N)}\lesssim \rho^{-\alpha} \|\phi\|_{L^\infty(\rr^N)},
\end{equation*}
which are easily obtained by Fubini's theorem. We write the difference as $I_{1,\lambda}+I_{2,\lambda}+I_{3,\lambda}$ where
\[
\begin{aligned}
I_{1,\lambda} &=\int_\tau^T\int_{\rr^N} \sgn (u_\lambda-k)\phi \big(\ml^{\lambda,>\rho}u_\lambda-\ml ^{\lambda,>\rho}U\big),
\\
I_{2,\lambda} &=\int_\tau^T\int_{\rr^N} \big(\sgn (u_\lambda-k) -\sgn (U-k) \big)\phi (\ml ^{\lambda,>\rho}U),
\\
I_{3,\lambda} &=\int_\tau^T \int_{\rr^N} \sgn (U-k) \phi \big (\ml ^{\lambda,>\rho}U - \ml'^{,>\rho}U\big).
\end{aligned}
\]
The first one goes to zero by using the strong $L^1(\rr^N)$-convergence on $(\tau,T)$:
\[
|I_{1,\lambda}|\lesssim \|\phi\|_{L^\infty(Q)}\rho^{-\alpha}\int_\tau^T\|u_\lambda(t)-U(t)\|_{L^1(\rr^N)}\,{\rm d}t\rightarrow 0\quad\text{as } \lambda\to\infty.
\]
As for $I_{2,\lambda}$, we use $\| \ml ^{\lambda,>\rho}U \|_{L^\infty((\tau,T)\times \rr^N)}\lesssim \rho^{-\alpha}\| U\|_{L^\infty((\tau,T)\times \rr^N)}\lesssim C(\tau)$ and the dominated convergence theorem, since $u_\lambda\rightarrow U$ for a.e.~$(t,x)$ in
$Q$. For $I_{3,\lambda}$, since $U(t)\in L^1(\rr^N)$, we use Lemma~\ref{>rho,uniform} to get $\ml ^{\lambda,>\rho}U(t)\rightarrow  \ml'^{,>\rho}U(t)$ in $L^1(\rr^N)$, and the same holds for the whole time interval $(\tau,T)$.

We end this step by putting together all the above convergences to get that the limit profile $U$ satisfies for all $k\in \rr$, all $\rho>0$, and all $0\leq \phi\in C_\textup{c}^\infty(Q)$,
\[
\begin{split}
&\iint_Q \Big( |U-k|\partial_t\phi +  \sgn(U-k)\big(F(U)-F(k)\big)\cdot\nabla\phi \Big)\\
&\qquad-\iint_Q \Big(\sgn(U-k)\phi(\ml'^{,>\rho}U) + |U-k|(\ml'^{,\leq\rho}\phi)\Big)\geq 0,
\end{split}
\]
i.e., $U$ solves~\eqref{lim.3} according to Definition~\ref{def.entropySolution}(b).

\smallskip
\noindent(iii) \emph{Identification of the initial data.} Let us recall that $u_\lambda$ is not only an entropy solution but also a very weak solution, i.e., for all $\phi \in C_\textup{c}^\infty(\overline{Q})$,
\[
\iint_Q  \Big( u_\lambda  \partial_t\phi + u_\lambda^q\partial_{x_N}\phi
 -u_\lambda (\mathcal{L}^\lambda  \phi) \Big)+\int _{\rr^N}u_{0,\lambda}\phi(0)=0.
 \]
As in the supercritical/critical case, we have that for all test functions $\psi\in C_\textup{c}^\infty(\rr^N)$
 \begin{align*}
\label{}
\bigg|  \int_{\rr^N}u_\lambda(t)\psi-\int_{\rr^N}u_{0,\lambda}\psi\bigg|\leq
\int_0^t\int_{\rr^N}\Big(  | u_\lambda|^q |\partial_{x_N}\psi |+|u_\lambda| |\mathcal{L}^\lambda \psi|\Big).
\end{align*}
The estimate in~\eqref{est.L.lambda} gives $\|\ml^\lambda \psi\|_{L^\infty(\rr^N)}\lesssim \|\psi\|_{W^{2,\infty}(\rr^N)}$. Thus
the same arguments as in the proof of the tail control in Proposition~\ref{uniform.subcritical}(f) show that
\[
\bigg|  \int_{\rr^N}u_\lambda(t)\psi -\int_{\rr^N}u_{0,\lambda}\psi\bigg|\leq
  C(\varphi)(t+t^{a(\alpha,q,N)}).
\]
By letting $\lambda\rightarrow\infty$, we get
\[
\bigg|  \int_{\rr^N}U (t)\psi-M\psi(0)\bigg|\leq
  C(\varphi)(t+t^{a(\alpha,q,N)})\quad\text{for a.e. }t>0,
\]
which is then Definition~\ref{def.entropySolution}(c) for all function $\psi\in C_\textup{c}^\infty(\rr^N)$. To extend this result to all $\psi\in C_\textup{b}(\rr^N)$,   we first obtain the tail control for  the limit $U$. Letting $\lambda\rightarrow\infty$ in
~\eqref{tail.control.2} gives us that $U$ satisfies
 \[
 \int_{|x|>2R}|U(t)|\lesssim \frac{t}{R^{\alpha}}+\frac{t^{a(\alpha,q,N)}}{R^{b(q,N)}}\quad\text{for a.e. } t>0\text{ if }R>0.
 \]
As in the supercritical/critical case,
\[
\esslim_{t\to 0^+}\int_{\rr^N} U(t)\psi=M\psi(0) \quad\text{for all }\psi\in C_\textup{b}(\rr^N). \qedhere
\]
\end{proof}

%%%%%%%%%%%%%%%%%%%%%%%%%%%%%%%%%%%%%%%%%%%%%%%%%%%%%%%%%%%%%%%%%%%%%%%%%%%%%%%%%%
\section{Uniqueness of the limit problems}
\label{sec:UniquenessOfLimitProblems}
\setcounter{equation}{0}

This section is devoted to prove that the kind of problems that may appear as limits of rescaled solutions have, under certain assumptions on $q$ and $N$, a unique fundamental solution with mass $M$. This is the content of Theorem~\ref{thm.UniquenessOfLimitEquationsIandII}.

\begin{proof}[Proof of Theorem~\ref{thm.UniquenessOfLimitEquationsIandII}]
(a)   This was already proved in \cite[Section 3]{dePablo-Quiros-Rodriguez-2020}.

\noindent(b) The proof is very technical. Hence, we divide it into several steps. Assume that~\eqref{lim.2} has two entropy solutions $u, \ou \in L^\infty((0,\infty);L^1(\rr^N))\cap L^\infty_{\textup{loc}}((0,\infty);L^\infty(\rr^N))$ with the same initial data $M\delta_0$.

\smallskip
\noindent(i) \emph{Integrating in the direction $x_N$.} Using the ideas of the proof of Lemma~\ref{lem.PrimitiveSolvesParabolic}, it is easy to check that
\begin{equation}\label{integrated.u}
v(t,x'):=\int_{\rr}u(t,x',x_N)\,{\rm d}x_N \quad\textup{for a.e. }(t,x')\in (0,\infty)\times \rr^{N-1}
\end{equation}
belongs to $L^\infty((0,\infty);L^1(\rr^{N-1}))$ and is the unique very weak solution of $\partial_tv+\widetilde\ml v=0$ in $(0,\infty)\times \rr^{N-1}$ with initial data $M\delta_0$ in the sense of bounded measures, where $\widetilde\ml$ is the $(N-1)$-dimensional $\alpha$-stable operator introduced in~\eqref{tildeL}. Thus, $v=M\Phi$, where $\Phi$ is the unique fundamental solution of the equation with mass 1, which has a self-similar form,
\[
\Phi(t,x')=t^{-\frac{N-1}\alpha}\mathcal{G}(t^{-\frac 1\alpha}x')\quad\textup{for all }(t,x')\in (0,\infty)\times \rr^{N-1}
\]
for some smooth positive profile $\mathcal{G}$ such that $\int_{\mathbb{R}^N}\mathcal{G}=1$; see~\cite{dePablo-Quiros-Rodriguez-2020}.
In terms of $u$ and $\ou$, we then have
\begin{equation*}
\label{both.Nmass}
  \int_{\rr} u(t,x',x_N)\,{\rm d}x_N=  \int_{\rr} \ou(t,x',x_N)\,{\rm d}x_N=M\Phi(t,x') \quad\textup{for all }(t,x')\in (0,\infty)\times \rr^{N-1}.
\end{equation*}

\smallskip
\noindent(ii) \emph{Approximation by entropy solutions with initial data in $L^1(\rr^N)\cap L^\infty(\rr^N)$.} We start by constructing the initial data for our approximations.
Inspired by~\cite{EVZIndiana}, for each $r>0$,  we define
\[
\begin{gathered}
\varphi_r(t,x')=\int_{-r}^r u(t,x',x_N)\,{\rm d}x_N-M\Phi(t,x')=-\int_{|x_N|>r} u(t,x',x_N)\,{\rm d}x_N,
\\
u^r_{0,n}(x',x_N):=\big(u(1/n,x',x_N)-\frac 1{2r}\varphi_r(1/n,x')\big)\mathds{1}_{(-r,r)}(x_N).
\end{gathered}
\]
Notice that $ -M\Phi(t,x')\leq \varphi_r(t,x')\leq 0$. Hence, the functions $u^r_{0,n}$ are nonnegative and  bounded,
\[
\|u^r_{0,n}\|_{L^\infty(\rr^{N})}\leq \| u(1/n)\|_{L^\infty(\rr^N)}+\frac{M}{2r}\|\Phi (1/n)\|_{L^\infty(\rr^{N-1})}.
\]
On the other hand, they have integral $M$ since
\begin{equation}
	\label{mass.u_0n}
\int_{\rr} u^r_{0,n}(x',x_N)\,{\rm d}x_N=M\Phi(1/n,x')\quad\text{for a.e. }x'\in \rr^{N-1}.
\end{equation}
Let $u_n$ (we omit the dependence on the parameter $r$ for simplicity) be the entropy solution of~\eqref{lim.2}  with initial data $u _{0,n}^r$. Let us prove that,   for any $t>0$, the sequence  $u_n(t)$ converges to $u(t)$ in $L^1(\rr^N)$ as $n\to\infty$. To this aim we observe that both $u_n$ and $u(\cdot +1/n)$ are entropy solutions of~\eqref{lim.2},  whose respective initial data, $u_{0,n}^r$ and $u(1/n)$, belong to $L^1(\rr^N)\cap L^\infty(\rr^N)$. Then the $L^1$-contraction property, Theorem~\ref{thm.UniquenessPropertiesEntropy}(a)(ii), shows that
\begin{align*}
\big\|u_n(t)-u\big(t+1/n\big)\big\|_{L^1(\rr^N)}&\leq \| u^r_{0,n}-u(1/n)\|_{L^1(\rr^N)}\\
&\leq
\int_{\rr^{N-1}}\bigg(\int_{|x_N|>r} u(1/n,x',x_N)\,{\rm d}x_N+ |\varphi_r(1/n,x')| \bigg)\,{\rm d}x'\\
&\leq 2\int_{\rr^{N-1}}\int_{|x_N|>r} u(1/n,x',x_N)\,{\rm d}x_N{\rm d}x' \rightarrow 0 \quad\text{as  $n\rightarrow\infty$.}
\end{align*}
Since $u\in C((0,\infty);L^1(\rr^N))$ (see Remark~\ref{regularity.issue}), we have that $u(t+1/n)$ converges to $u(t)$ in $L^1(\rr^N)$ as $n\rightarrow \infty$. Therefore, $u_n(t)\rightarrow u(t)$ in $L^1(\rr^N)$ as $n\rightarrow \infty$. The contraction property implies  that $u_n\rightarrow u$ in $C([\tau, \infty);L^1(\rr^N))$ as $n\rightarrow \infty$ for all  $\tau>0$.

Defining $\ou _{0,n}^r$  and then  $\ou_n$ in a similar way, we get by the same reasoning that $\ou_n(t)\rightarrow \ou(t)$ in $L^1(\rr^N)$ as $n\rightarrow \infty$.

\smallskip
\noindent (iii) \emph{Approximation by classical solutions.} If $q>1$, let $u_n^\eps$ be the solution of the regularized problem~\eqref{reg}  with initial data $u_{0,n}^r$ and nonlinearity $f_0(u)=u^q$. If $q<1$, we consider the same problem but with nonlinearity
\begin{equation}
\label{f.eta}
f_\eta(u)=(u^2+\eta)^{q/2}-\eta^{q/2},\quad \eta>0.
\end{equation}The corresponding solution will depend on $\eps$ and $\eta$ but for simplicity
we still denote it by  $u_n^\eps$ since   $\eta=\eta_\eps\rightarrow 0$ as $\eps\rightarrow 0$. Note that $f_\eta$ is $C^\infty(\rr)$ and satisfies the inequality
\begin{equation}
\label{ineq.f.eta}
 0\leq  f_\eta(s)=(s^2+\eta)^{q/2}-\eta^{q/2}\leq s^q,\quad \text{for all $s\geq 0$.}
\end{equation}

  The regularity results in Proposition~\ref{prop.propOfRegularizedEquationSmoothness2} show that $u_n^\eps$ is a classical solution of~\eqref{reg}, $u_n^\eps\in C^{1+\delta, 2+2\delta} (Q)$. Besides, Lemma~\ref{lem.StabilityOfSolutionsOfReg2} and Lemma~\ref{lem.StabilityOfSolutionsOfReg3} with $\lambda=1$,  proved in~Sections~\ref{sec.CompactnessCLloc1}--\ref{sec.Appendix.ConvInCL1}, and Lemma \ref{lem.StabilityOfperturbationofnonlinearity}  show that $u_n^\eps\rightarrow u_n$ in $C((0,\infty);L^1(\rr^N))$ as $\eps\to0^+$.  Hence,  in view of the previous step, $u_n^\eps\rightarrow u$ in $C([\tau,\infty);L^1(\rr^N))$ for all $\tau>0$.  A similar result holds for $\ou_n^\eps$, which is defined analogously.

\smallskip
\noindent (iv) \emph{The primitives are bounded solutions of the Hamilton-Jacobi equation.} It is easily checked that
\[
v_n^\eps (t,x',x_N):=\int_{-\infty}^{x_N} u^\eps_n(t,x',y_N)\,{\rm d}y_N\quad\text{for all }(t,x',x_N)\in (0,\infty)\times\rr^{N-1}\times\rr
\]
is a nonnegative classical solution of the Hamilton-Jacobi equation
\begin{equation}
\label{HJ-nonlocal}
  \partial_tv+\ml v -\eps \Delta v =f_\eta(\partial_{x_N}v) \quad\text{in }Q.
\end{equation}
Let us prove that it is bounded. Indeed, since $u^\eps_n\in C([0,\infty);L^1(\rr^N))$ is nonnegative,
\begin{equation}
	\label{vcomparedtow}
v_n^\eps(t,x',x_N)\leq \int_{\rr} u^\eps_n(t,x',y_N)\,{\rm d}y_N=:w_\eps^n(t,x'),
\end{equation}
where, by Lemma~\ref{lem.PrimitiveSolvesParabolic} (which also holds in the case of $f_\eta$, see Remark \ref{extended:lem.PrimitiveSolvesParabolic}) and \eqref{mass.u_0n},  the function $w_\eps^n$ belongs to $C([0,\infty);L^1(\rr^{N-1}))$ and solves 
\begin{equation}
\label{eq:wneps}
\partial_tw+\widetilde\ml w-\eps\Delta_{x'} w=0\quad\text{in $(0,\infty)\times\rr^{N-1}$,}\qquad w(0)=M\Phi(1/n)\quad\text{on $\rr^{N-1}$.}
\end{equation}
The solution of this problem is $w_\eps^n(t)=M\Phi(t)\ast G_{\eps t}\ast \Phi(\frac 1n)$, where $G_t$ is the classical Gaussian heat kernel. Hence,
$\|w_\eps^n(t)\|_{L^\infty(\rr^{N-1})}\leq M\| \Phi(1/n)\|_{L^\infty(\rr^{N-1})}\leq Mc_n$ for some positive constant $c_n$, and $v^\eps_n$ is  bounded. The same results, with the same bound, hold for $\overline v^{\eps}_n$, which is defined analogously.

\smallskip
\noindent (v) \emph{Comparison of the traces.} We claim that, for a.e.~$(x',x_N)\in \rr^{N-1}\times\rr$,
\begin{equation}
\label{eq:comparison.integrals}
\int_{-\infty}^{x_N-2r} u_{0,n}^r(x',y_N)\,{\rm d}y_N\leq \int_{-\infty}^{x_N} \ou_{0,n}^r(x',y_N)\,{\rm d}y_N \leq \int_{-\infty}^{x_N+2r} u_{0,n}^r(x',y_N)\,{\rm d}y_N.
\end{equation}
We prove only the first inequality since the second one can be obtained similarly.
By construction, $u_{0,n}^r(x',\cdot)$ is supported in $(-r,r)$. Thus, if $x_N<r$,   the left-hand-side term vanishes, and,   hence the inequality is true since $\ou_{0,n}^r$ is nonnegative. If $x_N>r$, using  the support of
$\ou_{0,n}^r(x',\cdot) $ and~\eqref{mass.u_0n},
\begin{align*}
\label{}
  \int_{-\infty}^{x_N-2r} u_{0,n}^r(x',y_N)\,{\rm d}y_N&\leq  \int_{-\infty}^{\infty} u_{0,n}^r(x',y_N)\,{\rm d}y_N=
  M\Phi(1/n,x')=\int_{-\infty}^{\infty} \ou_{0,n}^r(x',y_N)\,{\rm d}y_N\\
  &=\int_{-\infty}^{r} \ou_{0,n}^r(x',y_N)\,{\rm d}y_N =\int_{-\infty}^{x_N} \ou_{0,n}^r(x',y_N)\,{\rm d}y_N,
\end{align*}
and the inequality is also true.

On the other hand, since $u^\eps_n\in C([0,\infty),L^1(\rr^N))$, then $u^\eps_n(t)\rightarrow u_{0,n}^r$  in $L^1(\rr^N)$ as $t\rightarrow 0^+$. Thus,
$$
v_n^\eps (t,x',x_N)\rightarrow \int_{-\infty}^{x_N} u_{0,n}^r(x',y_N)\,{\rm d}y_N\quad\text{in }L^1_{x'} (\rr^{N-1};L_{x_N}^\infty(\rr))\text{ as }t\rightarrow 0^+,
$$
and,   hence, a.e.~in $\rr^{N-1}\times\rr$ as well. Similar arguments hold for $ \overline v^\eps_n$. Therefore,~\eqref{eq:comparison.integrals} translates into
\begin{equation}
\label{claim.v}
  v^\eps_n(0+,x',x_N-2r)\leq \overline v^\eps_n(0+,x',x_N)\leq v^\eps_n(0+,x',x_N+2r) \quad\text{for a.e} \ (x',x_N)\in \rr^{N-1}\times\rr.
\end{equation}

\smallskip
\noindent(vi) \emph{Comparison of the primitives.} Let us now show that the inequalities~\eqref{claim.v} for the traces imply
\begin{equation}
\label{claim.v.2}
   v^\eps_n(t,x',x_N-2r)\leq \overline v^\eps_n(t,x',x_N)\leq v^\eps_n(t,x',x_N+2r)\quad\text{for a.e. }(t,x',x_N)\in (0,\infty)\times \rr^{N-1}\times\rr.
\end{equation}
We only prove the first inequality, since the proof of the second is analogous. To this aim,   we define
\[
g(t,x',x_N):=v^\eps_n(t,x',x_N-2r)- \overline v^\eps_n(t,x',x_N)\quad\text{for a.e. }(t,x',x_N)\in (0,\infty)\times\rr^{N-1}\times\rr.
\]
By step (iv), $g$ is uniformly bounded by $2Mc_n$. Moreover, $g\in C^{1+\delta, 2+2\delta} (Q)$ and satisfies
\begin{eqnarray}\label{eq.g}
\partial_t g+ \ml g-\eps\Delta g= a(t,x) \partial_{x_N}g\quad\text{in }Q,\qquad \lim _{t\to0^+}g(t)\leq 0\quad\text{a.e. on }\rr^{N},
\end{eqnarray}
where
$$
a(t,x):=f_\eta'(\zeta(t))\quad\text{with $\zeta(t)$ between $\partial_{x_N} v^\eps_n(t)$ and $\partial_{x_N}\overline v^\eps_n(t)$}.
$$
Note that $|a(t,x)|\leq c_{\eta,q}$ if $q<1$, and $|a(t,x)|\leq c_{n,q}$ if $q>1$.

Following~\cite[Theorem 3]{MR2422079}, we consider $\psi:\rr^N\rightarrow \rr_+$ smooth such that  $\psi(x)= 0$ for $|x|\leq 1$, $\psi(x)>\|g\|_{L^\infty(Q)}$ for $|x|\geq2$, with $\psi, \nabla\psi$ and $D^2\psi$ bounded in $\rr^N$. With this choice we introduce, for all $\beta>0$, $\psi_\beta(x)=\psi(\beta x)$. It follows that   $\psi_\beta(x)>\|g\|_{L^\infty(Q)}$ for $|x|\geq2/\beta$, and $\nabla \psi_\beta$, $D^2\psi_\beta$, and $\ml\psi_\beta$ all go to zero uniformly on $\rr^N$ when $\beta \rightarrow 0^+$. Let  $\tilde\delta>0$ and $h(t,x):=g(t,x)-\tilde\delta t-\psi_\beta(x)$. Then $h_+(t):=\max\{h(t),0\}$ is smooth and satisfies $\nabla h_+(t)=\nabla h(t)$ in the interior of its support,  which is contained in the ball $|x|<2/\beta$ for all $t>0$.  Hence, by~\eqref{eq.g},
\begin{align*}
	\frac 12\frac{\rm d}{{\rm d}t}&\int_{\rr^N }h_+^2(t)=\int_{\rr^N}h_+(t)\partial_th(t)=	\int_{\rr^N}h_+(t)(\partial_tg(t)-\tilde\delta)\\
&=\int_{\rr^N } h_+(t)  (-  \ml g(t)+ \eps \Delta g(t)+a \partial_{x_N}g(t)-\tilde\delta)\\
	&=\int_{\rr^N } h_+(t)\big(-  \ml  h(t)+ \eps \Delta h(t)+a \partial_{x_N}h(t)\big)+
	\int_{\rr^N } h_+(t)\big(-  \ml  \psi_\beta+ \eps \Delta  \psi_\beta+a \partial_{x_N} \psi_\beta-\tilde\delta\big).
\end{align*}
Observe that for $\beta$ small enough, $\beta <\beta(\tilde\delta)$, we have
\[
\big|-  \ml  \psi_\beta+ \eps \Delta  \psi_\beta+a \partial_{x_N} \psi_\beta\big|<\tilde\delta \quad\text{in }\rr^N.
\]
Moreover, for bounded and  smooth $h$ and compactly supported $h_+$,   we have
\[
\begin{aligned}
\int_{\rr^N } h_+  ( \ml h)\,{\rm d}x=\mathcal{E}(h_+,\ml h)&=\displaystyle\frac14\int_{\mathbb{R}^{N}}\int_{\mathbb{R}^{N}}
	\big(h_+(x+y)-h_+(x)\big)\big(h(x+y)-h(x) \big)\,{\rm d}\nu(y){\rm d}x\\
	&\quad+\frac14\int_{\mathbb{R}^{N}}\int_{\mathbb{R}^{N}}
	\big(h_+(x)-h_+(x-y) \big)\big(h(x)-h(x-y) \big)\,{\rm d}\nu(y){\rm d}x\geq 0,
\end{aligned}
\]
since the map $s\mapsto s _+$ is increasing. Hence, for all $t>0$, Young's inequality gives
\begin{align*}
	\frac 12\frac{\rm d}{{\rm d}t}\int_{\rr^N }h_+^2(t)&\leq 	-\eps\int_{\rr^N }\nabla  h_+(t)\cdot \nabla h(t)+\int_{\rr^N }a h_+(t)\partial_{x_N}h(t)\\
&=-\eps \int _{\rr^N }|\nabla h_+(t)|^2 +\int_{\rr^N} ah_+(t) \partial_{x_N}h_+(t)\\
	&\leq-\eps \int _{\rr^N }|\nabla h_+(t)|^2 + \frac{c^2_{n,q}}{2\vartheta}\int_{\rr^N}  h _+^2(t) +\frac {\vartheta }2
	\int_{\rr^N}|\partial_{x_N}h_+(t)|^2\leq C(\eps,n,q)  \int_{\rr^N} h_+^2(t),
\end{align*}
 where we chose $\vartheta=2\eps$.  Then for all $0<s<t$ we have
\[
\int_{\rr^N }h_+^2(t)\leq \int_{\rr^N }h_+^2(s) + C(\eps,n,q) \int_s^t  \int_{\rr^N }h_+^2.
\]
Observing that
\[
\lim_{s\rightarrow 0^+} \int_{\rr^N }h_+^2(s)= \lim_{s\rightarrow 0^+} \int_{|x|<2/\beta}\big(g(s)-\tilde\delta s-\psi_\beta\big)_+^2 =0,
\]
we conclude that $h_+\equiv 0$, by Gr\"onwall's inequality, so that
\[
g(t,x)\leq \tilde\delta t+\psi_\beta(x) \quad \text{for all }t>0,\text{ all }x\in \rr^N,\text{ and all }\beta<\beta(\tilde\delta).
\]
This implies that $g(t,x)\leq \tilde\delta t$ on $\{|x|<1/\beta\}$ and,   hence, on $\rr^N$ after letting $\beta\rightarrow 0^+$. Finally, the result follows by taking   the limit $\tilde\delta\rightarrow 0^+$.

\smallskip
\noindent(vii) \emph{Conclusion.} Inequalities~\eqref{claim.v.2} can be written as
\[
\int_{-\infty}^{x_N-2r} u_n^\eps(t,x',y_N)\,{\rm d}y_N\leq \int_{-\infty}^{x_N} \ou_n^\eps(t,x',y_N)\,{\rm d}y_N \leq \int_{-\infty}^{x_N+2r} u_n^\eps(t,x',y_N)\,{\rm d}y_N.
\]
Since  $u_n^\eps\rightarrow u_n$ in $C([\tau,\infty);L^1(\rr^N))$ for all $\tau>0$, we deduce that $u_n^\eps(t,x',\cdot)\rightarrow u_n (t,x',\cdot)$ in $L^1(\rr)$ for a.e.~$x'\in \rr^{N-1}$ and all $t\geq \tau$. Then, letting $\eps\rightarrow 0$ we get, up to a subsequence,
\[
\int_{-\infty}^{x_N-2r} u_n (t,x',y_N)\,{\rm d}y_N\leq \int_{-\infty}^{x_N} \ou_n (t,x',y_N)\,{\rm d}y_N \leq \int_{-\infty}^{x_N+2r} u_n (t,x',y_N)\,{\rm d}y_N\quad\text{for a.e. }x'.
\]
On the other hand, since   $u_n(t)\rightarrow u(t)$ in $ L^1(\rr^N)$,  up to a subsequence, $u_n (t,x',\cdot)\rightarrow u (t,x',\cdot)$ in $L^1(\rr)$ for a.e.~$x'\in \rr^{N-1}$. We conclude that,  for all $t\geq\tau$,
\[
\int_{-\infty}^{x_N-2r} u  (t,x',y_N)\,{\rm d}y_N\leq \int_{-\infty}^{x_N} \ou  (t,x',y_N)\,{\rm d}y_N \leq \int_{-\infty}^{x_N+2r} u  (t,x',y_N)\,{\rm d}y_N\quad\text{for a.e. }x'.
\]
Letting $r\to0^+$, we deduce that $u(t,x',x_N)=\ou (t,x',x_N)$ for a.e.~$(x',x_N)\in \rr^{N-1}\times\rr$ and all $t\geq \tau$, whence the uniqueness result.

\noindent(c) The proof  is similar to that of (b). Let us comment on the points where some care has to be taken.
In step (i) we still have, using the identity ${\ml'}(\psi(\cdot') {\mathds{1}}_{\mathbb{R}}(\cdot_N))(x',x_N)=(\tilde{\mathcal{L}}\psi )(x'){\mathds{1}}_{\mathbb{R}}(x_N)$ from Remark~\ref{symbols.3}, that $v$ defined by~\eqref{integrated.u} is the unique fundamental solution with mass $M$ of  $\partial_t v+\widetilde\ml u=0$. In step (iii) we need to consider the analogue of \eqref{reg}:
\begin{equation}\label{reg'}\tag{$\textup{P}_{\varepsilon}'$}
\partial_tu+\ml' u+\partial_{x_N} (f_\eta(u)) =\varepsilon \Delta u\quad\textup{in }Q,\qquad
  	u(0,\cdot)=u_0\quad\textup{in }\rr^N,
\end{equation}
Hence, the primitives in step (iv) satisfy the Hamilton-Jacobi equation
\[
\partial_t v+\ml' v-\eps \Delta v=f_\eta(\partial_{x_N}v) \quad \text{in }Q,
\]
instead of~\eqref{HJ-nonlocal}. We can bound these primitives in terms of a function $w_\eps^n$ defined as in~\eqref{vcomparedtow}, which solves the same problem~\eqref{eq:wneps} as the function $w_\eps^n$ that was defined for case (b), and that hence has the same bound. The proof of step (vi) also works nicely, once we observe that $\int_{\rr^N} h_+(\ml' h)\geq 0$.
\end{proof}

%%%%%%%%%%%%%%%%%%%%%%%%%%%%%%%%%%%%%%%%%%%%%%%%%%%%%%%%%%%%%%%%

\subsection*{Acknowledgements}
J. Endal received funding from the European Union’s Horizon 2020 research and innovation programme under the Marie Sk\l odowska-Curie grant agreement no. 839749 ``Novel techniques for quantitative behaviour of convection-diffusion equations (techFRONT)'' and from the Research Council of Norway under the MSCA-TOPP-UT grant agreement no.~312021.

L. I. Ignat received funding from Romanian Ministry of Research, Innovation and Digitization, CNCS-UEFISCDI, project number PN-III-P1-1.1-TE-2021-1539, within PNCDI III.

 F. Quir\'os was supported by grants CEX2019-000904-S, PID2020-116949GB-I00, and RED2018-102650-T, all of them funded by MCIN/AEI/10.13039/501100011033. 

We would like to thank the anonymous referees for helpful suggestions. 

%%%%%%%%%%%%%%%%%%%%%%%%%%%%%%%%%%%%%%%%%%%%%%%%%%%%%%%%%%%%%%%%
\appendix

%%%%%%%%%%%%%%%%%%%%%%%%%%%%%%%%%%%%%%%%%%%%%%%%%%%%%%%%%%%%%%%%
\section{Truncated operators}
\label{sect:appendix.truncated.operators}
\setcounter{equation}{0}

 The aim of this appendix is to prove that the rescaled truncated operators $\ml^{\lambda,>\rho}$ and $\ml^{\lambda,\le\rho}$ ``approach'' the truncated operators $\ml'^{,>\rho}$ and $\ml'^{,\le\rho}$, respectively, as $\lambda\to\infty$.  This fact plays an essential role in the proof of Theorem~\ref{asymptotic} in the convection regime $q<q_*(\alpha)$.

\begin{lemma}\label{>rho,uniform}
Assume $\alpha\in(0,2)$, $\beta>1/\alpha$, and~\eqref{mlas}, and consider the operators $\ml^{\lambda,>\rho}$ and $\ml'^{,>\rho}$ defined respectively through~\eqref{ml.lambda},~\eqref{tilde.L}, and also~\eqref{eq.SplittingOperatorInTwo}. Then, for all $\phi\in L^1(\rr^N)$,
\begin{equation*}
\label{>rho,uniform.est}
  \|\ml^{\lambda,>\rho}\phi-\ml'^{,>\rho}\phi\|_{L^1(\rr^N)}\rightarrow 0 \quad\text{as $\lambda\rightarrow \infty$.}
\end{equation*}
\end{lemma}

\begin{proof}
Using the definition of the two operators and a change of variables in $x'\in \rr^{N-1}$ we obtain
	\[
	\begin{split}
	\int_{\rr^N}| \ml^{\lambda,>\rho}\phi-\ml'^{,>\rho}\phi |
	\leq
	\sum _{\pm} \int_{\rho}^\infty\int_{\bs^{N-1}} \int_{\rr^{N  }} |\phi(x\pm r(0,\lambda^{\frac 1\alpha-\beta}\theta_N))-\phi(x)|  \,{\rm d}x{\rm d}\mu(\theta) \frac{{\rm d}r}{r^{1+\alpha}}.
	\end{split}
	\]
	Since, for all $(r,\theta)\in(\rho,\infty)\times\bs^{N-1}$,
	\[
	\int_{\rr^{N  }} |\phi(x\pm r(0,\lambda^{-\frac{1}{\alpha}(\alpha\beta-1)}\theta_N))-\phi(x)|\,{\rm d}x\rightarrow 0\quad\text{as }\lambda\rightarrow \infty,
	\]
	and this integral is uniformly bounded by the $L^1$-norm of $\phi$, we can apply the dominated convergence theorem to obtain the desired result.
\end{proof}

\begin{lemma}	\label{conv.L}
Assume $\alpha\in(0,2)$, $\beta>1/\alpha$, and~\eqref{mlas}, and consider the operators $\ml^{\lambda,\leq\rho}$ and $\ml'^{,\leq\rho}$ defined respectively through~\eqref{ml.lambda},~\eqref{tilde.L}, and also~\eqref{eq.SplittingOperatorInTwo}. Then, for all $\phi\in C_\textup{c}^{3}(\rr^N)$,
\[
\begin{aligned}
  &\| \ml^{\lambda,\leq\rho} \phi  -\ml'^{,\leq\rho}\phi \|_{L^\infty(\rr^N)}\lesssim c(\rho) \lambda^{\frac 1\alpha-\beta} \Big(\|D^2\phi \|_{L^\infty(\rr^N)}+\|\partial_{x_N}D^2\phi \|_{L^\infty(\rr^N)}\Big),\\
  &\| \ml^{\lambda,\leq\rho} \phi  -\ml'^{,\leq\rho}\phi \|_{L^1(\rr^N)}\lesssim c(\rho) \lambda^{\frac 1\alpha-\beta}  \Big(\|D^2\phi \|_{L^1(\rr^N)}+\|\partial_{x_N}D^2\phi \|_{L^1(\rr^N)}\Big),
\end{aligned}
\]
where $c(\rho)\eqsim \rho^{2-\alpha }+\rho^{3-\alpha}$, uniformly in $\lambda>1$.
\end{lemma}

\begin{proof}
We use Taylor's theorem with integral reminder:
\[
    f(x+z)-f(x)=\int_0^1 z\cdot \nabla f(x+sz)\,{\rm d}s=z\cdot \nabla f(x)+ \int_0^1 (1-s)z  D^2f(x+sz)z^t\,{\rm d}s.
\]
Hence, by defining $\sigma:=1/\alpha-\beta<0$, we have
\[
    \begin{aligned}
    (\ml^{\lambda,\leq\rho}\phi)(x)&=-\int_{0}^\rho\int_{\bs^{N-1}} \int_0^1(1-s)(\theta',\lambda^{\sigma }\theta_N)D^2\phi(x\pm sr(\theta',\lambda^{\sigma}\theta_N))(\theta',\lambda^{\sigma}\theta_N)^t\,{\rm d}s\frac{{\rm d}r}{r^{\alpha-1}}
    \,{\rm d}\mu(\theta),\\
    (\ml^{\lambda,\leq\rho}\phi)(x)&=-\int_{0}^\rho\int_{\bs^{N-1}} \int_0^1(1-s)(\theta',0)D^2\phi(x\pm sr(\theta',0))(\theta',0)^t\,{\rm d}s\frac{{\rm d} r}{r^{\alpha-1}}
    {\rm d}\mu(\theta).
    \end{aligned}
\]	
We only consider the part with $+$ sign, since the other is similar. We write
$$
    |(\ml^{\lambda,\leq\rho}\phi)(x)-(\ml'^{,\leq\rho}\phi)(x)|\le I_{1,\lambda}(x)+I_{2,\lambda}(x)+I_{3,\lambda}(x),
$$
where
\begin{align*}
I_{1,\lambda}(x)&:=\int_{0}^\rho\int_{\bs^{N-1}} \int_0^1(1-s)\Big|(\theta',\lambda^{\sigma }\theta_N)\big(D^2\phi(x+ sr(\theta',\lambda^{\sigma}\theta_N))\\
&\hskip6cm-D^2\phi(x+ sr(\theta',0)) \big) (\theta',\lambda^{\sigma}\theta_N)^t\Big|\,{\rm d}s\frac{{\rm d}r{\rm d}\mu(\theta)}{r^{\alpha-1}},\\
I_{2,\lambda}(x)&:=\int_{0}^\rho\int_{\bs^{N-1}} \int_0^1(1-s)\Big|\big((\theta',\lambda^{\sigma }\theta_N)  -(\theta',0)\big)D^2\phi(x+ sr(\theta',0))   (\theta',\lambda^{\sigma}\theta_N)^t\Big|\,{\rm d}s\frac{{\rm d}r{\rm d}\mu(\theta)}{r^{\alpha-1}},\\
I_{3,\lambda}(x)&:=\int_{0}^\rho\int_{\bs^{N-1}} \int_0^1(1-s)\Big|(\theta',0)D^2\phi(x+ sr(\theta',0))   \big((\theta',\lambda^{\sigma }\theta_N)  -(\theta',0)\big)^t\Big|\,{\rm d}s\frac{{\rm d}r{\rm d}\mu(\theta)}{r^{\alpha-1}}.
\end{align*}
By using the identity
		\begin{align*}
			D^2\phi(x+ sr(\theta',\lambda^{\sigma}\theta_N))&-D^2\phi(x+ sr(\theta',0))\\
			&=sr(0,\lambda^\sigma\theta_N)\cdot \int_0^1 \nabla (D^2\phi)\Big(x+sr((\theta',0)+t(0,\lambda^
			\sigma\theta_N))  \Big)\,{\rm d}t\\
			&=sr\lambda^{\sigma}\theta_N \int_0^1 \partial_{x_N}D^2\phi \Big(x+sr(\theta', t \lambda^
			\sigma\theta_N)  \Big)\,{\rm d}t,
 		\end{align*}
and the facts that $\sigma<0$ and $(\theta',\theta_N)\in \bs^{N-1}$, we obtain that
\[
\|I_{1,\lambda}\|_{L^1(\rr^N)}\lesssim \lambda^{\sigma}  \| \partial_{x_N}D^2\phi \|_{L^1(\rr^N)} \int_{0}^\rho\int_{\bs^{N-1}} \frac{{\rm d}r{\rm d}\mu(\theta)}{r^{\alpha-2}}\eqsim
\lambda^{\sigma} \rho^{3-\alpha} \| \partial_{x_N}D^2\phi \|_{L^1(\rr^N)}.
\]
The last two terms can be estimated by brute force in a similar manner:
\[
\|I_{2,\lambda}\|_{L^1(\rr^N)}+\|I_{3,\lambda}\|_{L^1(\rr^N)}\lesssim \lambda^{\sigma}  \|  D^2\phi \|_{L^1(\rr^N)} \int_{0}^\rho\int_{\bs^{N-1}} \frac{{\rm d}r{\rm d}\mu(\theta)}{r^{\alpha-1}}\eqsim \lambda^{\sigma} \rho^{2-\alpha} \|   D^2\phi \|_{L^1(\rr^N)}.
\]
A similar estimate holds  for the $L^\infty$-norm of the three terms.
\end{proof}

%%%%%%%%%%%%%%%%%%%%%%%%%%%%%%%%%%%%%%%%%%%%%%%%%%%%%%%%%%%%%%%%%%%%%%%%%%%%%%%%%%
\section{Basic results for entropy solutions}
\label{sect:appendix.results.entropy.solutions}
\setcounter{equation}{0}

In this section, we will build the basic theory of~\eqref{eq.main},~\eqref{reg}, and~\eqref{reg'} from scratch based on~\cite{Car99, MaTo03, AlibaudEntropy, CifaniJakobsenEntropySol, EndalJakobsen, AnBr20, Pan20}.  The operator $\ml'$ defined by~\eqref{tilde.L} only acts on the first $N-1$ spatial variables. However, most of the results we are going to deduce only depend on upper bounds of the operator, therefore we will always write them for $\ml$ and only give suitable comments on the case $\ml'$, if needed.

\subsection{Concept of entropy solutions}\label{sect:conceptofentropy}
The definition of entropy solutions of~\eqref{eq.main} was already stated as Definition~\ref{def.entropySolution}, let us also give it for~\eqref{reg}:

\begin{definition}[Entropy solution of~\eqref{reg}]\label{def.entropySolutionVeps}
A function $u$ is an \emph{entropy solution} of~\eqref{reg} if:
\begin{enumerate}[{\rm (a)}]
\item \textup{(Regularity)} $u\in
L^\infty((0,\infty);L_\textup{loc}^1(\rr^N))\cap L^\infty_\textup{loc}((0,\infty);L^\infty(\rr^N))\cap L_\textup{loc}^2((0,\infty);H_\textup{loc}^1(\rr^N))$.
\item \textup{(Entropy inequality)} For all $k\in\rr$, all $\rho>0$, and all $0\leq \phi\in C_\textup{c}^\infty(\overline{Q})$,
\begin{equation}\label{eq.entropyInequalityVeps.appendix}
\begin{split}
&\iint_Q \Big( |u-k|\phi_t + \sgn(u-k)\big(F(u)-F(k)\big)\cdot\nabla\phi+\varepsilon|u-k|\Delta \phi \Big)\\
&\qquad-\iint_Q \Big(\sgn(u-k)\phi(\ml^{> \rho}u) + |u-k|(\ml^{\leq \rho}\phi)\Big)+\int_{\rr^N}|u_0-k|\phi(0)\geq 0.
\end{split}
\end{equation}
\end{enumerate}
\end{definition}

\begin{remark}
In the above definition, $u\in L^\infty_\textup{loc}((0,\infty);L^\infty(\rr^N))$ implies that $u\in L_{\textup{loc}}^2(Q)$. Hence, the last regularity assumption is actually $\nabla u\in L_{\textup{loc}}^2(Q)$. Its use is better seen if we rewrite the term involving the Laplacian:
$$
\iint_Q|u-k|\Delta \phi=-\iint_Q\sgn(u-k)\nabla u\cdot\nabla\phi.
$$
\end{remark}

\subsection{Known uniqueness and a priori estimates}\label{sect:knownunique}
We follow the uniqueness argument presented in \cite[Section 4]{AnBr20}. It starts with a Kato inequality; see e.g. \cite[Theorem 3.9]{MaTo03} and \cite[Proposition 4.2]{EndalJakobsen}.

\begin{proposition}[Kato inequality]
Assume~\eqref{qas}, and~\eqref{mlas}. Let $u,\bar{u}$ be entropy solutions of~\eqref{eq.main} with initial data $u_0,\bar{u}_0\in L^\infty(\rr^N)$, respectively. Then, for all $0\leq \phi\in C_\textup{c}^\infty(\overline{Q})$,
\begin{equation}
\label{eq.KatoInequality}
\begin{split}
&\iint_Q \Big( |u-\bar{u}|\phi_t + \sgn(u-\bar{u})\big(F(u)-F(\bar{u})\big)\cdot\nabla\phi -|u-\bar{u}|(\ml\phi)\Big)+\int_{\rr^N}|u_0-\bar{u}_0|\phi(0)\geq 0.
\end{split}
\end{equation}
\end{proposition}

\begin{remark}
For entropy solutions $u,\bar{u}$ of~\eqref{reg}, we add the term $\displaystyle\varepsilon\iint_Q |u-\bar{u}|\Delta \phi$
in~\eqref{eq.KatoInequality}.
\end{remark}
To proceed we need the auxiliary function
\begin{equation}
\label{eq:definition.Phi}
\Phi(x)=\begin{cases}
1&\text{if }|x|<1,\\|x|^{-N-\alpha}&\text{if }|x|>2,
\end{cases}
\end{equation}
which is nonnegative and belongs to $W^{2,1}\cap W^{2,\infty}(\rr^N)$.  Let us obtain some estimates on it.
\begin{lemma}[{{\cite[Lemma 4.9]{AnBr20}}}]\label{lem.GoodTestPhi}
Assume~\eqref{mlas}. There is a constant $\gamma$ (depending on $\Lambda_1$) such that
$$
|\nabla\Phi|\leq \gamma \Phi,\qquad |\Delta \Phi|\leq \gamma \|\Phi\|_{W^{2,\infty}(\rr^N)},\quad\text{and }|\ml \Phi|\leq \gamma \|\Phi\|_{W^{2,\infty}(\rr^N)}.
$$
\end{lemma}

\begin{remark}
 In~\cite{AnBr20} the authors use for the isotropic case $\ml=(-\Delta)^{\alpha/2}$ the above function $\Phi$ that satisfies~$|(-\Delta)^{\frac{\alpha}{2}}\Phi|\leq \gamma \Phi$ for all $\alpha\in(0,2)$. The existence of a function $\Phi$ satisfying this property is not clear for a general anisotropic $\ml$.  However,  the uniqueness proof for $L^1\cap L^\infty$-solutions in~\cite{AnBr20} still works substituting the estimate (28) in that paper by
\begin{equation*}
\begin{split}
\int_{>\frac{1}{\varepsilon}}|u-v|(\ml\Phi)(\varepsilon\cdot)\geq -\gamma\|\Phi\|_{W^{2,\infty}(\rr^N)}\int_{>\frac{1}{\varepsilon}}|u-v|
\end{split}
\end{equation*}
This allows us to obtain estimate (32) in~\cite{AnBr20} with $s=1$ and $\sigma=q$. The proof is then completed without more effort since the integrability of solutions is already assumed in the next theorem.
\end{remark}

\begin{theorem}[{{\cite{AnBr20}}}]\label{thm.UniquenessPropertiesEntropy}
{\rm (a)} Assume $p\in[1,\infty)$, $u_0\in L^\infty(\rr^N)$,~\eqref{qas}, and~\eqref{mlas}. Then there is \emph{at least} one entropy solution $u\in L^\infty(Q)\cap C([0,\infty);L_\textup{loc}^p(\rr^N))$ of~\eqref{eq.main} with initial data $u_0$. Moreover, let $u,\bar{u}$ be entropy solutions of~\eqref{eq.main} with initial data $u_0,\bar{u}_0$, respectively. Then:
\begin{enumerate}[{\rm (i)}]
\item\textup{(Comparison principle)} If $u_0\leq \bar{u}_0$ a.e.~on $\rr^N$, then $u\leq \bar{u}$ a.e.~in $Q$.
\item\textup{($L^1$-contraction)}  If $u_0-\bar{u}_0\in L^1(\rr^N)$, then, for a.e.~$t>0$,
$$
\int_{\rr^N}|u(t)-\bar{u}(t)|\leq \int_{\rr^N}|u_0-\bar{u}_0|.
$$
\item\textup{($L^\infty$-bound)} For a.e.~$(t,x)\in Q$,  $\essinf_{x\in\rr^N}u_0(x)\leq u(t,x)\leq\esssup_{x\in\rr^N}u_0(x)$.
\item\textup{($L^1$-bound)} For a.e.~$t>0$, $\|u(t)\|_{L^1(\rr^N)}\leq \|u_0\|_{L^1(\rr^N)}$.
\item\textup{(Very weak solutions)} For all $\phi\in C_{\rm c}^\infty(\overline{Q})$,
\begin{equation*}
\begin{split}
\iint_Q\Big(u\partial_t\phi+F(u)\cdot\nabla\phi-u(\ml\phi)\Big)+\int_{\rr^N}u_0\phi(0)=0.
\end{split}
\end{equation*}
\item\textup{(Mass conservation)} For a.e.~$t>0$,
$$
\int_{\rr^N}u(t)=\int _{\rr^N}u_0.
$$
\end{enumerate}
\noindent{\rm (b)} Assume, in addition, $u_0\in L^1(\rr^N)\cap L^\infty(\rr^N)$. Then there is \emph{at most} one entropy solution $u$ of~\eqref{eq.main} with initial data $u_0$, and moreover, the solution satisfies $u\in C([0,\infty);L^1(\rr^N))$.
\end{theorem}

\begin{remark}\label{remarkb7-eps}
\begin{enumerate}[{\rm (a)}]
\item All of the above properties also hold for entropy solutions of~\eqref{reg}: In item (v), we have the additional term $\varepsilon\iint_Qu\Delta \phi$ on the left-hand side.
\item The definitions and the  results in Section \ref{sect:conceptofentropy} and Section \ref{sect:knownunique}  hold for any Lipschitz convection nonlinearity, in particular, $f_\eta$ in \eqref{f.eta}. 
\end{enumerate}
\end{remark}

\begin{proof}[Proof of Theorem~\ref{thm.UniquenessPropertiesEntropy}]
Let us comment on the two last items:

\smallskip
\noindent(v) By the convex inequality (or Kato inequality)
$$
\iint_Q\sgn(u-k)\phi(\ml^{> \rho}(u-k)) \geq \iint_Q|u-k|(\ml^{> \rho}\phi)
$$
and the choices $ k\geq \pm \|u\|_{L^\infty(Q)}$ in~\eqref{eq.entropyInequalityVeps.appendix}, the desired result follows.

\smallskip
\noindent(vi) Consider a standard cut-off function $\mathcal{X}_R$ defined in Section~\ref{sec:notation}.
The idea is then to choose $\phi$ as $\mathds{1}_{[0,\tau]}(t)\mathcal{X}_R(x)$ in the very weak formulation. Since $\|\ml\mathcal{X}_R\|_{L^\infty}\leq R^{-\alpha}\|\ml\mathcal{X}\|_{L^\infty}$ by $\alpha$-homogeneity, the gradient and nonlocal terms go to zero as $R\to\infty$.
\end{proof}

\subsection{Concept of classical solutions}\label{sec.Appendix.ConceptOfClassicalSolutions}
We define, for all $\phi\in C_\textup{c}^\infty(\rr^N)$,
\begin{equation*}
\begin{gathered}
(L\phi)(x):=-\varepsilon(\Delta\phi)(x)+(\ml\phi)(x), \quad\text{i.e.,}
\\
\widehat{L\phi}(\xi)=a(\xi)\widehat{\phi}(\xi) \quad\text{with }
a(\xi):=\varepsilon|\xi|^2+\int_{\bs^{N-1}}|\xi\cdot\theta|^\alpha \,{\rm d}\mu(\theta).
\end{gathered}
\end{equation*}
Now, solutions of~\eqref{reg} with initial data $0<\tilde{\varepsilon}\leq u_0\in L^\infty(\rr^N)$ are mild and also classical on $Q$:

\begin{proof}[Proof of Lemma~\ref{lem.ClassicalSolutionsOfReg}]
By the assumption~\eqref{mlas}, we get
$
\varepsilon|\xi|^2\leq a(\xi)\leq  \varepsilon|\xi|^2+\Lambda_2|\xi|^\alpha$.
This means that assumptions (2.8), (2.9), and (2.10) in~\cite{BilerKarchWoyczAsymp} are fulfilled (the latter with $a_0=\varepsilon$ and $\tilde{\alpha}=\alpha$). By \cite[Remark 1.2]{DroniouVanishing2003}, we are in the setting of  \cite[Theorem 1.1]{Droniou}. However, that theorem requires $F\in C^\infty(\rr)$, but, with our assumption on the initial data, $F(u)\in C^\infty((\tilde{\varepsilon}, \|u_0\|_{L^\infty(\rr^N)}))$. 
\end{proof}

We continue by providing a uniform bound for $\nabla u$ (cf. \cite[Lemma 2.1 and Corollary 2.2]{Pan20}).

\begin{lemma}\label{lem.localGradientEstimate}
Assume $p\in[2,\infty)$, $0<\tilde{\varepsilon}\leq u_0\in L^\infty(\rr^N)$,~\eqref{qas}, and~\eqref{mlas}.  Let  $\Phi$ be as in \eqref{eq:definition.Phi} and $0\leq \Theta\in C_\textup{c}^\infty((0,\infty))$. Then the unique mild  solution $u$ of~\eqref{reg} satisfies
$$
\begin{aligned}
&\varepsilon\frac{4(p-1)}{p^2}\iint_Q|\nabla u^{\frac{p}{2}}(t,x)|^2\Theta(t)\Phi(x)\,{\rm d}x{\rm d}t\leq C,\quad\text{where} \\
&C:=\max_{\tilde{\varepsilon}\leq u(x,t)\leq \|u_0\|_{L^\infty}}\bigg\{\frac{1}{p}u^p+\gamma u^{p-1}|F(u)|+\gamma\frac{1}{p}u^p\\
&\qquad+\frac{1}{2}\Big(\varepsilon^{-\frac{1}{2}}(p-1)^{\frac{1}{2}}u^{\frac{p}{2}-1}|F(u)|+\varepsilon^{\frac{1}{2}}\gamma (p-1)^{-\frac{1}{2}}u^{\frac{p}{2}}\Big)^2\bigg\}\bigg(\iint_Q\max\{|\Theta'|,\Theta\}\big(\Phi +|\ml\Phi|\big)\bigg).
\end{aligned}
$$
\end{lemma}

\begin{remark}
Note that indeed when $p=2$ we have, for all compact $K\subset \textup{supp}\,\Phi\times\rr^N$, that
\begin{equation*}
\begin{split}
C&\geq \varepsilon\iint_Q|\nabla u(t,x)|^2\Theta(t)\Phi(x)\,{\rm d}x{\rm d}t\geq \varepsilon\iint_{K}|\nabla u(t,x)|^2\Theta(t)\Phi(x)\,{\rm d}x {\rm d}t\geq
\varepsilon\min_{K}\{\Theta\Phi\}\iint_{K}|\nabla u|^2.
\end{split}
\end{equation*}
\end{remark}

\begin{proof}[Proof of Lemma~\ref{lem.localGradientEstimate}]
Since $u\in C_\textup{b}^\infty(Q)$ is a classical solution of the PDE in~\eqref{reg},
\begin{equation}
\label{eq:pointwise.eq}
\partial_tu+\nabla\cdot F(u)+\ml u-\varepsilon\Delta u=0 \quad\textup{pointwise in }Q.
\end{equation}
Let $0\leq \phi\in C_\textup{c}^\infty(Q)$.  Using Kato's inequality and the identity $\nabla u\cdot\nabla(u^{p-1})=\frac{4(p-1)}{p^2}|\nabla u^{\frac{p}{2}}|^2$,
\begin{equation*}
\begin{split}
\iint_Q(\ml u)u^{p-1}\phi &\geq  \frac{1}{p}\iint_Q(\ml u^p)\phi= \iint_Q\frac{1}{p}u^p(\ml\phi),\\
-\varepsilon\iint_Q(\Delta u)u^{p-1}\phi&=\varepsilon\frac{4(p-1)}{p^2}\iint_Q|\nabla u^{\frac{p}{2}}|^2\phi+\varepsilon\iint_Qu^{p-1}\nabla u\cdot\nabla\phi.
\end{split}
\end{equation*}
Therefore, if we multiply~\eqref{eq:pointwise.eq} by $u^{p-1}\phi$ and integrate over $Q$,
\begin{equation*}
\begin{split}
&\iint_Q\Big(\frac{1}{p}u^p\partial_t\phi+u^{p-1}F(u)\cdot \nabla\phi+\phi F(u)\cdot\nabla u^{p-1}\Big)\\
&\qquad\geq \iint_Q\Big(\varepsilon\frac{4(p-1)}{p^2}|\nabla u^{\frac{p}{2}}|^2\phi+\varepsilon u^{p-1}\nabla u\cdot\nabla\phi+\frac{1}{p}u^p(\ml\phi)\big).
\end{split}
\end{equation*}
Define $I:=\varepsilon\frac{4(p-1)}{p^2}\iint_Q|\nabla u^{\frac{p}{2}}|^2\Theta\Phi$. By the choice $\phi(x,t):=\Theta(t)\Phi(x)$, the identities
$$
\nabla u^{p-1}=(p-1)^{\frac{1}{2}}u^{\frac{p}{2}-1}\frac{2(p-1)^{\frac{1}{2}}}{p}\nabla u^{\frac{p}{2}} \quad\textup{and }\quad u^{p-1}\nabla u=(p-1)^{-\frac{1}{2}}u^{\frac{p}{2}}\frac{2(p-1)^{\frac{1}{2}}}{p}\nabla u^{\frac{p}{2}},
$$
and the regularity $\nabla u\in L^\infty(Q)$, we get
\begin{equation*}
\begin{split}
I&\leq \iint_Q\Big(\Phi\frac{1}{p}u^p\partial_t\Theta+\Theta u^{p-1}F(u)\cdot \nabla\Phi+\Theta\Phi(p-1)^{\frac{1}{2}}u^{\frac{p}{2}-1} F(u)\cdot\Big(\frac{2(p-1)^{\frac{1}{2}}}{p}\nabla u^{\frac{p}{2}}\Big)\\
&\qquad-\Theta\frac{1}{p}u^p\big(\ml\Phi\big)- \varepsilon\Theta (p-1)^{-\frac{1}{2}}u^{\frac{p}{2}}\Big(\frac{2(p-1)^{\frac{1}{2}}}{p}\nabla u^{\frac{p}{2}}\Big)\cdot\nabla\Phi\Big)\\
&\leq \iint_Q\Big(\Phi\frac{1}{p}u^p|\Theta'|+\Theta u^{p-1}|F(u)||\nabla\Phi|+\Theta\Phi(p-1)^{\frac{1}{2}}u^{\frac{p}{2}-1}|F(u)|\Big|\frac{2(p-1)^{\frac{1}{2}}}{p}\nabla u^{\frac{p}{2}}\Big|\\
&\qquad+\Theta\frac{1}{p}u^p\big|\ml\Phi\big|+\varepsilon\Theta (p-1)^{-\frac{1}{2}}u^{\frac{p}{2}}\Big|\frac{2(p-1)^{\frac{1}{2}}}{p}\nabla u^{\frac{p}{2}}\Big||\nabla\Phi|\Big).
\end{split}
\end{equation*}
Using now Lemma~\ref{lem.GoodTestPhi},
\begin{equation*}
\begin{split}
I&\leq \iint_Q\Big(\Phi\frac{1}{p}u^p|\Theta'|+\Theta u^{p-1}|F(u)|\gamma\Phi+\Theta\Phi(p-1)^{\frac{1}{2}}u^{\frac{p}{2}-1}|F(u)|\frac{2(p-1)^{\frac{1}{2}}}{p}|\nabla u^{\frac{p}{2}}|\\
&\qquad+\Theta\frac{1}{p}u^p|\ml\Phi|+\varepsilon\Theta (p-1)^{-\frac{1}{2}}u^{\frac{p}{2}}\frac{2(p-1)^{\frac{1}{2}}}{p}|\nabla u^{\frac{p}{2}}|\gamma\Phi\Big)\\
&=\iint_Q\bigg(\Phi\max\{|\Theta'|,\Theta\}\Big(\frac{1}{p}u^p+\gamma u^{p-1}|F(u)|+\gamma\frac{1}{p}u^p\Big)\\
&\qquad+(\varepsilon\Theta\Phi)^{\frac{1}{2}}\frac{2(p-1)^{\frac{1}{2}}}{p}|\nabla u^{\frac{p}{2}}|\Big(\varepsilon^{-\frac{1}{2}}(\Theta\Phi)^{\frac{1}{2}}(p-1)^{\frac{1}{2}}u^{\frac{p}{2}-1}|F(u)|
+(\varepsilon\Theta\Phi)^{\frac{1}{2}}\gamma (p-1)^{-\frac{1}{2}}u^{\frac{p}{2}}\Big)\bigg)\\
&\qquad+\iint_Q\Theta\frac{1}{p}u^p|\ml\Phi|.
\end{split}
\end{equation*}
Now, apply Young's inequality $ab\leq \frac{1}{2}a^2+\frac{1}{2}b^2$ to obtain
\begin{equation*}
\begin{split}
\frac{1}{2}I&\leq \iint_Q\bigg(\Phi\max\{|\Theta'|,\Theta\}\Big(\frac{1}{p}u^p+\gamma u^{p-1}|F(u)|+\gamma\frac{1}{p}u^p\Big)\\
&\qquad+\frac{1}{2}\Big(\varepsilon^{-\frac{1}{2}}(\Theta\Phi)^{\frac{1}{2}}(p-1)^{\frac{1}{2}}u^{\frac{p}{2}-1}|F(u)|+(\varepsilon\Theta\Phi)^{\frac{1}{2}}\gamma (p-1)^{-\frac{1}{2}}u^{\frac{p}{2}}\Big)^2\bigg)+\iint_Q\Theta\frac{1}{p}u^p|\ml\Phi|\\
&\leq \iint_Q\Phi\max\{|\Theta'|,\Theta\}\Big(\frac{1}{p}u^p+\gamma u^{p-1}|F(u)|+\gamma\frac{1}{p}u^p\\
&\qquad+\frac{1}{2}\Big(\varepsilon^{-\frac{1}{2}}(p-1)^{\frac{1}{2}}u^{\frac{p}{2}-1}|F(u)|+\varepsilon^{\frac{1}{2}}\gamma (p-1)^{-\frac{1}{2}}u^{\frac{p}{2}}\Big)^2\Big)+\iint_Q\Theta\frac{1}{p}u^p|\ml\Phi|.
\end{split}
\end{equation*}
Taking the maximum over the range of $u$ and multiplying by $2$ gives the final estimate.
\end{proof}

To write down our much needed stability result, we define, for all $\tilde{\varepsilon}>0$,
\begin{equation}\label{eq.LiftedInitialData}
L^\infty(\rr^N)\ni u_{0,\tilde{\varepsilon}}:=u_0+\tilde{\varepsilon} \qquad\textup{with $u_0$ satisfying~\eqref{u_0as}.}
\end{equation}
Now, since mild solutions are classical solutions and then entropy solutions of~\eqref{reg}, they enjoy all the properties of Theorem~\ref{thm.UniquenessPropertiesEntropy}. We then get the following convergence result.

\begin{proposition}\label{prop.PointwiseStabilityOfEntropySolutions}
Assume~\eqref{qas} and~\eqref{mlas}. Let $u_{\tilde{\varepsilon}}$ be the unique mild solution of~\eqref{reg} with initial data $u_{0,\tilde{\varepsilon}}$ defined in~\eqref{eq.LiftedInitialData}. Then there exists a function $u\in L^\infty(Q)$ such that
$$
\begin{aligned}
&u_{\tilde{\varepsilon}}\to u\quad\text{a.e.~in }Q\text{ as }\tilde{\varepsilon}\to0^+,\\
&\nabla u_{\tilde{\varepsilon}}^{\frac{p}{2}}\rightharpoonup \nabla u^{\frac{p}{2}}\quad\text{in }L_\textup{loc}^2(Q)\text{ as } \tilde{\varepsilon}\to0^+ \quad\text{for all }p\in[2,\infty).
\end{aligned}
$$
Moreover, $u$ is an entropy solution of~\eqref{reg} with initial data $u_{0}$ satisfying~\eqref{u_0as}.
\end{proposition}

\begin{proof}
If  $0\leq \tilde{\varepsilon}_1\leq \tilde{\varepsilon}_2$, then
$0\leq u_0\leq u_{0,\tilde{\varepsilon}_1}\leq u_{0,\tilde{\varepsilon}_2}$, whence the comparison principle (see Theorem~\ref{thm.UniquenessPropertiesEntropy}) implies that $0\leq u_{\tilde{\varepsilon}_1}\leq u_{\tilde{\varepsilon}_2}$.
Therefore,  the sequence $\{u_{\tilde{\varepsilon}}(t,x)\}_{\tilde{\varepsilon}>0}\subset \rr$
is nonincreasing and uniformly bounded from below by $0$. Thus, there exists a function $u$ such that
$u_{\tilde{\varepsilon}}\to u$ a.e.~in $Q$ as $\tilde{\varepsilon}\to0^+$.
Moreover, if we assume $0<\tilde{\varepsilon}\leq 1$, then
$0\leq u\leq u_{\tilde{\varepsilon}}\leq \|u_0\|_{L^\infty(\rr^N)}+\tilde{\varepsilon}\leq \|u_0\|_{L^\infty(\rr^N)}+1$,
and, by Lemma~\ref{lem.localGradientEstimate}, $\nabla u_{\tilde{\varepsilon}}^{\frac{p}{2}}\in L_\textup{loc}^2(Q)$ uniformly in $\tilde{\varepsilon}$, which yields
$$
\nabla u_{\tilde{\varepsilon}}^{\frac{p}{2}}\rightharpoonup h \quad\textup{in }L_\textup{loc}^2(Q)\text{ as }\tilde{\varepsilon}\to0^+.
$$
Since weak limits are distributional limits, and distributional limits are unique, we can identify the limit by considering $\phi\in C_{\textup{c}}^\infty(Q)$ such that
$$
\iint_Q\nabla u_{\tilde{\varepsilon}}^{\frac{p}{2}}\phi=-\iint_Q u_{\tilde{\varepsilon}}^{\frac{p}{2}}\nabla\phi.
$$
Now, $\phi\in C_\textup{c}^\infty(Q)\subset L^2(Q)$, $u_{\tilde{\varepsilon}}\to u$ a.e.~as ${\tilde{\varepsilon}}\to0^+$, and $0\leq u_{\tilde{\varepsilon}}(t,x)\leq \|u_0\|_{L^\infty(\rr^N)}+1$ give
$$
\iint_Qh\phi=-\iint_Q u^{\frac{p}{2}}\nabla\phi,
$$
i.e., by the definition of weak derivatives, $h=\nabla u^{\frac{p}{2}}$ in $L_\textup{loc}^2(Q)$.

Finally, since $u_{\tilde{\varepsilon}}$ is an entropy solution with initial data $u_{0,\tilde{\varepsilon}}$ in the sense of Definition~\ref{def.entropySolutionVeps}, we obtain that $u$ is an entropy solution with initial data $u_{0}$ in the sense of Definition~\ref{def.entropySolutionVeps} by simply taking the a.e.-limit as $\tilde{\varepsilon}\to0^+$ in the entropy inequality~\eqref{eq.entropyInequalityVeps.appendix}.
\end{proof}

\subsection{Entropy solutions enjoy an energy estimate}
The goal of this section is to show that entropy solutions $u$ of~\eqref{reg} with initial data $0\leq u_0\in L^1(\rr^N)\cap L^\infty(\rr^N)$ indeed satisfy the $L^p$-energy inequality with $p\in[2,\infty)$: for a.e.~$t>0$,
\begin{equation}\label{eq.L2Energy}
\begin{split}
&\frac{1}{p}\|u(t)\|^p_{L^p(\rr^N)}+\Lambda_1 \frac{4(p-1)}{p^2}
\int_0^t\int_{\rr^N} |(-\Delta)^{\alpha/4} u^{\frac{p}{2}}|^2+\varepsilon\frac{4(p-1)}{p^2}\int_0^t\int_{\rr^N}|\nabla u^{\frac{p}{2}}|^2\leq \frac{1}{p}\|u_0\|_{L^p(\rr^N)}.
\end{split}
\end{equation}
Hence, entropy solutions $u$ of~\eqref{reg} belong to $L^2((0,\infty);H^1(\rr^N)\cap H^{\alpha/2}(\rr^N))$; Lemma~\ref{parabolic.estimates}(b) follows.

For the cut-off function $\mathcal{X}_R$ defined in Section~\ref{sec:notation}, we will basically choose $u_{\tilde{\varepsilon}}\mathcal{X}_R$ as a test function, and then take the limit as $\tilde{\varepsilon}\to0^+$ in the resulting estimate.

\begin{lemma}\label{lem.approxL2Energy}
Assume $p\in[2,\infty)$, $u_{0,\tilde{\varepsilon}}$ as in~\eqref{eq.LiftedInitialData},~\eqref{qas}, and~\eqref{mlas}. Then the unique mild solution $u_{\tilde{\varepsilon}}$ of~\eqref{reg} with initial data $u_{0,\tilde{\varepsilon}}$ satisfies
\begin{equation}\label{eq.approxL2Energy}
\begin{split}
&\frac{1}{p}\int_{\rr^N}u_{\tilde{\varepsilon}}(t)^p\mathcal{X}_R+\int_0^t\mathcal{E}(u_{\tilde{\varepsilon}}(s),u_{\tilde{\varepsilon}}^{p-1}(s)\mathcal{X}_R)\, {\rm d}s+\varepsilon\frac{4(p-1)}{p^2}\int_0^t\int_{\rr^N}|\nabla u_{\tilde{\varepsilon}}^{\frac{p}{2}}|^2\mathcal{X}_R\\
&\qquad = \frac{q}{q+p-1}\int_0^t\int_{\rr^N}u_{\tilde{\varepsilon}}^{q+p-1}\partial_{x_N}\mathcal{X}_R+\frac{1}{p}\int_0^t\int_{\rr^N}u_{\tilde{\varepsilon}}^{p}\Delta\mathcal{X}_R
+\frac{1}{p}\int_{\rr^N}u_{0,\tilde{\varepsilon}}^p\mathcal{X}_R,
\end{split}
\end{equation}
where $\mathcal{E}$ is the bilinear form associated with $\ml$ defined in Section~\ref{sec:notation}.
\end{lemma}

\begin{proof}
We ease the notation by writing $u,u_0$ for $u_{\tilde{\varepsilon}},u_{0,\tilde{\varepsilon}}$. Since $u\in C_\textup{b}^\infty(Q)$ is a classical solution of the PDE in~\eqref{reg}, we simply multiply it by $u^{p-1}\mathcal{X}_R$ and integrate over $(0,t)\times\rr^N$. Let us calculate term by term keeping in mind that $\mathcal{X}_R\in C_\textup{c}^\infty$ and e.g.~Proposition 4.1 in~\cite{DTEnJa18a} to handle the nonlocal term:
\begin{equation*}
\begin{aligned}
\int_0^t\int_{\mathbb{R}^N}(\partial_tu)u^{p-1}\mathcal{X}_R&=\frac{1}{p}\int_{\rr^N}u(t)^p\mathcal{X}_R-\frac{1}{p}\int_{\rr^N}u_0^p\mathcal{X}_R,\\
\int_0^t\int_{\mathbb{R}^N}(\nabla\cdot F(u))u^{p-1}\mathcal{X}_R
&=\frac{q}{q+p-1}\int_0^t\int_{\mathbb{R}^N}(\nabla\cdot \mathbf{a}u^{q+p-1})\mathcal{X}_R\\
&=\frac{q}{q+p-1}\int_0^t\int_{\mathbb{R}^N} \nabla(u^{q+p-1}\mathbf{a}\mathcal{X}_R)
-\frac{q}{q+p-1}\int_0^t\int_{\mathbb{R}^N} u^{q+p-1}\mathbf{a}\cdot\nabla\mathcal{X}_R\\
&=-\frac{q}{q+p-1}\int_0^t\int_{\mathbb{R}^N} u^{q+p-1}\mathbf{a}\cdot\nabla\mathcal{X}_R,\\
\int_0^t\int_{\mathbb{R}^N}(\ml u)u^{p-1}\mathcal{X}_R&=\int_0^t\mathcal{E}(u(s),u^{p-1}(s)\mathcal{X}_R)\,{\rm d}s,
\\
-\varepsilon\int_0^t\int_{\mathbb{R}^N}(\Delta u)u^{p-1}\mathcal{X}_R &= \varepsilon\frac{4(p-1)}{p^2}\int_0^t\int_{\mathbb{R}^N}|\nabla u^{\frac{p}{2}}|^2\mathcal{X}_R+\varepsilon\int_0^t\int_{\mathbb{R}^N}\nabla\Big(\frac{1}{p}u^p\Big)\cdot\nabla\mathcal{X}_R\\
&= \varepsilon\frac{4(p-1)}{p^2}\int_0^t\int_{\mathbb{R}^N}|\nabla u^{\frac{p}{2}}|^2\mathcal{X}_R -\varepsilon\int_0^t\int_{\mathbb{R}^N}\frac{1}{p}u^p\Delta\mathcal{X}_R.
\end{aligned}
\end{equation*}
Adding all the above terms finishes the proof.
\end{proof}

The estimate is indeed transferred to the limit $u$ as $\tilde{\varepsilon}\to0^+$.

\begin{lemma}\label{lem.L2Energy}
Assume~\eqref{u_0as},~\eqref{qas}, and~\eqref{mlas}. Then the unique entropy solution $u$ of~\eqref{reg}  satisfies
\begin{equation*}
\begin{split}
&\frac{1}{p}\int_{\rr^N}u^p(t)\mathcal{X}_R+\int_0^t\mathcal{E}(u(s),u^{p-1}(s)\mathcal{X}_R)\, {\rm d}s+\varepsilon\frac{4(p-1)}{p^2}\int_0^t\int_{\rr^N}|\nabla u^{\frac{p}{2}}|^2\mathcal{X}_R\\
&\qquad\leq \frac{q}{q+p-1}\int_0^t\int_{\rr^N}u^{q+p-1}\partial_{x_N}\mathcal{X}_R+\frac{1}{p}\int_0^t\int_{\rr^N}u^{p}\Delta\mathcal{X}_R+\frac{1}{p}\int_{\rr^N}u_0^p\mathcal{X}_R,
\end{split}
\end{equation*}
and moreover, as $R\to\infty$ in the above estimate, we obtain~\eqref{eq.L2Energy}.
\end{lemma}

\begin{proof}
We have already proved that the mild solution $u_{\tilde{\varepsilon}}$ of~\eqref{reg} with initial data $u_{0,\tilde{\varepsilon}}$ satisfies~\eqref{eq.approxL2Energy}. By Proposition~\ref{prop.PointwiseStabilityOfEntropySolutions}, there exists an entropy solution $u$ with initial data $u_0$ such that, as $\tilde{\varepsilon}\to0^+$,
$$
u_{\tilde{\varepsilon}}\to u\quad\textup{a.e.~in $Q$},\qquad\nabla u_{\tilde{\varepsilon}}^{\frac{p}{2}}\rightharpoonup \nabla u^{\frac{p}{2}} \quad\textup{in }L_\textup{loc}^2(Q).
$$
The first two terms on the left-hand side of~\eqref{eq.approxL2Energy} directly converge when applying Fatou's  lemma and the a.e.-convergence as $\tilde{\varepsilon}\to0^+$. Since $x\mapsto|x|^2$ is convex, we have
$$
\int_0^t\int_{\rr^N}|\nabla u_{\tilde{\varepsilon}}^{\frac{p}{2}}|^2\mathcal{X}_R\geq \int_0^t\int_{\rr^N}|\nabla u^{\frac{p}{2}}|^2\mathcal{X}_R+2\int_0^t\int_{\rr^N}\nabla u^{\frac{p}{2}}\cdot\big(\nabla u_{\tilde{\varepsilon}}^{\frac{p}{2}}-\nabla u^{\frac{p}{2}}\big)\mathcal{X}_R,
$$
where the latter term goes to zero by the weak convergence of the gradients in $L_\textup{loc}^2(Q)$ as $\tilde{\varepsilon}\to0^+$. The right-hand side of~\eqref{eq.approxL2Energy}  converges  through the use of the Lebesgue's  dominated convergence theorem.

Recall that $\mathcal{X}_R\to1$ a.e.~in $\rr^N$ as $R\to\infty$. Hence, the left-hand side of the resulting inequality at the limit $\tilde{\varepsilon}\to0^+$ converges as $R\to\infty$ by Fatou's lemma. To handle the right-hand side, we use the fact  that
$$
\partial_{x_N}\mathcal{X}_R=\frac{1}{R}\partial_{x_N}\mathcal{X},\qquad \Delta\mathcal{X}_R=\frac{1}{R^2}\Delta\mathcal{X},
$$
together with the boundedness and integrability of $u$ and $u_0$, to obtain the limit as $R\to\infty$ through Lebesgue's dominated convergence theorem.

To rewrite the nonlocal term in a convenient form, we use the general Stroock-Varopoulos inequality (cf. e.g. \cite[Lemma 4.10]{DTEnJa18a}) and assumption~\eqref{mlas} to get
\begin{equation*}
\mathcal{E}(u(t),u^{p-1}(t))\geq \frac{4(p-1)}{p^2}  \int_{\rr^N} |(\ml)^\frac{1}{2} u^{\frac{p}{2}}(t)|^2\geq \Lambda_1 \frac{4(p-1)}{p^2}
\int_{\rr^N} |(-\Delta)^{\alpha/4} u^{\frac{p}{2}}(t) |^2.\qedhere
\end{equation*}
\end{proof}

\begin{remark}
When considering~\eqref{reg'} instead of~\eqref{reg}, i.e., $\ml'$ instead of $\ml$, we need to be careful with the application of the general Stroock-Varopoulos inequality. Indeed, we can only get control of the seminorm associated with the space $L^2(\rr; H^{\alpha/2 }(\rr^{N-1}))$ defined in Appendix~\ref{sec:CompactEmbeddingOfMixedSpaces}.
\end{remark}

\begin{proof}[Proof of Lemma~\ref{lem.EntropySolutionsOfReg}(c)]
Since $u$ is an entropy solution, it is then also a very weak solution (cf. Theorem~\ref{thm.UniquenessPropertiesEntropy}(a)(v)). By the energy estimate~\eqref{eq.L2Energy} with $p=2$, for all $\phi\in C_\textup{c}^\infty(Q)$,
\begin{equation*}
\begin{split}
\iint_Q\Big(u\partial_t\phi+u^q\partial_{x_N}\phi-\varepsilon \nabla u\cdot \nabla \phi\Big)+\int_0^\infty\mathcal{E}(u(s),\phi(s))\,{\rm d}s=0.
\end{split}
\end{equation*}
Let us choose a smooth bounded domain $\Omega$ and $\phi$ supported in $(0,T)\times\Omega$.
When $q\geq 1/2$ we have that $u^q\in L^2(\rr^N)$ and then the regularity of $\partial_tu$. When $q<1/2$ we choose a large $p$ such that $p\geq 1/q\geq 2\geq p'$ and again we obtain the desired estimate.
\end{proof}

\subsection{Compactness in $C([0,\infty);L_{\textnormal{\textup{loc}}}^p(\rr^N))$}\label{sec.CompactnessCLloc1}
The initial data $u_0$ is assumed to satisfy
\begin{equation}\label{eq.InitialDataForLloc1Convergence}
0\leq u_0\in L^\infty(\rr^N),\qquad\int_{\rr^N}|u_0(x+\xi)-u_0(x)|\,{\rm d}x\to0\quad\textup{as }\xi\to0^+.
\end{equation}

With the previous results in mind, it is known that a combination of the Arzel\`a-Ascoli and Fr\'echet-Kolmogorov-Riesz compactness theorems gives convergence in $C([0,\infty);L_{\textnormal{\textup{loc}}}^1(\rr^N))$ as long as a uniform control for the time- and space-translations are given on compact sets.

Due to the $L^1$-contraction and the translation invariance of~\eqref{reg}, space-translations are easily controlled: for all compact $K\subset \rr^N$, and all $t>0$,
\begin{equation*}
\begin{split}
\int_K|u(t,x+\xi)-u(t,x)|\,{\rm d}x\leq \int_{\rr^N}|u_0(x+\xi)-u_0(x)|\,{\rm d}x.
\end{split}
\end{equation*}
It turns out that the $L^1$-contraction also provides control of the time-translations:

\begin{proposition}\label{time.compactness}
Assume $0<\varepsilon\leq 1$,~\eqref{qas}, and~\eqref{mlas}. Let $u$ be an entropy solution of~\eqref{reg} with initial data $u_0$ satisfying~\eqref{eq.InitialDataForLloc1Convergence}. Then, for all compact $K\subset \rr^N$ and all $t,s\in [0,T]$,
$$
\int_K|u(t)-u(s)|\leq \lambda(|t-s|^{\frac{1}{3}}),
$$
where $\lambda$ is a modulus of continuity depending on $K$, $\|u_0\|_{L^\infty(\rr^N)}$, and $\|u_0(\cdot+\xi)-u_0\|_{L^1(\rr^N)}$.
\end{proposition}

\begin{proof}
Assume $s<t$. Define $u_\delta(t,x):=u(t,\cdot)\ast\omega_\delta(x)$, where $\omega_\delta$ is a standard mollifier. By the triangle inequality and the $L^1$-contraction,
\begin{equation*}
\begin{split}
\|u(t)-u(s)\|_{L^1(K)}&\leq \|u(t)-u_\delta(t)\|_{L^1(K)}+\|u_\delta(t)-u_\delta(s)\|_{L^1(K)}+\|u_\delta(s)-u(s)\|_{L^1(K)}\\
&\leq 2\sup_{|\xi|\leq \delta}\|u_0(\cdot+\xi)-u_0\|_{L^1(\rr^N)}+\|u_\delta(t)-u_\delta(s)\|_{L^1(K)}.
\end{split}
\end{equation*}
Now, it is standard to choose $(\tau,y)\mapsto\mathds{1}_{[s,t]}(\tau)\omega_\delta(x-y)$ as a test function in the very weak formulation of~\eqref{reg} to obtain
\begin{equation*}
\begin{split}
\|u_\delta(t)-u_\delta(s)\|_{L^1(K)}&\leq \int_K\int_s^t\big|u(\tau)^q\ast\partial_{x_N}\omega_\delta-u(\tau)\ast(\ml\omega_\delta)+\varepsilon u(\tau)\ast\Delta\omega_\delta\big|\,{\rm d}\tau\\
&\leq |t-s||K|\Big(\|u_0\|_{L^\infty(\rr^N)}^q\|\partial_{x_N}\omega_\delta\|_{L^1(\rr^N)}+\|u_0\|_{L^\infty(\rr^N)}\|\ml\omega_\delta\|_{L^1(\rr^N)}\\
&\quad+\varepsilon\|u_0\|_{L^\infty(\rr^N)}\|\Delta\omega_\delta\|_{L^1(\rr^N)}\Big).
\end{split}
\end{equation*}
Since
$$
\|\partial_{x_N}\omega_\delta\|_{L^1(\rr^N)}\lesssim \delta^{-1},\quad \|\ml\omega_\delta\|_{L^1(\rr^N)}\lesssim \delta^{-\alpha},\quad\textup{and }\|\Delta\omega_\delta\|_{L^1(\rr^N)}\lesssim\delta^{-2},
$$
we can choose $\delta^2=|t-s|^{\frac{2}{3}}$ to obtain the desired result.
\end{proof}

 In view of the previous discussion on the application of Arzelà-Ascoli and Fréchet-Kolmogorov-Riesz compactness theorems, we obtain Lemma~\ref{lem.StabilityOfSolutionsOfReg}.

\subsection{Compactness in $C([0,\infty);L^p(\rr^N))$}\label{sec.Appendix.ConvInCL1}
This follows from convergence in $C([0,\infty);L^p_{\textup{loc}}(\rr^N))$ and a tail estimate; see e.g.~\cite{DTEnJa22}.  Hence we start by proving~Lemma~\ref{lem.EntropySolutionsOfReg}(d).

\begin{proof}[Proof of Lemma~\ref{lem.EntropySolutionsOfReg}(d)]
We take (after a standard approximation procedure) $\mathds{1}_{[0,t]}\rho_R(x)$ as a test function in the very weak formulation of~\eqref{reg}, where $\rho_R$ is the tail-control function defined in Section~\ref{sec:notation}. Then,
\begin{equation*}
\begin{split}
\int_{\rr^N}u(t)\rho_R=\int_{\rr^N}u_0\rho_R+\int_0^{t}\int_{\rr^N}\big(F(u)\cdot\nabla\mathcal{\rho}_R-u(\ml\mathcal{\rho}_R)+\varepsilon u\Delta\mathcal{\rho}_R\big).
\end{split}
\end{equation*}
Using that $\mathds{1}_{|x|>2R}\leq\rho_R(x)\leq \mathds{1}_{|x|>R}$ and $u,u_0\geq 0$, we get
\begin{equation*}
\begin{split}
\int_{|x|>2R}u(t)\le\int_{|x|>R}u_0+\int_0^{t}\int_{\rr^N}\big(u^q|\partial_{x_N}\rho_R|+u|\ml\rho_R|+\varepsilon u|\Delta\rho_R|\big).
\end{split}
\end{equation*}
Since $\|\Delta\rho_R\|_{L^\infty(\rr^N)}\lesssim R^{-2}$, the result follows.
\end{proof}

This preliminary tail control provided by Lemma~\ref{lem.EntropySolutionsOfReg}(d) is not enough for our purposes, since it is not uniform in $\eps$. In order to improve it
we use the $\alpha$-homogeneity of the operator~$\ml$, which  yields $\|\ml\rho_R\|_{L^\infty(\rr^N)}\lesssim R^{-\alpha}$, and observe also that
\begin{equation}\label{eq:TailEstimateForAllq}
\int_0^{t}\int_{\rr^N}u^q|\partial_{x_N}\rho_R|\lesssim
\begin{cases}
t\|u_0\|_{L^\infty(\rr^N)}^{q-1}\|u_0\|_{L^1(\rr^N)}R^{-1}&\textup{if }q>1,\\
t\|u_0\|_{L^1(\rr^N)}^qR^{-(1-N(1-q))}&\textup{if }1-\frac{1}{N}<q<1.
\end{cases}
\end{equation}
In the second estimate we used H\"older's inequality. Tail control is thus established uniformly in $\eps$ in the interval $[0,t]$ for all  $t>0$,  which is enough to prove Lemmas~\ref{lem.StabilityOfSolutionsOfReg2} and~\ref{lem.StabilityOfSolutionsOfReg3}.

\begin{remark}
Since the sequences above are always uniformly bounded, the mentioned convergence results are actually in $C([0,\infty);L_{\textup{loc}}^p(\rr^N))$ and $C([0,\infty);L^p(\rr^N))$, respectively, for all $p\in[1,\infty)$.
\end{remark}

\subsection{Perturbation of the convection nonlinearity} \label{sect:StabilityOfperturbationofnonlinearity}

In the proof of Theorem \ref{thm.UniquenessOfLimitEquationsIandII}, we need stability of solutions with respect to the convection nonlinearity. It is contained in the following lemma:

\begin{lemma}\label{lem.StabilityOfperturbationofnonlinearity}
Assume~\eqref{u_0as},~\eqref{qas}, ~\eqref{mlas}, $\eps>0$, $\eta>0$, $q<1$, and $f_\eta$ given by \eqref{f.eta}. Let $u^\eps$ be the entropy solution of
	\begin{equation}
\label{p.eps.2}
  \partial_tu+\ml u+\partial_{x_N}(f_0(u)) =\eps \Delta u \quad\textup{in }Q,\qquad
  	u(0)=u_0 \quad\textup{on }\rr^N,
\end{equation}
and $u^{\eps,\eta}$
 the entropy solution of
 \[
	\partial_tu+\ml u+\partial_{x_N}(f_\eta(u)) =\eps \Delta u \quad\textup{in }Q,\qquad
  	u(0)=u_0 \quad\textup{on }\rr^N.
\]
Then,
$$
u^{\eps,\eta}\to u^\eps \quad\text{in } C([0,\infty);L^p(\rr^N))\text{ as }
\eta\to0.
$$
\end{lemma}

\begin{proof}
	For simplicity we will drop the dependence on $\eps$. The $L^1(\rr^N)$-contraction property in Theorem~\ref{thm.UniquenessPropertiesEntropy} and Proposition \ref{time.compactness} applied to the sequence $\{u^{\eta}\}_{\eta>0}$ shows that   it is relatively compact in $C([0,\infty);L^1_{\textup{loc}}(\rr^N))$. The tail control of that sequence follows by    Lemma~\ref{lem.EntropySolutionsOfReg}(d) (see the previous section). It implies the compactness in $C([0,\infty);L^1(\rr^N))$. We show that any limit point $u$ of the sequence $\{u^\eta\}_{\eta>0}$ is an entropy solution of~\eqref{p.eps.2}. It is immediate that $u$ satisfies~\eqref{eq:initial.data}.
Using the same arguments as in step (ii) of the proof of Theorem \ref{asymptotic} when $q\geq q_*(\alpha)$, we can obtain that   $u$  satisfies~\eqref{eq:entropy.inequality}, with the difference that we have to prove that for all $k\in \rr$ and all $0\leq  \phi\in C_{\textup{c}}^\infty(\bar{Q})$,
\[
	\iint_Q \sgn(u^\eta-k)\big(f_\eta(u^\eta )-f_\eta(k)\big)\cdot\nabla\phi \rightarrow
		\iint_Q \sgn(u-k)\big(f_0(u )-f_0(k)\big)\cdot\nabla\phi .
\]
Up to a subsequence we can assume that $u^\eta\rightarrow u$ a.e.~in $Q$. By the triangle inequality and Lebesgue's dominated convergence theorem  it remains to check that $f_\eta(u^\eta)-f(u^\eta)\rightarrow 0$ and $f_0(u^\eta)-f_0(u)\rightarrow 0$ a.e. in $Q$ as $\eta\rightarrow 0$. The second one trivially follows since $f_0$ is continuous and since $\{u^\eta\}_{\eta>0}$ is uniformly bounded. The first one will follow once we prove that $f_\eta\rightarrow f_0$ uniformly as $\eta\rightarrow 0$. This is a consequence of
		\[
		|f_\eta(u)-f_0(u)|=|(u^2+\eta)^{q/2}-\eta^{q/2}-u^q|\lesssim \eta^{q/2}\quad\text{for all $u\geq 0$ and all $\eta>0$.}
		\]
Since the limit equation has a unique entropy solution by Theorem \ref{thm.UniquenessPropertiesEntropy} and Remark \ref{remarkb7-eps}, the whole sequence $\{u^\eta\}_{\eta>0}$ converges to $u$, not only up to a subsequence.
\end{proof}

\subsection{Primitives of entropy solutions solve parabolic problems}
\label{sec:PrimitivesOfEntropySolutions}

We prove now Lemma~\ref{lem.PrimitiveSolvesParabolic}, needed  for the proofs of Theorem~\ref{thm.UniquenessOfLimitEquationsIandII} and Lemma~\ref{estimari.hiperbolice}.

\begin{proof}[Proof of Lemma~\ref{lem.PrimitiveSolvesParabolic}]
Since $u$ is an entropy solution of~\eqref{reg},
\[
\iint_Q\Big(u\partial_t\phi+u^q\partial_{x_N}\phi-u(\ml\phi+\eps \Delta \phi)\Big)+\int_{\rr^N}u_0\phi(0)=0\quad\text{for all }\phi\in C_{\textup{c}}^\infty(\overline{Q}).
\]
We then choose $\phi_R(t,x',x_N)=\psi (t,x')\mathcal{X}_R(x_N)$ with $\psi\in C_{\textup{c}}^{\infty}([0,\infty)\times\rr^{N-1})$ and $\mathcal{X}_R$ the usual cut-off function.  Since $u\in C([0,\infty);L^1(\rr^{N}))\cap L^\infty(Q)$, the term involving $u^q$ can be estimated similarly as in~\eqref{eq:TailEstimateForAllq}, and we may apply the dominated convergence theorem to obtain

\begin{align*}
&\lim_{R\to\infty}\bigg(\iint_Q\big(u\partial_t\phi_R+u^q\partial_{x_N}\phi_R-u(\ml\phi_R+\eps \Delta \phi_R)\big)+\int_{\rr^N}u_0\phi_R(0)\bigg)\\
&\qquad= \int_0^T\int_{\rr^{N-1}}\big(v\partial_t\psi -v(\widetilde \ml\psi+\eps \Delta_{x'} \psi)\big)+\int_{\rr^{N-1}}v_0\psi(0)
\end{align*}
since (see Remark~\ref{symbols.3})
\[
\lim _{R\rightarrow \infty}\ml (\psi \mathcal{X}_R)(x',x_N)=\ml (\psi(\cdot') \mathds{1}_{\mathbb{R}}(\cdot_N))(x',x_N)= (\widetilde \ml\psi)(x')\mathds{1}_{\mathbb{R}}(x_N)\quad\text{for any }\psi\in C_\textup{c}^\infty(\rr^{N-1}). \qedhere
\]
\end{proof}

\begin{remark}\label{extended:lem.PrimitiveSolvesParabolic}
	The results in Section \ref{sec.CompactnessCLloc1}, Section \ref{sec.Appendix.ConvInCL1} and Section \ref{sec:PrimitivesOfEntropySolutions} hold for $f_\eta$ in \eqref{f.eta} since it satisfies inequality \eqref{ineq.f.eta}.
\end{remark}

\subsection{Parabolic regularization}

We include a variant of Lemma~\ref{lem.ClassicalSolutionsOfReg}, which will be useful in the proof of Theorem~\ref{thm.UniquenessOfLimitEquationsIandII}. It applies in the particular cases $f(u)=u^q$, $q>1$ and $f_\eta(u)=(u^2+\eta )^{q/2}-\eta ^{q/2}$, $q>0$. 

\begin{proposition}[Parabolic regularization]\label{prop.propOfRegularizedEquationSmoothness2}
Assume~\eqref{u_0as},  ~\eqref{mlas},  $f\in C^{1+\beta}(\rr)$ for some $\beta\in (0,1]$ and $f(0)=0$.  Let $u$ be the entropy solution of~\eqref{reg}/\eqref{reg'} with initial data $u_0$.
  Then $u\in C^{1+\delta,2+2\delta}(Q)$ for some $\delta>0$, i.e., $u$ is a classical solution of the PDE in~\eqref{reg}/\eqref{reg'}.
\end{proposition}

\begin{proof}In the first step we prove that
\begin{equation}
	\label{reg.1}
 u\in L^p_\textup{loc}((0,\infty); H^{2,p}(\rr^N))\ \text{and} \ \partial_tu\in L^p_\textup{loc}((0,\infty); L^{p}(\rr^N))\ \text{for all}\ p\in(1,\infty).
 \end{equation}
Let us first observe that $u\in L^\infty(Q)\cap C([0,\infty);L^1(\rr^N))$ is a distributional solution  of
$$
\partial_tu-\eps\Delta u= g
$$
with $g=\ml u +\partial_{x_N}(f(u))$ or $g=\ml' u +\partial_{x_N}(f(u))$. Since $f\in C^1(\rr)$ and $f(0)=0$ we have that $f(u)\in L^p((T',T); L^p(\rr^N))$ for all $p> 1$. This implies that $\partial_{x_N}(f(u)) \in H^{-1,p}(\rr^N)$. Also, in view of~\eqref{ml.lp} and~\eqref{ml'.lp} we have $\ml u,\ml' u\in L^p((T',T); H^{-\alpha,p}(\rr^N))$, hence $g\in L^p((T',T); H^{-\max\{\alpha,1\},p}(\rr^N))$. Regularity results for the heat equation, i.e. Theorem 4 and the last lines of Section 1 in~\cite{MR161099},  give us that $u\in L^p((T',T); H^{2-\max\{\alpha,1\},p}(\rr^N))$.

Assume that $u\in L^p((T',T);H^{\sigma,p}(\rr^N))$ for some $\sigma\in (0,1]$. The fractional chain rule  \cite[Proposition~5.1, p. 112]{MR1766415} gives us that
\[
\|f(u)\|_{H^{\sigma,p}(\rr^N)}\lesssim \|f'(u)\|_{L^\infty(\rr^N)} \|u\|_{H^{\sigma,p}(\rr^N)}.
\]
Since $u\in L^\infty((0,\infty);L^\infty(\rr^N))$ and $f\in C^1(\rr)$ with $f(0)=0$, then $f'(u)\in L^\infty((0,\infty);L^\infty(\rr^N))$, so that
$$
f(u)\in L^p((T',T); H^{\sigma,p}(\rr^N)),\qquad\partial_{x_N}(f(u))
\in L^p((T',T); H^{\sigma-1,p}(\rr^N)).
$$
Therefore, the right-hand-side term $g\in L^p((T',T); H^{\sigma-\max\{\alpha,1\},p}(\rr^N))$.
We set  $\sigma_0=2-\max\{\alpha,1\}>0$. As long as $\sigma_n\leq 1$
we define $\sigma_{n+1}=2+\sigma_n-\max \{1,\alpha\}$ and thus $u\in L^p((T',T); H^{\sigma_{n+1},p}(\rr^N))$. Since the sequence $\{\sigma_n\}_{n\geq 0}$ is increasing we have  for some $n\geq0$, $\sigma_n>1$, hence $u\in L^p((T',T); H^{1,p}(\rr^N))$.

If $\alpha \in (0,1]$ we repeat the above argument with $\sigma =1$ and we obtain that $u\in L^p((T',T); H^{2,p}(\rr^N))$ and then $\partial_tu\in L^p((T',T); L^{p}(\rr^N))$.

Let us now consider the case when $\alpha\in (1,2)$. We know that $ u\in L^p((T',T);H^{\sigma,p}(\rr^N))$ for some $\sigma\geq 1$. It implies that $\partial_{x_N}(f(u))\in L^p((T',T); L^{ p}(\rr^N))$, $\ml u\in L^p((T',T); H^{\sigma -\alpha,p}(\rr^N))$ and
$g$ belongs to $L^p((T',T); H^{\min \{\sigma-\alpha,0\},p}(\rr^N))$. Defining $\sigma_{n+1}=2+\min\{\sigma_n-\alpha, 0\}$ we obtain that
$u\in L^p((T',T); H^{\sigma_{n+1},p}(\rr^N))$. For some $n\geq0$, we have $\sigma_n>\alpha$,  which implies $\sigma_{n+1}=2$ and we obtain the desired property \eqref{reg.1}.

In the second step we will show that for any $0<T'<T$, the function $g$ belongs to $C^{\delta,2\delta}((T',T)\times\rr^N)$ for some $\delta>0$, where $C^{\delta,2\delta}$ are the parabolic H\"older spaces \cite[Section~8.5, p.~117]{MR1406091}. Regularity results for the inhomogeneous heat equation
\cite[Theorem 8.7.3, p.~123]{MR1406091}  applied to $\tilde u=\chi (t) u(t)$,  where $\chi(t)$ is a smooth cut-off function,  show that $u\in C^{1+\delta,2+2\delta}((T',T)\times\rr^N)$, so $u$ is a classical solution.

In what follows,  we will show that  $g$ belongs to $C^{\delta,2\delta}((T',T)\times\rr^N)$ for some positive $\delta$. Let us recall that
 $u\in L^p((T',T);H^{2,p}(\rr^N))$ with $\partial_tu\in L^p((T',T);L^{ p}(\rr^N))$ for all $1<p<\infty$, i.e.
 \[
 u\in L^p((T',T); H^{2,p}(\rr^N)) \cap W^{1,p}((T',T); L^{ p}(\rr^N)).
 \]
 Let us show that the nonlocal term has the required regularity. Denoting $v=\ml u$ or $v=\ml' u$,  we have that
 $v\in L^p((T',T); H^{2-\alpha,p}(\rr^N)) \cap W^{1,p}((T',T);H^{-\alpha,p}(\rr^N))$.
 Therefore,
 $$
 (I-\Delta)^{-\alpha/2}v\in L^p((T',T);H^{2,p}(\rr^N)) \cap W^{1,p}((T',T);L^{ p}(\rr^N)).
 $$
 Using the interpolation results for Sobolev spaces in \cite[Proposition~3.2]{MR2863860} to $(I-\Delta)^{-\alpha/2}v$, we get
 $$
 (I-\Delta)^{-\alpha/2}v\in W^{s,p}((T',T);H^{2(1-s) ,p}(\rr^N))\quad\text{for all }s\in (0,1),
 $$
 and then $v\in W^{s,p}((T',T);H^{2(1-s)-\alpha ,p}(\rr^N))$.
 We emphasize that we have to apply the cited results to $(I-\Delta)^{-\alpha/2}v$ and not directly to $v$ since the results in~\cite{MR2863860} hold for nonnegative index Bessel potential spaces.
The classical Morrey estimates for Bessel potential and Sobolev spaces give us that $v\in C^{\delta}((T',T); C^{2\delta}(\rr^N))\subset C^{\delta,2\delta}((T',T)\times\rr^N)$ provided $s-1/p\geq \delta$ and $2(1-s)- \alpha - N/p\geq 2\delta$. Choosing $s=(2-\alpha)/4, $  and $p$ large enough the two inequalities hold for any $\delta<(2-\alpha)/4$.

Let us analyze the nonlinear term $v=\partial_{x_N}(f(u))=f'(u)\partial_{x_N}u$. The regularity of $u$ and the space interpolation give us that  for any $s\in (0,1)$,
$$
u\in W^{s,p}((T',T);H^{2(1-s),p}(\rr^N))\subset C^{s-1/p}((T',T);C^{2(1-s)-N/p}(\rr^N)).
$$
Choosing $s=1/2$,  we obtain that $u\in C^{\delta,2\delta}((T',T)\times \rr^N)$ for all $\delta<1/2$. Using that $f'\in C^\beta(\rr)$ this implies that
 $$
 f'(u)\in C^{\beta\delta ,2\beta\delta}((T',T)\times \rr^N)\quad\text{for all }\delta<1/2.
 $$
 Regularity of $u$ yields
 \begin{align*}
 \partial_{x_N}u&\in L^p ((T',T);H^{1,p}(\rr^N)\cap W^{1,p}((T',T);H^{-1,p}(\rr^N))\\
 &\subset W^{s,p} ((T',T);H^{-s+(1-s),p}(\rr^N)\subset C^{s-1/p}((T',T);C^{1-2s-N/p}(\rr^N)).
\end{align*}
 Choosing $s=1/4$ we obtain that  $\partial_{x_N}u \in
 C^{\delta ,2\delta} ((T',T)\times \rr^N)$ for all $\delta<1/4$. This finally implies that $\partial_{x_N}(f(u))=f'(u)\partial_{x_N}u\in C^{\beta\delta,2\beta\delta}((T',T)\times \rr^N)$ for all $\delta<1/4$.
\end{proof}

\begin{remark}When $f(u)=u^q$ with $q<1$, we can only perform the first step in the above proof. This will give us that $u\in C(Q)$. Indeed, $f(u)\in L^p((T',T);L^p)$ for all $p\geq 1/q$. Hence,
$$
u\in L^p((T',T); H^{2-\max\{\alpha,1\},p}(\rr^N)),\qquad u_t\in   L^p((T',T); H^{-\max\{\alpha,1\},p}(\rr^N)),
$$
i.e., $u \in W^{1,p}((T',T); H^{-\max\{\alpha,1\},p}(\rr^N)) )$. If we repeat now the interpolation argument above, we get $u \in W^{s,p}((T',T); H^{(1-s)(2-\max\{\alpha,1\})-s\max\{\alpha,1\} ,p}(\rr^N)) )$ for all $s\in [0,1]$ and $p>1/q$. Choosing $s- 1/p>0$ and $(1-s)(2-\max\{\alpha,1\})-s\max\{\alpha,1\} -N/p>0$, i.e., $2-\max\{\alpha,1\}-2s-N/p>0$,  we get $u\in C_{\rm b}([T',T]\times \rr^N)$.
\end{remark}

%%%%%%%%%%%%%%%%%%%%%%%%%%%%%%%%%%%%%%%%%%%%%%%%%%%%%%%%%%%%%%%%%%%%%%%%%%%%%%%%%%
\section{Compact embedding of mixed spaces}
\label{sec:CompactEmbeddingOfMixedSpaces}
\setcounter{equation}{0}

 We gather here two compactness results that are used in the proof of our convergence result.

\begin{theorem}[Aubin-Lions-Simon, {{\cite[Theorem 5]{Simon}}}]\label{Aubin-Lions-Simon}
Consider the three Banach spaces
$$
X\hookrightarrow B\hookrightarrow Y,
$$
where $X\hookrightarrow B$ is compact. Assume $1\leq p\leq \infty$ and:
\begin{enumerate}[{\rm (i)}]
\item $\mathcal{F}\subset L^p((0,T);X)$ is bounded.
\item $\textstyle{\int_0^{T-h}\|f(t+h)-f(t)\|_Y\,{\rm d}t\to0}$ as $h\to0^+$ uniformly for $f\in \mathcal{F}$.
\end{enumerate}
Then $\mathcal{F}$ is relatively compact in $L^p((0,T);B)$ (and in $C([0,T];B)$ if $p=\infty$).
\end{theorem}

Assume that $\alpha\in(0,2)$ and $x=(x',x_N)\in \rr^{N-1}\times \rr$, and define the Hilbert space
$$
\begin{aligned}
L^2(\rr; W^{\alpha/2,2}(\rr^{N-1}))&:=\{\phi\in L^2(\rr^N) \,:\, |\phi|_{L^2(\rr; W^{\alpha/2,2}(\rr^{N-1}))}<\infty\},
\quad\text{where }\\
|\phi|_{L^2(\rr; W^{\alpha/2,2}(\rr^{N-1}))}&:=\bigg(\int_{\rr^N}\int_{\rr^{N-1}}\frac{|\phi(x'+z',x_N)-\phi(x',x_N)|^2}{|z'|^{(N-1)+\alpha}}\,{\rm d}z'{\rm d}(x',x_N)\bigg)^{\frac{1}{2}},
\end{aligned}
$$
and moreover, $\|\phi\|_{L^2(\rr; W^{\alpha/2,2}(\rr^{N-1}))}=\|\phi\|_{L^2(\rr^N)}+|\phi|_{L^2(\rr; W^{\alpha/2,2}(\rr^{N-1}))}$.
It is classical that the above spaces are the same as the Bessel potential spaces, i.e.  $W^{\alpha/2,2}(\rr^{N-1})=H^{\alpha/2 }(\rr^{N-1})$.

\begin{lemma}\label{lem:CompactEmbeddingInL2}
Assume $\alpha\in(0,2)$,  $\beta>0$, $R', R>0$   and a set $\mathcal{F}$ such that:
\begin{enumerate}[{\rm (i)}]
\item $\|\phi\|_{L^2(\rr; W^{\alpha/2,2}(\rr^{N-1}))}+\|\phi\|_{L^\infty(\rr^N)}\leq C$ for some $C$ independent of $\phi\in \mathcal{F}$.
\item For   all $|h_N|\leq R$ and some $\beta>0$,
$$
\int_{|x'|\leq R', |x_N|\leq R}|\phi(x',x_N+h_N)-\phi(x',x_N)|^2\,{\rm d}x\leq C|h_N|^\beta
$$
for some constant C independent of $\phi\in \mathcal{F}$.
\end{enumerate}
Then  $\mathcal{F}$ is relatively compact in  $L^2(B_{R'}\times(-R,R))$.
\end{lemma}

\begin{remark}\label{X.compact}
For any $R,R'>0$ the space $X_{R,R'}\subset L^\infty(\rr^N) \cap L^2(\rr; W^{\alpha/2,2}(\rr^{N-1}))$ endowed with the norm
\[
\|\psi\|_{X_{R,R'}}=\|\phi\|_{L^2(\rr; W^{\alpha/2,2}(\rr^{N-1}))}+\|\phi\|_{L^\infty(\rr^N)}+\sup_{|h_N|<R}\int_{|x'|\leq R', |x_N|\leq R}
\frac{|\phi(x',x_N+h_N)-\phi(x',x_N)|}{|h_N|}\,{\rm d}x
\]
is compactly embedded in $ L^2(B_{R'}\times(-R,R))$.
\end{remark}

\begin{proof}
We will employ the Fr\'echet-Kolmogorov-Riesz compactness theorem (see,  e.g., ~\cite{H-OHoMa19}); hence, we need uniform estimates on translations in $L^2$.

By the triangle inequality,
\begin{equation*}
\begin{split}
|\phi(x+h)-\phi(x)|^2&=|\phi(x+(h',h_N))-\phi(x)|^2\\
&\leq |\phi(x+(h',h_N))-\phi(x+(0,h_N))|^2+|\phi(x+(0,h_N))-\phi(x)|^2.
\end{split}
\end{equation*}
Integrating the above inequality and using the change of variables $x+(0,h_N)\mapsto x$ gives
\begin{equation}\label{eq:TranslationsInx'Andx_N}
\begin{gathered}
\int_{|x'|\leq R', |x_N|\leq R}|\phi(x+h)-\phi(x)|^2\,{\rm d}x\\
\leq \int_{\rr^N}|\phi(x+(h',0))-\phi(x)|^2\,{\rm d}x+\int_{|x'|\leq R', |x_N|\leq R}|\phi(x+(0,h_N))-\phi(x)|^2\,{\rm d}x.
\end{gathered}
\end{equation}
The second term can already be estimated by assumption (ii), while the first term needs further consideration.
Using Plancherel's identity and the inequality
\[
|\textup{e}^{\textup{i}\xi' \cdot h'}-1|\leq \min\{2, |\xi'\cdot h'|\}\leq C_\alpha |\xi'\cdot h'|^{\alpha/2}\leq C_\alpha|\xi'|^{\alpha/2} |h'|^{\alpha/2},
\]

we get
\begin{align*}
\int_{\rr^N}|\phi(x+(h',0))-\phi(x)|^2\,{\rm d}x&= \int_{\rr^N}|\textup{e}^{\textup{i}\xi' \cdot h' }-1|^2|\widehat\phi(\xi)|^2\,{\rm d}\xi\leq
C_\alpha |h'|^{\alpha }\int_{\rr^N} |\xi'|^{\alpha } |\widehat\phi(\xi)|^2\,{\rm d}\xi\\
&=C_\alpha |h'|^{\alpha} |\phi|_{L^2(\rr; W^{\alpha/2,2}(\rr^{N-1}))}^2.
\end{align*}
Going back to~\eqref{eq:TranslationsInx'Andx_N}, we therefore get
\begin{equation}\label{eq:TranslationsInx'Andx_N2}
\int_{|x'|\leq R', |x_N|\leq R}|\phi(x'+h',x_N+h_N)-\phi(x',x_N)|^2\,{\rm d}(x',x_N)\leq C\big(|h'|^{\alpha}+|h_N|^{\beta}\big),
\end{equation}
where $C$ is again independent of $\phi$ by assumption (ii).

Now, define
$$
\psi(x):=\phi(x)\mathds{1}_{B_{R'}\times(-R,R)}(x). 
$$
In order to apply Fr\'echet-Kolmogorov-Riesz's compactness theorem, we need to check that the family $\{\psi \,:\, \psi\in L^2(\rr^N)\}\subset L^2(\rr^N)$ is equicontinuous and equitight. The latter is automatically satisfied in $\rr^N\setminus (B_{R'}\times(-R,R))$, and we are left with estimating the translations, for $h\in \rr^{N-1}\times(-R,R)$:
\begin{equation*}
\begin{split}
\big\|\psi(\cdot+h)-\psi\big\|_{L^2(\rr^N)}&\leq \big\|\phi(\cdot+h)\big(\mathds{1}_{B_{R'}\times(-R,R)}(\cdot+h)-\mathds{1}_{B_{R'}\times(-R,R)}\big)\big\|_{L^2(\rr^N)}\\
&\quad+\big\|\mathds{1}_{B_{R'}\times(-R,R)}\big(\phi(\cdot+h)-\phi\big)\big\|_{L^2(\rr^N)}\\
&\leq \big\|\phi\big\|_{L^\infty(\rr^N)}\big\|\mathds{1}_{B_{R'}\times(-R,R)}(\cdot+h)-\mathds{1}_{B_{R'}\times(-R,R)}\big\|_{L^2(\rr^N)}\\
&\quad+\big\|\phi(\cdot+h)-\phi\big\|_{L^2(B_{R'}\times (-R,R))}. 
\end{split}
\end{equation*}
By assumption (i), $\|\phi\|_{L^\infty(\rr^N)}$ is uniformly bounded, and hence, the first term goes to zero as $h\to0^+$ uniformly. Moreover, the second term does so too since~\eqref{eq:TranslationsInx'Andx_N2} holds. Hence, there is a convergent subsequence in $L^2(B_{R'}\times (-R,R))$.
\end{proof}

%%%%%%%%%%%%%%%%%%%%%%%%%%%%%%%%%%%%%%%%%%%%%%%%%%%%%%%%%%%%%
\section{An auxiliary lemma}
\setcounter{equation}{0}

This appendix is devoted to extend to the case of integrable and locally bounded functions a result that was already proved in \cite[Lemma~1.2]{EVZArma} for smooth functions and in \cite[Lemma~A.1]{LaurencotSimondon}, \cite[Lemma~2.7]{LaurencotFast} for continuous and integrable ones. This extension is used in the course of the proof of Lemma~\ref{estimari.hiperbolice} to obtain the hyperbolic estimates.
\begin{lemma}
\label{ineq.dis.appendix}
Consider a nonnegative function $f\in L^1(\rr)\cap L^\infty_{\rm loc}(\rr)$ satisfying
\begin{equation}
\label{ineq.distributions}
(f^k)'\leq C\quad \text{or} \quad (f^k)'\geq -C\quad\text{in } \mathcal{D}'(\rr)
\end{equation}
for some real numbers $k>0$ and $C>0$. Then
\begin{equation*}
\label{ineq.L1.Linfty}
0\leq f(x)\leq \Big(\frac{C(k+1)}k \|f\|_{L^1(\rr)}\Big)^{1/(k+1)}  \ a.e.\ x\in \rr.
\end{equation*}
\end{lemma}
\begin{proof} It is enough to prove the first case; the second one can be reduced to it by taking $\tilde f(x)=f(-x)$.

Let us first prove that a function $g\in L_{\rm loc}^1(\rr)$ such that $g'\leq C$ in $\mathcal{D}'(\rr)$ satisfies
\begin{equation}
\label{ineg.g}
  	g(y)-g(x)\leq C(y-x) \quad\text{a.e. }x<y\in \rr.
\end{equation}
By density, in the inequality
\[
	-\int_{\rr } g \varphi'\leq C\int_{\rr} \varphi \quad \text{for all } \varphi\in C^1_{\rm c}(\rr),\ \varphi\ge0,
\]
we can use nonnegative and compactly supported test functions $\varphi$ which are only \emph{piecewise} $C^1$. Fix two points $x<y$ and a function $\varphi_\eps$ such that $\varphi_\eps=1$ in $(x+\eps,y-\eps)$, linear on $(x-\eps,x+\eps)\cup (y-\eps,y+\eps)$ and vanishing identically outside the interval $(x-\eps,y+\eps)$. Then
\[
    -\frac1{2\eps}\int_{x-\eps}^{x+\eps}g +\frac 1{2\eps}\int_{y-\eps}^{y+\eps}g\leq C\int _{x-\eps}^{y+\eps} \varphi_\eps=C(y-x).
\]
Since the function $g$ is locally integrable,  Lebesgue's differentiation theorem guarantees that for a.e.~$x$ and $y$ we can let $\eps\rightarrow 0$ in the above inequality to obtain~\eqref{ineg.g}.
	
Let us now go back to function $f$. Since $f\in L^\infty_{\rm loc}(\rr)$, then $f^k\in L^1_{\rm loc}(\rr)$. Hence, we may apply the preliminary result to $g=f^k$ to obtain, using the inequality~\eqref{ineq.distributions},
\[
	f^k(y)-f^k(x)\leq C(y-x)\quad\text{a.e. } x<y\in \rr.
\]
Let us fix a point $y$ for which the above inequality holds. Then
\[
	f^k(x)\geq C(x-\bar y)\quad\text{a.e. } x<y, \quad\text{where }\bar y:=y-\frac{f^k(y)}C.
\]
We use this inequality on the interval $(\bar y,y)$ to get, since $f$ is nonnegative,
\begin{align*}
\label{}
\|f\|_{L^1(\mathbb{R})}&=  \int_{\rr}f\geq \int_{\bar y}^{y}f\geq C^{1/k} \int_{\bar y}^{y}(x-\bar y)^{1/k}\,{\rm d}x
=C^{1/k}\frac{k}{k+1}(y-\bar y)^{1+1/k}=\frac{k}{C(k+1)} f^{k+1}(y),
\end{align*}
from which  the result follows immediately.
\end{proof}

%
%\section{Conflict of interest}
%On behalf of all authors, the corresponding author states that there is no conflict of interest.

%%%%%%%%%%%%%%%%%%%%%%%%%%%%%%%%%%%%%%%%%%%%%%%%%%%%%%%%%%%%%%%%%%%%%%%%%%%%%%%%%%
%References

%\bibliographystyle{abbrv}
%\bibliography{biblio}

\end{document}